\documentclass[a4paper,12pt]{amsart}
\usepackage{amsmath}
\usepackage{amssymb}
\usepackage{mathrsfs}
\usepackage{color}
\usepackage{epsfig}
\usepackage{graphicx}
\usepackage{enumerate}

\theoremstyle{plain}
\newtheorem{theorem}{Theorem}[section]
\newtheorem{proposition}[theorem]{Proposition}
\newtheorem{corollary}[theorem]{Corollary}
\newtheorem{lemma}[theorem]{Lemma}

\theoremstyle{definition}
\newtheorem{remark}[theorem]{Remark}

\newtheorem{definition}[theorem]{Definition}
\newtheorem{example}[theorem]{Example}

\newtheorem{assumption}[theorem]{Assumption}

\newcommand{\C}{\mathbb{C}}
\newcommand{\R}{\mathbb{R}}
\newcommand{\Z}{\mathbb{Z}}
\newcommand{\N}{\mathbb{N}}

\newcommand{\CP}{\mathbb{C}\mathrm{P}}

\newcommand{\id}{\mathop{\mathrm{id}}\nolimits}

\newcommand{\CO}{\mathcal O}

\renewcommand{\tilde}{\widetilde}
\renewcommand{\setminus}{\smallsetminus}
\newcommand{\nin}{/\kern-2.1ex\in}
\newcommand{\abs}[1]{\lvert#1\rvert}
\newcommand{\norm}[1]{\lVert#1\rVert}

\def\pr{\operatorname{pr}}
\def\<{\left\langle}
\def\>{\right\rangle}
\def\End{\operatorname{End}}

\def\ind{\operatorname{ind}}
\def\Aut{\operatorname{Aut}}
\def\supp{\operatorname{supp}}

\numberwithin{equation}{section}

\title[Torus fibrations and localization of index II]{Torus fibrations and localization of index II \\ 
- Local index for acyclic compatible system - }

\author[H. Fujita]{Hajime Fujita}
\author[M. Furuta]{Mikio Furuta}
 \author[T. Yoshida]{Takahiko Yoshida}

\subjclass[2000]{} 

\address{Department of Mathematical and Physical Sciences, Japan Womens's University, 
2-8-1 Mejirodai, Bunkyo-ku, Tokyo, 112-8681, Japan}
\email{fujitah@fc.jwu.ac.jp}

\address{Graduate School of Mathematical Sciences, The University of Tokyo, 
3-8-1 Komaba, Meguro-ku, Tokyo, 153-8914, Japan}
\email{furuta@ms.u-tokyo.ac.jp}

\address{Department of Mathematics, Graduate School of Science and Technology, Meiji University, 1-1-1 Higashimita, Tama-ku, Kawasaki, 214-8571, Japan}
\email{takahiko@meiji.ac.jp}

\begin{document}

\maketitle

\begin{abstract}
We give a framework of localization for 
the index of a Dirac-type operator 
on an open manifold. 
Suppose the open manifold has a 
compact subset whose complement is covered by a family of finitely many open subsets, each of which
has a structure of the total space of a torus bundle.
Under an acyclic condition 
we define the index of the Dirac-type operator by using the Witten-type deformation, and show that the index has several properties, such as excision property and a product formula. In particular, we show that the index is localized on the compact set.
\end{abstract}


\section{Introduction}

This paper is the second of the series concerning a localization of the 
index of elliptic operators.

For a linear elliptic operator on a closed manifold,
its Fredholm index is sometimes determined by the information on a specific subset
under appropriate geometric conditions.
Such a phenomenon is called {\it localization of index}.
A typical example is Hopf's theorem identifying 
the index of the de Rham operator with the number of zeros of a vector field
counted with sign and multiplicity.
In this case the geometric condition is given by the vector field,
and the index is localized around the zeros of the vector field.
Another typical example is Atiyah-Segal's Lefschetz formula
for the equivariant index for a torus action,
when the geometric condition is given by the torus action.
The index is localized around its fixed point set, and
the localization is understood in terms of an algebraic localization of
equivariant K-group.
In particular when the manifold is symplectic and the elliptic operator 
is a Dirac-type operator, the localization is extensively investigated
using the relation between the algebraic localizations in 
equivariant K-theory and that in equivariant ordinary cohomology theory.

In the previous paper \cite{Fujita-Furuta-Yoshida},
we investigated the case of a closed symplectic manifold equipped with 
a prequantizing line bundle and the structure of a Lagrangian fibration,
and described a localization of the index of a Dirac-type operator, twisted by the
prequantizing line bundle,
on the subset consisting of Bohr-Sommerfeld fibers and
singular fibers.
A novel feature of our method is
that we do not use a global group action but use only the structure of a 
torus bundle on an open subset of the manifold.

In the present paper we generalize our method
to deal with the case when we do not have a global torus bundle
on the open subset, but we just have the structure of a torus bundle 
on a neighborhood of each points, which gives a family
of torus bundles satisfying some compatibility condition.
The various torus bundles may have tori of various dimensions as their fibers.
This generalization enables us to describe the localization phenomenon
more precisely. Even for the case in the previous paper,
we could replace the subset on which the index is localized
with a smaller subset.
A typical example of our generalization is 
the localization for the Riemann-Roch number of toric manifolds, 
for which we would need an orbifold version of our formulation. 
Moreover we can deal with some prequantized singular Lagrangian fibration 
without global toric action (Section~\ref{four-dim ex}). 
In our subsequent paper we will use the localization to give 
an approach to V. Guillemin and S. Sternberg's conjecture 
concerning \lq\lq quantization commutes with reduction" 
in the case of torus actions. 
Though our motivating example is the Riemann-Roch number of a symplectic manifold, 
the localization of the index is formulated for more general cases.
In fact we first establish a general framework to formulate the index of 
elliptic operators on complete manifolds (Section~\ref{An index theory for complete Riemannian manifolds}).
This section is independent of the other sections and
the framework might itself be interesting.

The mechanism of our localization is
explained as a version of Witten's deformation,
where the potential term itself is a first order differential operator.
Our geometric input data is a family of torus bundles.
Roughly speaking we deform the operator like
an {\it adiabatic limit} shrinking the various fiber directions
at the same time in a compatible manner.
The {\it potential term} corresponds 
to some average of the de Rham operators 
along the various fiber directions.

Formally our localization is formulated as
a property for the {\it index} of the elliptic operator on an open manifold:
let $D$ be an elliptic operator on a (possibly non-compact) manifold $X$,
and $V$ is an open subset of $X$ whose
complement  $X \setminus V$ is compact. 
Suppose $V$ has a certain geometric structure $s$,
by which we can modify $D$ to construct a Fredholm operator.
The index of the Fredholm operator depends on the data $(X,V,s,D)$.
Suppose the index satisfies the following properties.
Firstly the index is deformation invariant.
Secondly if $X'$ is an open subset of $X$ containing $X \setminus V$,
and hence $X' \setminus V$ is compact. 
Let $D'$ be the restriction of $D$ on $X'$.
We assume that
the structure $s$ has its {\it restriction $s'$}
on $V'=X'\cap V$.
Then we have the index of the Fredholm operator
constructed from the data
$(X',V',s',D')$.
The required excision property is 
the equality between the two indices.
We will construct Fredholm operators
which satisfy the above type of  {excision property}.
The structure $s$ on $V$ is not extended on 
the whole $X$. In this sense $X \setminus V$ is
regarded as the {\it singular locus} of the structure. 
The index is {\it localized} on the singular locus $X \setminus V$,
and we call it the {\it local index} of the data $(X,V,s,D)$.
When  $X \setminus V$ is of the form of
the disjoint union of
finitely many compact subsets,
the {\it localized} index is equal to
the sum of the contributions from the compact subsets.

Our first main result is the construction of the local index
when  the structure $s$ is the {\it acyclic compatible system} defined in Section~\ref{Compatible fibration and acyclic compatible system}.
Our second main result is a few basic properties of the local index,
in particular a {\it product formula} of the local index in 
Section~\ref{product formula of local indices}.

The organization of this paper is as follows. 
In Section~\ref{Summary} we give brief explanation of 
the motivation and the main results. 
In Section~\ref{An index theory for complete Riemannian manifolds}
 we give a formulation of the index of 
elliptic operators on complete Riemannian manifolds.
This formulation is a generalization of the one given in
Section 5 of \cite{index1}.
%
In Section~\ref{new} we give an abstract geometric framework to define an 
analytic index. 
We consider a collection of data 
which consists of an open covering of the manifold and 
operators defined on the open subsets, 
which satisfies several inequalities. 
By using a perturbation of an elliptic operator by the local operators,   
we show that an analytic index is well-defined for the given data 
in the framework of Section~\ref{An index theory for complete Riemannian manifolds}. 
In Section~\ref{sufficient condition} 
we explain that a Hamiltonian torus action on 
a prequantized symplectic manifold naturally gives rise to the 
data as in Section~\ref{new}. 
In Section~\ref{Compatible fibration and acyclic compatible system} 
we define the notion of an {acyclic compatible system},
which is a generalization of the geometric structure 
extracted from the Hamiltonian torus action case as 
in Section~\ref{sufficient condition}. 
It can be used as the geometric structure $s$ in the above explanation.
In Section~\ref{Main theorem} we define the {\it local index} 
under the assumption that
an acyclic compatible system is given on an end of the manifold.
In Section~\ref{product formula of local indices} we show a product formula for
the index defined in Section~\ref{Main theorem}.
In Section~\ref{four-dim ex} we give an example of our formulation
using some $4$-dimensional Lagrangian fibration with singular fibers.
In Appendix we give proofs of several theorems and lemmas
used in the main part.

\section{Summary}\label{Summary}

Main results in this paper are the following three. 
\begin{enumerate}
\item We introduce a family of operators which satisfy 
specific estimates. 
We establish an index theory on open manifolds with boundary condition described in terms of the family of operators.
\item We show that the index has several properties, such as, deformation invariance, an excision formula, and a gluing formula. 
\item We show a product formula in the index theory. 
\end{enumerate}


In \cite{Fujita-Furuta-Yoshida} we established an index theory 
under the following setting. 
Let $M$ be a Riemannian manifold and $V$ an open subset of $M$ 
whose complement is compact. 
Suppose that there exists a Riemannian submersion $\pi:V\to U$ 
and a Dirac type operator $D_{\pi}$ along the fibers of $\pi$. 
We assume that the kernel of $D_{\pi}$ is trivial along each orbit and 
$D_{\pi}$ anti-commutes with the Clifford multiplication of the 
lift of a tangent vector of $U$. 
Let $D$ be a Dirac type operator on $M$. 
By using a perturbation $D_t:=D+tD_{\pi}$ for $t\gg 1$ we can construct a 
Fredholm operator, 
where $D_{\pi}$ is extended as a zero operator on $M\setminus V$ by 
a cut-off function. The perturbation can be understood as an infinite dimensional version of Witten's deformation. 
The Fredholm index has deformation invariance, an excision formula, and a gluing formula. 
In particular, as a corollary of the excision formula, 
we have a localization formula of the index of Dirac type operator 
over a closed manifold.

%

In this paper we consider a generalization of the above setting. 
Typical example of the setting is as follows. 
Suppose that there exists a finite open covering $\cup_{\alpha\in A} V_{\alpha}$ of 
$V$. 
Moreover we assume that there exist a fiber bundle structure $\pi_{\alpha}:V_{\alpha}\to U_{\alpha}$ for each $\alpha\in A$ and 
a Dirac type operator $D_{\alpha}$ along orbits of $\pi_{\alpha}$ which 
have the same property as for the above $D_{\pi}$ and 
some compatibility conditions and estimates over $V_{\alpha}\cap V_{\beta}$. 
For this kind of data we develop an index theory 
based on the perturbation $D+t\sum_{\alpha}D_{\alpha}$ for $t\gg 1$, 
and we have the localization formula of index. 
Though the resulting index has the deformation invariance, 
it depends on the discrete data in general, e.g.,  
the number of the open subsets $\{V_{\alpha}\}_{\alpha}$. Such a structure appears naturally from Hamiltonian torus action 
on a prequantized symplectic manifold and the obtained localization formula is more precise than that in \cite{Fujita-Furuta-Yoshida}. 


In Section~\ref{Compatible fibration and acyclic compatible system} 
we formulate such a structure as a {\it compatible fibration} and 
a {\it  compatible system}. 
The non-degenerate condition and the estimates 
of Dirac-type operators along fibers, 
which are defined locally, is related to the 
non-degenerate condition of the weighted average of the operators (Lemma~\ref{strongly -> acyclic}). 
For the Hamiltonian torus action case the condition 
is determined by the existence of global parallel section of 
the prequantizing line bundle over the orbit (Proposition~\ref{Hamiltonian acyclic}).  


One of the main results is 
the definition of the {\it local index} for a 
manifold equipped with an acyclic compatible system on an end of $M$ 
(Definition~\ref{def. of local the index}). The local index has deformation invariance, an excision formula, and a gluing formula.


Another main result is the {\it product formula} 
for the local indices (Theorem~\ref{product}). 
Once we have an appropriate definition of the product of 
compatible fibrations and compatible systems, 
the product formula can be derived from the general setting in 
Section~\ref{An index theory for complete Riemannian manifolds}. 
As an application of the product formula 
we give a computation for 4-dimensional case 
in Section~\ref{four-dim ex}. 


In \cite{Braverman}\cite{Ma-Zhang}\cite{Paradan,Paradan2}, 
the authors developed an index theory of transverse elliptic operators 
using the perturbation of the Dirac-type operator along the orbit of 
the group action. 
This index theory has conceptual similarity with our index theory 
in the torus equivariant setting. 
One can settle the setting 
to compare these two equivariant index theories. 
The invariant part of these indices coincide, and 
for Hamiltonian circle action whose momentum map is proper and 
have no critical points, the two indices coincides with each other. But there is an example in which these do not coincide with each other.  
See \cite{Hajime} for details. 

\section{An index theory for complete Riemannian manifolds}
\label{An index theory for complete Riemannian manifolds}
In this section we give an index theory for complete Riemannian manifolds. Proofs are given in Appendix~\ref{Appendix B}. 
\subsection{Formulation of index on complete manifolds}
\label{formulation of index}

Suppose $M$ is a complete Riemannian manifold,
$W$ is a ${\mathbb Z}/2$-graded Hermitian vector bundle,
and $\sigma:TM \to \End (W)$ is a homomorphism such that
$\sigma(v)$ is a skew-Hermitian isomorphism of degree-one
for each $v\in TM\setminus \{0\}$.
Let $D$ be a degree-one formally self-adjoint first-order 
elliptic differential operator on
$W$ with principal symbol $\sigma$.
We assume that $\sigma$ and
the coefficients of $D$ are smooth.
We formulate an index theory on $M$ under the following assumption.
\begin{assumption}\label{assumption for operator}
\begin{itemize}
\item
$D$ has finite propagation speed:
there exists a positive real number $C_0$ satisfying
$|\sigma| \leq C_0$ uniformly on $M$,
\item
There exist a positive real number $\lambda_0>0$ and 
an open subset $V$ of $M$ with
its complement $M \setminus V$ compact such that
$$
\lambda_0 ||s||^2_{V}  \leq  ||Ds||^2_{V}  
$$
for any smooth compactly-supported section $s$ of $W$
with support contained in $V$.
\end{itemize}
\end{assumption}
It is known that the finite propagation speed
implies that $D$ is essentially self-adjoint \cite{Chernoff} and we have
the following property. 
\begin{theorem} \label{spectrum}
$D$ is essentially self-adjoint
as an operator on $L^2$-sections of $W$ and its
spectrum is discrete in $(-\sqrt{\lambda_0},\sqrt{\lambda_0})$. 
\end{theorem}
The first part of Theorem~\ref{spectrum} is well-known, for example, see~\cite{Braverman-Milatovich-Shubin,Lesch}.  
The rest of the statement follows from Proposition~\ref{single operator}.

\begin{definition}
$E_{\lambda}$ is the vector space of smooth sections $s$ of $W$
such that $s$ is $L^2$-bounded and satisfies
$D^2 s =\lambda s$.
\end{definition}
Theorem~\ref{spectrum} implies that
$E_\lambda$ is zero for $\lambda<0$, and
$E_\lambda$ is finite dimensional for $\lambda <\lambda_0$.
Moreover there are only discrete values $\lambda< \lambda_0$
for which $E_\lambda$ is non-zero.
Note that
the super dimension of $E_\lambda$ is zero
for $0< \lambda <\lambda_0$, and hence
the super dimension of
$
\oplus_{\lambda<\lambda_1} E_\lambda
$
is constant for $0< \lambda_1<\lambda_0$.
\begin{definition}
$\ind D$ is the super dimension of $E_0$, or
the super dimension of $
\oplus_{\lambda<\lambda_1} E_\lambda
$ for  $0< \lambda_1 <\lambda_0$.
\end{definition}

The index has the following deformation invariance.
Let $\{D_t\}$ $(|t|<\epsilon)$ be a one-parameter 
family of degree-one formally self-adjoint first-order elliptic
differential operators
on $W$ with principal symbols $\{\sigma_t\}$.
\begin{assumption} \label{global continuity}
\begin{itemize}
\item Each $D_t$ and
$\sigma_t$ satisfy
Assumption~\ref{assumption for operator}
for common $\lambda_0$ and $V$.
\item
On each compact subset of $M$
the coefficients of $D_t$ are $C^\infty$ convergent
to those of $D_0$ as $t\to 0$.
\end{itemize}
\end{assumption}
We do not assume that the propagation speed is
uniform with respect to $t$.


Given the essentially self-adjointness as above a standard argument implies:
\begin{theorem}  \label{deformation invariance}
Under Assumption~\ref{global continuity} 
$\ind D_t$ is constant with respect to $t$.
\end{theorem}
A proof is given in the end of Appendix~\ref{proof of index theory}. 
\begin{remark}{\rm
So far we are fixing $M$ and $W$.
In Section~\ref{gluing} we will formulate
a deformation for which $M$ and $W$ can vary.
We give a proof of Theorem~\ref{deformation invariance}
so that it can be directly generalized to this case.
The generalization immediately
implies an excision property of index for complete Riemannian
manifolds.
}
\end{remark}

\subsection{Generalized deformation invariance and gluing formula} \label{gluing}

In this subsection we generalize Theorem~\ref{deformation invariance}.
We first need to generalize Assumption~\ref{global continuity}.

Let $(M, W, \sigma, D,V)$ be 
as in Section~\ref{formulation of index}.
Let $\{M_i\}_{i\in {\mathbf N}}$ be a family of 
complete Riemannian manifolds. 
We further 
suppose that there exists a $\Z/2$-graded Clifford module bundle $W_i$ 
over each $M_i$ and an open subset $V_i$ 
such that $M_i\setminus V_i$ is compact. 
Suppose that there exists the following data. 
\begin{itemize}
\item sequence of compact subsets $\{K_i\}$ of $M$ 
\item  
isometric open embeddings $\{\iota_i: int(K_i) \to M_i\}$ 
\item  
isomorphisms $\{\tilde\iota_i:W|_{int(K_i)} \cong\iota_i^* W_i\}$
as ${\mathbb Z}/2$-graded Hermitian vector bundle over $int(K_i)$. 
\item differential operators $\{\tilde D_i\}$, each of which is defined on a neighborhood of $K_i$ 
\end{itemize}




\begin{assumption}\label{assumption gluing}
{\rm
\begin{enumerate}
\item $K_1 \cup V=M$,$K_i \subset int(K_{i+1})$ and $M=\cup_i K_i$. 
\item $\iota_i(M \setminus V)=M_i \setminus V_i$.
\item The coefficients of $\tilde D_i$ are $C^\infty$-convergent to those of $D$ on each compact subset as $i \to \infty$.
\item The differential operator $\iota_i\tilde D_i$ on $\iota_i(int(K_i))$ has an extension $D_i$ to $M_i$. 
\item
The data $(M_i, W_i, \sigma(D_i), D_i,V_i)$ satisfy
Assumption~\ref{assumption for operator}
with the same constant $\lambda_0$ for
$(M, W, \sigma, D,V)$. 
\end{enumerate}
}
\end{assumption}
We do not assume that the propagation speed of $\{D_i\}$
is uniformly bounded with respect to $i$.

\begin{theorem} \label{generalized deformation invariance}
Under Assumption~\ref{assumption gluing},
$\ind\, D_i$ is equal to $\ind\, D$ for large $i$. 
\end{theorem}

A proof of Theorem~\ref{generalized deformation invariance} 
is explained in Appendix~\ref{proof of Theorem 3.9}.
In particular we have the following formula, which 
is a generalization of the gluing formula 
of indices of manifolds with cylindrical end. 


\begin{proposition}\label{gluing prop}
Let $(M, W, \sigma, D, V)$ be a data 
satisfying Assumption~\ref{assumption for operator}. 
Let $\{(M_i, W_i, \sigma(D_i), D_i, V_i)\}_{i}$ be a data satisfying Assumption~\ref{assumption gluing}. 
Suppose that each $M_i$ is connected and $M$ is the disjoint union of $M'$ and $M''$.
Let $(W',\sigma',D',V')$ and $(W'',\sigma'',D'',V'')$ be
the restrictions of $(W,\sigma,D,V)$ to
$M'$ and $M''$ respectively.
Then we have 
$$
\ind\,D_i=\ind\,D'+\ind\,D'', 
$$
for sufficiently large $i$. 
\end{proposition}

\begin{proof}
Since $\ind\,D=\ind\, D' +\ind\, D''$,
Theorem~\ref{generalized deformation invariance} implies 
the required gluing formula. \
\end{proof}

When $M'$ and $M''$ have  cylindrical end components and 
all $M_i$ are diffeomorphic manifold which are obtained by 
gluing the cylindrical end components of $M'$ and $M''$, 
Proposition~\ref{gluing prop} can be regarded as the 
usual gluing formula.

On the other hand essentially self-adjointness and the 
inequality in Assumption~\ref{assumption for operator} implies the following vanishing lemma. 
\begin{lemma}\label{vanishing lemma}
Let $(M, W, \sigma, D, V)$ be a data 
satisfying Assumption~\ref{assumption for operator}. 
If $M=V$, then we have $\ker D\cap L^2(W)=0$. 
\end{lemma}

Using Proposition~\ref{gluing prop} and 
the above Lemma~\ref{vanishing lemma}, 
we have the following excision formula of index. 

\begin{proposition}
Let $(M, W, \sigma, D, V)$ be a data 
satisfying Assumption~\ref{assumption for operator}. 
Let $\{(M_i, W_i, \sigma(D_i), D_i, V_i)\}_{i}$ be a data satisfying Assumption~\ref{assumption gluing}. 
Suppose $M$ is the disjoint union of $M'$ and $M''$,
and $M''$ is contained in $V$.
Let $(W',\sigma',D',V')$ be
the restriction of $(W,\sigma,D,V)$ to
$M'$.
Then we have 
$$
\ind\,D_i=\ind\,D' 
$$
for sufficiently large $i$.  
\end{proposition}

\subsection{Product formula}

Following Atiyah and Singer \cite{Atiyah-Singer},
we formulate a product formula for elliptic operators.
Except that we need Lemma~\ref{partial integral 1}
for partial integration on complete Riemannian manifolds, 
the argument is exactly the same as in \cite{Atiyah-Singer}.
The main purpose of this subsection is to formulate
Assumption~\ref{assumption product} below, which is crucial
for the case of complete manifolds.
To apply the product formula it is necessary to
check the assumption for specific operators, 
which is our another task  
and is carried out in Section~\ref{product formula of local indices}.

For $k=0,1$ let $M_k$ be a complete Riemannian manifold, 
$W_k$ a ${\mathbb Z}/2$-graded Hermitian vector bundle
over $M_k$, and $D_k:\Gamma(W_k) \to \Gamma(W_k)$ 
a degree-one formally self-adjoint order-one elliptic
operator
with principal symbol $\sigma_k$.


Let $G$ be a compact Lie group and
$P \to M_0$ a principal $G$-bundle.
Suppose $G$ acts on $M_1$ isometrically,
$W_1$ is $G$-equivariant ${\mathbb Z}/2$-graded Hermitian vector bundle,
and $D_k$ is a $G$-invariant operator.

Then $M=P \times_G M_1$ is a fiber bundle over $M_0$ with
fiber $M_1$. We write $\pi:M \to M_0$ for the projection map.
Let $\tilde W_0$ and $\tilde W_1$ be the 
vector bundles over $M$ defined by
$\tilde W_0=\pi^* W_0$ and
$\tilde W_1=P \times_G W_1$,
and we put $W=\tilde W_0 \otimes \tilde W_1$.

We would like to lift $D_0$ and $D_1$ 
as operators on $W$.
The lift of $D_1$ is given straightforward:
Defining the operator 
$
\tilde D_1
$
on
$\Gamma(\tilde W_0 \otimes \tilde W_1)$
by $\epsilon \otimes D_1$ on each fiber of $\pi:M \to M_0$,
where $\epsilon:W_0 \to W_0$ is equal to $+id$ on the
degree $0$ part of $W_0$, and to $-\id$ on the degree $1$ part of $W_1$.

We next construct  
$\tilde D_0:\Gamma(W) \to \Gamma(W)$.
Let $\{V_\alpha\}$ be an open covering  of $M_0$ and
$\{\rho^2_\alpha\}$ a partition of unity.
Suppose we have local trivializations $P|_{V_\alpha}\cong V_\alpha \times G$
with transition functions $g_{\alpha\beta}$.
Using the local trivialization on $V_\alpha$ we have the identifications
$
{\pi^{-1}(V_\alpha)} \cong {V_\alpha} \times M_1$ and
$W |_{\pi^{-1}(V_\alpha)} \cong {W_0}|_{V_\alpha} \times W_1$.
Let $\tilde{D}_{0,\alpha}$ be the operator on
$W |_{\pi^{-1}(V_\alpha)}$ defined by $D_0$ using the product structure.
We put
$
\tilde D_0:=\sum_\alpha \rho_\alpha 
\tilde D_{0,\alpha} \rho_\alpha.
$
\begin{lemma}\label{anti-commute:section3}
$\tilde D_0 \tilde D_1 +\tilde D_1 \tilde D_0 =0$.
\end{lemma}
\begin{proof}
It follows from
$\tilde D_1 \tilde D_{0,\alpha}+\tilde D_{0,\alpha}\tilde D_1=0$
and $\tilde D_1 \rho_\alpha -\rho_\alpha \tilde D_1=0$. \
\end{proof}

Let $R$ be a $G$-invariant ${\mathbb Z}/2$-graded 
finite dimensional subspace
of $\Gamma(M_1,W_1)$,
and $\tilde R$ the fiber bundle $P \times_G R$ over $M_0$.
Then we have an embedding 
\begin{equation} \label{embedding}
\Gamma(M_0, W_0 \otimes \tilde R) \to \Gamma(M, \tilde W)
\end{equation}
which is preserved by the action of $\tilde D_0$.
Let $\tilde D_{R}$ be the restriction of $\tilde D_0$
on $\Gamma(M_0, W_0 \otimes \tilde R)$. 
Then $\tilde D_{R}$ is a differential operator on $W_0 \otimes \tilde R$
with principal symbol $\sigma_0 \otimes \id_{\tilde R}$.
We put $R=E_0(D_1)$ and assume the following conditions. 

\begin{assumption}\label{assumption product}
\begin{enumerate}
\item
$D_0$ has finite propagation speed, i.e., $\sigma_0$ is $L^\infty$-bounded.
\item
The data $(M_1,W_1,D_1)$ satisfies Assumption~\ref{assumption for operator}.
\item
The data $(M,W,\tilde D_0 +\tilde D_1)$ satisfies Assumption~\ref{assumption for operator}.
\end{enumerate}
\end{assumption}
We do not assume the second condition 
of Assumption~\ref{assumption for operator} for 
the data $(M_0,W_0,D_0)$.

Recall that $\tilde D_{R}$ is given by $\tilde D_0$ via
the embedding (\ref{embedding}). Since $\tilde D_1=0$ on the image
of the embedding, 
Assumption~\ref{assumption product} implies that
the data $(M_0, W_0 \otimes \tilde R,\tilde D_{R})$ 
satisfies Assumption~\ref{assumption for operator} as well.
\begin{theorem}[product formula]
$\ind (\tilde D_0+\tilde D_1)= \ind \tilde D_{R}$
\end{theorem}

\begin{proof}
We show that the embedding (\ref{embedding}) gives
the isomorphism $E_0(\tilde{D}_{R}) \cong E_0(\tilde D_0+\tilde D_1)$.
If $s$ is in the image of $E_0(D_{R})$, then the construction of $D_{R}$
implies that $s$ is obviously in  $E_0(\tilde D_0+\tilde D_1)$.


If $s$ is an element of
$E_0(\tilde D_0+\tilde D_1)$ we have
\[
0=\int_M ((\tilde D_0+\tilde D_1)^2 s, s)
=\int_M (\tilde D_0^2 s +\tilde D_1^2 s, s)
=||\tilde D_0 s||^2+||\tilde D_1 s||^2,
\]
i.e.,
$\tilde D_0 s=\tilde D_1 s=0$. 
Note that, since $\tilde D_0$ and $\tilde{D}_1$ are essentially self-adjoint, partial integration formula is available here.
In particular $\tilde D_1 s=0$
implies that $s$ is in the image of the embedding (\ref{embedding}).
Moreover $\tilde D_0 s=0$ implies $s$ is in the image of
$E_0(D_{R})$. \
\end{proof}
\section{Perturbation and local index}\label{new}
In this section, using a perturbation by certain locally defined operators we construct an operator which satisfies Assumption~\ref{assumption for operator} in Section~\ref{An index theory for complete Riemannian manifolds}. 

\subsection{Setting}\label{Setting}
Let $(M,W,D)$ be the data satisfying the following properties.
\begin{enumerate}
\item  $M$ is a Riemannian manifold.
\item  $W$ is a $\Z/2$-graded vector bundle over $M$ with hermitian metric.
\item $D:\Gamma(W) \to \Gamma(W)$ is a formally self-adjoint differential operator of degree odd and of first order with finite propagation speed, i.e.,  
the $L^\infty$-norm of  $|\sigma(D)|$ is bounded.
\end{enumerate}
Suppose we have a collection of data $(V_\infty, \rho_\infty, \epsilon_\infty)$ and  $\{V_\alpha , D_\alpha,  C_\alpha, \delta_\alpha,  \rho_\alpha, \epsilon_\alpha \}_{\alpha \in A}$ with the following properties.
\begin{enumerate}
\item
$A$ is a finite set. We write $\tilde{A}=A \coprod \{\infty\}$
\item
$\{ V_\alpha\}_{\alpha \in \tilde{A}}$ is an open covering of $M$.
We write $V_{\alpha \beta}$ for
$V_\alpha \cap V_\beta$.
\item  $\{\delta_\alpha\}_{\alpha \in A}$ and $\{C_\alpha\}_{\alpha\in A}$ are families of positive real numbers. For $\alpha \in A$,
$$D_\alpha: \Gamma(W|_{V_\alpha}) \to \Gamma(W|_{V_\alpha})$$ is a formally self-adjoint differential  operator of degree odd and of first order with finite propagation speed satisfying the following properties.
(Formally we put $D_\infty=0$.) 
\begin{enumerate}
\item
For $\alpha \in A \,(\alpha \neq \infty)$ we have
$$
\int_{V_\alpha}(D_\alpha^2s,s) \geq \delta_\alpha \int_{V_\alpha}\abs{s}^2
$$
for any smooth compactly supported section $s\in \Gamma(W|_{V_\alpha})$.
\item For $\alpha,\beta \in A$ we have
$$
\int_{V_{\alpha\beta}}(\{D_\alpha, D_\beta\}s,s)  \geq 0
$$ 
for any smooth compactly supported section $s\in \Gamma(W|_{V_{\alpha\beta}})$, where $\{D_\alpha, D_\beta \}=D_\alpha  D_\beta + D_\beta D_\alpha$.
\item For $\alpha \in A$, we have
$$
\int_{V_\alpha}(\{D,  D_\alpha\}s,s) \leq C_\alpha \int_{V_\alpha}(D_\alpha^2s,s)
$$
for any smooth compactly supported section $s\in \Gamma(W|_{V_\alpha})$, where $\{D,  D_\alpha\}=D D_\alpha + D_\alpha D$.
\end{enumerate}
%
\item
$\{\rho_\alpha\}_{\alpha \in \tilde{A}}$ 
 is a family of smooth bounded functions on $M$ 
and $\{\epsilon_\alpha \}_{\alpha \in \tilde{A}}$ is a family of positive real numbers 
 satisfying the following properties.
\begin{enumerate}
\item $\supp \rho_\infty$ is compact.
\item For $\alpha \in \tilde{A}$ 
$$\supp \rho_\alpha \subset V_\alpha$$ 
\item
$$M=\bigcup_{\alpha \in \tilde{A}} \rho_\alpha^{-1}((\epsilon_\alpha,\infty))$$.
\item For $\alpha,\beta \in A$ we have
$$[D_\alpha, \rho_\beta]=0,
$$  
where 
$[D_\alpha, \rho_\beta]=D_\alpha  \rho_\beta - \rho_\beta D_\alpha$.
\item 
For $\alpha \in \tilde{A}$ 
\[
\rho_\alpha \in L^\infty_2,
\]
i.e., $\rho_\alpha, d\rho_\alpha, \nabla d\rho_\alpha \in L^\infty$, where $\nabla$ is the Levi-Civita connection with respect to the Riemannian metric. In particular, 
$$[D, \rho_\alpha],\ [D_\beta ,[D,\rho_\alpha]],\ [D, [D,\rho_\alpha]]\in L^\infty.$$
\end{enumerate}
\item
For each $x\in M$ and $\xi \in T^*_xM\setminus \{0\}$ 
$$\left(\sigma(D)+\sum_{\alpha\in A,\ x\in V_\alpha}t_\alpha\sigma(D_\alpha)\right)(\xi)$$ is isomorphic 
if $t_\alpha\ge 0$. 
In particular, for $t\ge 0$ the operator $D_t:\Gamma(W) \to \Gamma(W)$ defined by
$$
D_t=D+ t \sum_{\alpha \in A}\rho_\alpha D_\alpha \rho_\alpha
$$
is elliptic.
\end{enumerate}
\begin{lemma}\label{admissible p-o-u2}
There exists a partition of unity 
$\{\chi_{\alpha}^2\}_{\alpha\in \tilde A}$ subordinate to $\{V_\alpha\}_{\alpha\in \tilde A}$ such that $[D,\chi_\alpha]$ is bounded and we have
$$
\supp \chi_\alpha \subset \rho_\alpha^{-1}([\epsilon_\alpha ,\infty))
$$
for $\alpha \in \tilde{A}$, 
and 
$$
[D_\alpha,\chi_\beta]=0
$$
for $\alpha,\beta \in A$. 
\end{lemma}
\begin{proof}
We first take and fix a non-decreasing function $\varphi_\alpha:\R_{\ge 0}\to \R_{\ge 0}$ such that $\varphi_\alpha(r)=0$ if $r\in [0,\epsilon_\alpha]$ and $\varphi_\alpha(r)>0$ otherwise. 
We define $\hat\rho_{\alpha}:M\to \R$ by the composition 
$\hat\rho_{\alpha}:=\varphi_\alpha\circ\rho_{\alpha}$. 
Then, by definition,  we have 
${\rm supp}(\hat\rho_{\alpha}) \subset \rho_\alpha^{-1}([\epsilon_\alpha, \infty))$
and 
$$
M=\cup_{\alpha \in \tilde{A}} \hat{\rho}_\alpha^{-1}((0,\infty)).
$$
We obtain  the required  partition of unity 
$\{\chi_{\alpha}^2\}_{\alpha\in\tilde A}$ by putting
$$
\chi_\alpha=\hat{\rho}_\alpha \left(\sum_{\alpha \in \tilde{A}} (\hat{\rho}_\alpha)^2\right)^{-1/2}.
$$
\
\end{proof}

\subsection{An estimate for the perturbed operator}
Let $(M,W,D)$, $(V_\infty, \rho_\infty, \epsilon_\infty)$, and
$
\{V_\alpha , D_\alpha,  C_\alpha, \delta_\alpha, \rho_\alpha, \epsilon_\alpha \}_{\alpha \in A}$ be as above. 
In this subsection we assume that the Riemannian metric on $M$ is complete. 
We put 
$$
V:=\bigcup_{\alpha\in A}V_{\alpha} 
$$and 
$$
D_{\alpha}':=\rho_{\alpha}D_{\alpha}\rho_{\alpha}.
$$ for each $\alpha\in A$. 
For a positive real number $t$, 
we defined the operator acting on $\Gamma(W)$ by 
$$
D_{t}:=D+t\sum_{\alpha\in A} D_{\alpha}'.
$$
One of our assumption is that $D_t$ is elliptic for $t\gg 1$. 
The following is a main theorem in this subsection.

\begin{theorem}\label{vanishing for closed case}
Suppose that the Riemannian metric on $M$ is complete. 
For any $t \gg 1$ we have 
$$
\|s\|^2_{V}  \leq  \|D_ts\|^2_{V}  
$$
for any smooth compactly-supported section $s$ of $W$
with support contained in $V$.
\end{theorem}


Let $\{\chi_{\alpha}^2\}_{\alpha\in \tilde A}$ 
be the partition of unity 
constructed in Lemma~\ref{admissible p-o-u2}
and put $K_{\alpha}:=supp~\chi_{\alpha}$. 
\begin{lemma}\label{lemma-local}
There exists a bounded operator $Z$ which does not contain 
any differential terms and satisfying 
$$
D_t^2=\sum_{\alpha}\chi_{\alpha}D_t^2\chi_{\alpha}+Z. 
$$
Moreover $Z$ does not depend on $t$.  
\end{lemma}

\begin{proof}
Suppose  $\chi$ is a smooth function satisfying 
$[D_{\alpha}, \chi]=0$ for every $\alpha \in A$.  
Then we have $[D_t,\chi]=[D,\chi]$. 
Using this equality and the fact $[D,\chi]$ 
does not contain any differential operators 
we have 
\begin{eqnarray*}
[[D_t^2,\chi],\chi] &=& [ (D_t[D_t,\chi]+[D_t,\chi]D_t),\chi] \\
 &=& [ (D_t[D,\chi]+[D,\chi]D_t),\chi] \\
 &=& [D_t,\chi][D,\chi]+[D,\chi][D_t,\chi] \\
&=& 2[D,\chi]^2. 
\end{eqnarray*}
Put $\chi:=\chi_{\alpha}$ and take summation for all $\alpha$ we have
$$
2D_t^2 - 2\sum_\alpha \chi_\alpha D_t^2 \chi_\alpha
= \sum_\alpha [[D_t^2,\chi],\chi]  
= 2 \sum_\alpha [D,\chi_\alpha]^2. 
$$Then $Z:=\sum_\alpha [D,\chi_\alpha]^2$ 
is the required operator of order $0$. \
\end{proof}

\begin{proposition}\label{KeyProp-local}
If $s_{\alpha}$ is an $L^2$-bounded section of $W$ such that 
$D_t s_{\alpha}$
is also $L^2$-bounded section 
and ${\mathrm supp}~s_{\alpha}\subset K_{\alpha}$, 
then we have the inequality  
$$
\int_M |D_t s_{\alpha}|^2  \geq 
\left|\int_M (Zs_{\alpha},s_{\alpha})\right| +\int_M |s_{\alpha}|^2 
$$for all $t\gg 1$. 
\end{proposition}

We use the above proposition and lemmas 
to show Theorem~\ref{vanishing for closed case} as follows. 

\begin{proof}[Proof of Theorem~\ref{vanishing for closed case} assuming Proposition~\ref{KeyProp-local}]
We take $t\gg 1$ so that the inequality in 
Proposition~\ref{KeyProp-local} holds. 
When $s$ is a smooth compactly-supported section of $W$ 
with support contained in $V$.
We have 
\[
D_ts_\alpha=[D,\chi_\alpha]s,\ D_t^2s_\alpha=[D_t,[D,\chi_\alpha]]s\in L^2,
\]
where $s_\alpha:=\chi_\alpha s$. 
We obtain the required inequality as follows; 
\begin{eqnarray*}
\|D_ts\|_V^2&=&\int_M(D_t^2s,s) \\
&=&\sum_{\alpha}\int_M(\chi_{\alpha}D_t^2\chi_{\alpha}s,s)+
\int_M(Zs,s) \qquad (\text{Lemma~\ref{lemma-local}})\\ 
&=&\sum_{\alpha}\left(\int_M|D_ts_{\alpha}|^2+
\int_M(Zs_{\alpha},s_{\alpha})\right) \\
&\geq&\sum_{\alpha}\left(\left| \int_M(Zs_{\alpha},s_{\alpha}) \right| + 
\int_{M}|s_{\alpha}|^2+\int_M(Zs_{\alpha},s_{\alpha})\right) 
\qquad (\text{Proposition~\ref{KeyProp-local}})\\  
&\geq&\sum_{\alpha}\int_{M}|s_{\alpha}|^2=\int_M|s|^2=\|s\|_V^2.  
\end{eqnarray*}\
\end{proof}

\subsubsection{Proof of Proposition~\ref{KeyProp-local}}
Since $D_t$ has finite propagation speed, $D_t$ is essentially self-adjoint. It implies that it suffices to show the inequality when $s_\alpha$ is smooth and compactly supported, which we assume in the following argument.

For each fixed $\alpha\in A$, we can write $D_t$ as
$$
D_t= D_{\neq\alpha}+tD_{\alpha}'
$$ on $V_{\alpha}$, where we put 
$$
D_{\neq\alpha}:=D+t\sum_{\beta\neq\alpha}D'_{\beta}. $$

Using these notations together with 
$Q_\alpha:=D D'_\alpha +  D'_\alpha D$ and 
$Q_{\beta\alpha}:=D'_\beta D'_\alpha + D'_\alpha D'_\beta$, 
we have 
$$
D_t^2=D_{\neq\alpha}^2 + 
tQ_\alpha+t^2\sum_{\beta\neq\alpha} Q_{\beta\alpha}
+t^2D_{\alpha}'^2, 
$$and since 
$\int_M(D_{\neq\alpha}^2s_{\alpha},s_{\alpha})$ and 
$\int_M(Q_{\beta\alpha}s_{\alpha},s_{\alpha})$ are non-negative, 
we also have 
$$
\int_M |D_t s_{\alpha}|^2 \geq t^2\int_M |D_{\alpha}' s_{\alpha}|^2 \\
-t \left| \int_M (Q_\alpha s_{\alpha},s_{\alpha})\right|. 
$$
Since $\rho_\alpha \ge \epsilon_\alpha$ on $\supp s_\alpha$ and 
$\{D,D_{\alpha}\}\leq C_{\alpha}D_{\alpha}$ we have the estimates
$$
\epsilon_\alpha^4 \int_M |
D_\alpha s_\alpha|^2 \leq \int_M |D_\alpha' s_\alpha |^2, \quad 
\left|\int_M (Q_\alpha s_{\alpha}, s_{\alpha})\right|
 \leq C'_{\alpha}\int_M |D_{\alpha}' s_{\alpha}|^2 . 
$$
for $C'_{\alpha}:=C_{\alpha}/\epsilon_{\alpha}^4$. 
Combining these inequalities we have 
$$
\int_M |D_t s_{\alpha}|^2 
\ge  (t^2-C'_{\alpha}t) \int_M |D_{\alpha}' s_{\alpha}|^2. 
$$
On the other hand 
 ${\mathrm supp}~s_{\alpha}\subset K_{\alpha}$ 
and $\delta_{\alpha}\leq D_{\alpha}^2$ imply  that 
by putting $C''_{\alpha}:=(1+\max Z)/\epsilon_{\alpha}^4\delta_{\alpha}$ 
we have 
$$
 \left|\int_M (Zs_{\alpha},s_{\alpha})\right| +\int_M |s_{\alpha}|^2 
 \leq C''_{\alpha}\int_M |D_{\alpha}' s_{\alpha}|^2,  
$$and hence 
if we take $t\gg 1$ so that 
$t^2-C'_{\alpha}t\ge C''_{\alpha}$, then we have 
$$ 
\int_M |D_t s_{\alpha}|^2 \geq 
\left|\int_M (Zs_{\alpha},s_{\alpha})\right| +\int_M |s_{\alpha}|^2.  
$$
Since $A$ is a finite set we complete the proof. 
\

\subsection{Definition of the local index}
In this subsection we give the definition of the local index 
for the given data $(M,W,D)$, $(V_\infty, \rho_\infty, \epsilon_\infty)$, and
$
\{V_\alpha , D_\alpha,  C_\alpha, \delta_\alpha, \rho_\alpha, \epsilon_\alpha \}_{\alpha \in A}$ 
under the assumption that the Riemannian metric on $M$ is complete. 
As in the previous Subsection 
we put $D_t:=D+t\sum_{\alpha}\rho_{\alpha}D_{\alpha}\rho_{\alpha}$. 
Since we assume that $D$ and $\{D_{\alpha}\}_{\alpha\in A}$ 
have finite propergation speed, we have the following.

\begin{lemma}\label{Assumption in FFY}
If $t$ is large enough so that the inequality in 
Proposition~\ref{KeyProp-local} holds, then 
$D_t$ satisfies the Assumption~\ref{assumption for operator}. 
In particular if $M=V$, then the kernel of $D_t$ is trivial. 
\end{lemma}
By Theorem~\ref{spectrum} and Theorem~\ref{deformation invariance} 
we have the following.  
\begin{proposition}\label{mainprposition}
If $t$ is large enough so that the inequality in 
Proposition~\ref{KeyProp-local} holds, 
then the space of $L^2$-solutions of $D_ts=0$ 
is finite dimensional and its super-dimension is 
independent for $t\gg 1$ and any other 
continuous deformations of data. 
\end{proposition}
\begin{definition}\label{def. of abstract local index}
We define the local index $\ind(M,V,W)$ 
as the index of $D_t$ in the sense of Section~3. 
\end{definition}

In the case of cylindrical end we have the following sum formula of local indices by 
Proposition~\ref{gluing prop}. 

\begin{lemma}\label{sum formula} 
For $i=1,2$ 
let $M_i$ be manifolds with cylindrical ends $V_i=N_i\times \R_{>0}$ and 
$N^0_i$ be connected components of $N^0_i$.  
Suppose that there is an isometry $\phi:N_1^0 \to N_2^0$,
and for some $R>0$ the map $\phi:N_1^0 \times (0, R) \to N_2^0 \times (0,R)$
given by $(x, r) \mapsto (\phi (x), R-r)$ induces 
the isomorphism between the acyclic compatible systems on them. 
Then we can glue $M_1\setminus (N_1^0 \times [R,\infty))$ 
and $M_2 \setminus (N_2^0 \times [R,\infty))$ 
to obtain a new manifold $\hat{M}$ 
with cylindrical end $\hat{V}=\hat{N} \times (0,\infty)$ for
$\hat{N}=(N_1\setminus N_1^0)\cup (N_2 \setminus N_2^0)$, 
and we also have a Clifford module bundle $\hat W$ 
obtained by gluing $W$ and $W'$ on 
$N_1^0 \times (0, R) \cong N_2^0 \times (0,R)$. 
Then we have 
$$
\ind(\hat{M},\hat{V},\hat W)=\ind(M_1,V_1,W_1)+\ind(M_2,V_2,W_2). 
$$
\end{lemma}

\section{
The case of Hamiltonian torus actions}\label{sufficient condition}
Let $G$ be a compact torus, $(M,\omega)$ a closed Hamiltonian $G$-manifold, and $(L,\nabla)\to (M,\omega)$ a $G$-equivariant prequantizing line bundle, i.e., $L$ is a Hermitian line bundle over $M$ with a Hermitian connection $\nabla$ whose curvature form is equal to $-2\pi\sqrt{-1}\omega$, and all these data are $G$-equivariant. Choose a $G$-invariant compatible almost complex structure $J$ on $M$. We denote by $g_J$ the $G$-invariant Riemannian metric associated with $J$. Let $TM_\C$ be the tangent bundle regarded as a complex vector bundle with the almost complex structure $J$, and we set $W:=\wedge^\bullet_\C TM_{\C}\otimes L$. $W$ is equipped with a structure of a $\Z/2$-graded Clifford module bundle. See Subsection~\ref{DH} for more details. Let $D\colon \Gamma (W)\to \Gamma (W)$ be a Dirac-type operator. Then, the main theorem of this section is as follows. 
\begin{theorem}\label{sufficient condition theorem}
Let $V$ be an $G$-invariant open set of $M$ which satisfies the following condition: for each $x\in V$ there exists an open neighborhood $V_x$ of $x$ such that for all $y\in V_x$ the restriction of $L$ on $G_x^{\perp}$-orbit $G_x^{\perp}\cdot y$ has no nontrivial parallel sections. Here $G_x^{\perp}$ is the orthogonal complement of the stabilizer of $x$, see the proof of Proposition~\ref{rem good open covering} for more details. Then, for the above data $M$, $W$, and $D$, there exist data $\{V_\infty ,\rho_\infty ,\epsilon_\infty\}$ and $\{V_H , D_H,  C_H, \delta_H,  \rho_H, \epsilon_H \}_{H \in A}$ which satisfy the conditions in Section~\ref{new}. 
\end{theorem}

We will show Theorem~\ref{sufficient condition theorem} by two steps. As a first step, we construct $\{V_H\}_{H\in A}$ for a $G$-action with finitely many orbit types in Subsection~\ref{VH}. As a second step, in Subsection~\ref{DH} 
for a closed Hamiltonian $G$-space equipped with $G$-equivariant prequantizing line bundle we construct the other data and show that they satisfy the conditions  in Section~\ref{new}. 
\subsection{A construction of $\{ V_H\}_{H\in A}$ for torus actions}\label{VH}
Let $X$ be a manifold equipped with an action of a compact torus $G$. Let $A$ be the set of all the subgroups of $G$ which appear as  stabilizers
\[
G_x:=\{ g\in G\,|\, g x=x\}
\]
at some points $x\in X$. Note that $A$ has a partial order with respect to the inclusion. 
The following proposition is well-known. 
\begin{proposition}[\cite{Kawakubo}]\label{finiteness_of_A}
If $X$ is compact, then $A$ is a finite set.
\end{proposition}
\begin{lemma}[Existence of a good open covering]\label{good open covering}
Suppose that $A$ is a finite set. Then, there exists an open covering $\{V_H \}_{H \in A}$ of $X$ parameterized by $A$ satisfying the following properties. 
\begin{enumerate}
\item Each $V_H$ is $G$-invariant.
\item For each $x \in V_H$ we have $G_x \subset H$.
\item If $V_H \cap V_{H'}\neq \emptyset$, then we have $H \subset H'$ or $H \supset H'$.
\end{enumerate}
Moreover, if for each $x\in X$ there is an open neighborhood $V_x$ of $x$ such that $V_{gx}=gV_x$ for all $g\in G$, then $\{V_H\}_{H\in A}$ satisfies $V_H\subset \bigcup_{G_x=H}V_x$. 
\end{lemma}
We give a proof of Lemma~\ref{good open covering} in Appendix~\ref{proof of good open covering}. The following technical lemma will be used in Proposition~\ref{aprioriestimate2}. 
\begin{lemma}\label{G-invariant refinement of good open covering}
Let $\{V_H \}_{H \in A}$ be the open covering of $X$ as in Lemma~\ref{good open covering}. Then there exists a refinement $\{ V'_H\}_{H\in A}$ of $\{V_H \}_{H \in A}$ such that $\{ V'_H\}_{H\in A}$  satisfies the conditions in Lemma~\ref{good open covering} and $\overline{V'_H}\subset V_H$ for each $H\in A$. 
\end{lemma}
\begin{proof}
Since any $G$-invariant refinement $\{ V'_H\}_{H\in A}$ of $\{V_H \}_{H \in A}$  automatically satisfies the conditions in Lemma~\ref{good open covering}, it is sufficient to prove that there exists a $G$-invariant refinement $\{ V'_H\}_{H\in A}$ of $\{V_H\}_{H\in A}$ such that $\overline{V'_H}\subset V_H$ for each $H\in A$. 

Let $\pi\colon X\to X/G$ be the quotient map. We set $U_H:=\pi(V_H)$ for each $H\in A$. 
We take a finite open covering $\{ V''_H\}_{H\in A}$ of $X$ which satisfies $\overline{V''_H}\subset V_H$ for each $H\in A$. 
We set $U''_H:=\pi(V''_H)$ for each $H\in A$. Then since $\pi$ is a closed map it is easy to see that $\{ U''_H\}_{H\in A}$ is an open covering of $X/G$ which satisfies $\overline{ U''_H}\subset U_H$ for each $H\in A$. 
We set $V'_H:=\pi^{-1}(U''_H)$ for each $H\in A$. 
One can show that $\{V'_H\}_{H\in A}$ is the required open covering of $X$. 
This completes the proof. 
\
\end{proof}

By using Lemma~\ref{good open covering} we show the following proposition. 
\begin{proposition}\label{rem good open covering}
Suppose that $A$ is a finite set. Then, there exists a collection of data $\{V_H, G_H, \iota_{H H'} \ | \ H, H'\in A\}$ satisfying the following properties. 
\begin{enumerate}
\item $\{V_H\}_{H\in A}$ is the open covering of $X$ obtained in Lemma~\ref{good open covering}. 
\item For each $H\in A$, $G_H$ is a compact torus which acts on $V_H$  whose isotropy subgroup $(G_H)_x$ at each $x\in V_H$ is a finite group.   
\item For each $H,H'\in A$ with $V_H\cap V_{H'}\neq \emptyset$ and $\dim G_{H}\leq \dim G_{H'}$, $\iota_{HH'}$ is an embedding of $G_{H}$ into $G_{H'}$ as a subgroup. We assume that $\iota_{HH}={\rm id}_{G_H}$.   
\item For each $H,H'\in A$, $x\in V_H\cap V_{H'}$ and $g\in G_{H}$ with $\dim G_{H}\leq \dim G_{H'}$, we have $\iota_{H H'}(g)x=gx$. 
\end{enumerate}
\end{proposition}
\begin{proof}
We endow $G$ with a rational flat Riemannian metric. Precisely speaking we take a Euclidean metric on the Lie algebra of $G$ such that the intersection of the integral lattice and the lattice generated by some orthonormal basis has rank $n$. We extend it to the whole $G$. 

For a subgroup $H$ of $G$ let $H^{\bot}$ be the orthogonal complement of $H$ defined as the image of the orthogonal complement of the Lie algebra of $H$ by the exponential map. Since the metric is rational $H^{\bot}$ is well-defined as a compact subgroup of $G$ and it has only finitely many intersection points $H \cap H^{\bot}$. 

Let $\{ V_H\}_{H\in A}$ be the open covering of $X$ obtained in Lemma~\ref{good open covering}. For each $H\in A$ we define $G_{H}$ to be $H^{\perp}$. By definition $G_H$ acts on $V_H$ with finite isotropic subgroup. If $V_{H}\cap V_{H'}\neq \emptyset$ for $H, H'\in A$, then we have $H \subset H'$ or $H \supset H'$ by Lemma~\ref{good open covering}. In case of $H \supset H'$ we have $G_{H}\subset G_{H'}$ by definition. Then we denote by $\iota_{H H'}$ is the natural embedding of $G_{H}$ into $G_{H'}$. In case of $H \subset H'$ we define $\iota_{H' H}$ similarly. Then $\{ V_H, G_H, \iota_{H H'} \ | \ H, H'\in A\}$ is the required data. \hfill\hfill\hfill\
\end{proof}

\subsection{Constructions of $\{V_\infty ,\rho_\infty ,\epsilon_\infty\}$ and $\{V_H , D_H,  C_H, \delta_H,  \rho_H, \epsilon_H \}_{H \in A}$ for Hamiltonian torus actions}\label{DH}
Let $(M,\omega)$ be a $2m$-dimensional closed symplectic manifold equipped with a Hamiltonian action of an $n$-dimensional compact torus $G$. In this case each orbit is an affine isotropic torus in $M$. Suppose that there is a $G$-equivariant prequantizing line bundle $(L,\nabla)$ on $(M,\omega)$. Since all orbits are isotropic the restriction of $(L,\nabla)$ on each orbit is a flat line bundle. 

Fix a $G$-invariant $\omega$-compatible almost complex structure $J$ on $M$. We denote by $g_J$ the $G$-invariant Riemannian metric associated with $J$. We remark that the metric $g_J$ restricted to each orbit is a flat affine metric because it is $G$-invariant.  

We define a Hermitian vector bundle $W$ on $M$ by 
\[
W:=\wedge^\bullet_\C TM_{\C}\otimes L, 
\]
where $TM_\C$ is the tangent bundle regarded as a complex vector bundle with an almost complex structure $J$. A Clifford module structure $c\colon Cl(TM)\to \End (W)$ is defined by 
\begin{equation}\label{Clifford}
c(u)(\varphi)=u\wedge \varphi-u\llcorner \varphi 
\end{equation}
for $u\in TM$, $\varphi\in W$, where $h$ is the Hermitian metric on $M$ defined by
\[
h_x(u,v)=g_x(u,v)+\sqrt{-1}g_x(u,Jv)
\]
for $x\in M$, $u,v\in T_xM$ and $\llcorner$ is the interior product with respect to $h$, namely, 
\[
v\llcorner (v_1\wedge v_2\wedge \cdots \wedge v_k):=\sum_{i=1}^k(-1)^{i-1}h(v_i,v)v_1\wedge \hat{v_i}\wedge \cdots \wedge v_k
\]
for $v_1,\ldots , v_k\in TM$. 

Let $V$ be a $G$-invariant open set of $M$. 
Since $M$ is compact $G$ acts on $V$ with finitely many orbit types. Hence we can apply Proposition~\ref{rem good open covering} to $X=V$. Let $\{V_H, G_H, \iota_{H H'} \ | \ H, H'\in A\}$ be the data for $V$ which is obtained by applying Proposition~\ref{rem good open covering} to $X=V$. We construct the data $\{D_H\}_{H\in A}$ as in Section~\ref{new} in this case. For $H\in A$, let $\pi_H\colon V_H\to V_H/G_H$ be the natural projection. By the construction of $V_H$, $\pi_H$ is a fiber bundle in the sense of orbifolds. We denote by $T[\pi_H]$ the tangent bundle along fibers of $\pi_H$. Let $(T[\pi_H]\oplus JT[\pi_H])^\bot$ the orthogonal complement of $T[\pi_H]\oplus JT[\pi_H]$ with respect to $g_J$. Since $(g_J,J)$ is $G$-invariant the $H^\bot$-action preserves $(T[\pi_H]\oplus JT[\pi_H])^\bot$ and $JT[\pi_H]$. Then we define 
\[
\begin{split}
E_H&:=(T[\pi_H]\oplus JT[\pi_H])^\bot /H^{\bot},\\
F_H&:=JT[\pi_H]/H^{\bot}.
\end{split}
\] 
It is obvious that $T\left(V_H/G_H\right)$ is identified with $E_H\oplus F_H$. 

Since $g_J$ is invariant under $J$, $J$ preserves $(T[\pi_H]\oplus JT[\pi_H])^\bot$. In particular $T[\pi_H]\oplus JT[\pi_H]$ and $(T[\pi_H]\oplus JT[\pi_H])^\bot$ have the structures of Hermitian vector bundles with respect to the restriction of $(g_J,J)$ to them. Moreover $(g_J,J)$ is $G$-invariant it descends to the Hermitian structure on $E_H$. Then, we define $W_{1,H}$ and $W_{2,H}$ by 
\[
\begin{split}
W_{1,H}&:=\wedge^\bullet_\C(T[\pi_H]\oplus JT[\pi_H])_\C\otimes L,\\
W_{2,H}&:=\wedge^\bullet_\C(E_H)_\C\otimes L, 
\end{split}
\]
and define the Clifford module structures 
\[
\begin{split}
&c_{1,H}\colon Cl(T[\pi_H]\oplus \pi_H^*F_H)\to \End\left( W_{1,H}\right) ,\\
&c_{2,H}\colon Cl\left( E_H\right)\to \End\left(W_{2,H}\right)
\end{split}
\]
by the same formula as in \eqref{Clifford}. By definition, $TV_H$ has a decomposition 
\begin{equation*}
TV_H=(T[\pi_H]\oplus \pi_H^*F_H)\oplus \pi_H^*E_H
\end{equation*}
as Hermitian vector bundles. With respect to this decomposition there are the following isomorphisms 
\begin{eqnarray*}
Cl(TM)|_{V_H}&\cong& Cl(T[\pi_H]\oplus \pi_H^*F_H)\otimes Cl(\pi_H^*E_H),\\
W|_{V_H}&\cong& W_{1,H}\otimes \pi_H^*W_{2,H}.
\end{eqnarray*}
Then by the direct calculation one can check $c=c_{1,H}\otimes c_{2,H}$ under the above identifications. 
\begin{remark}
Our convention of the tensor product of Clifford modules is as follows. Let $A=A^0\oplus A^1$ and $B=B^0\oplus B^1$ be two $\Z/2$-graded algebras. Then we define a structure of a $\Z/2$-graded algebra on the tensor product $A\otimes B$ as follows. The $\Z/2$-grading is defined in the usual way. 
The multiplicative structure is defined by 
\[
(a\otimes b)\cdot (a'\otimes b')=(-1)^{\deg b\deg a'}(aa')\otimes (bb'),
\]
where $a,a'\in A^0\cup A^1$ and $b,b'\in B^0\cup B^1$. Now let $R_A$ and $R_B$ be $\Z/2$-graded $A$ and $B$ modules respectively. Then we define a structure of a $\Z/2$-graded $A\otimes B$-module on the tensor product $R_A\otimes R_B$ by the following formula 
\[
(a\otimes b)\cdot(r_A\otimes r_B)=(-1)^{\deg b\deg r_A}(ar_A\otimes br_B), 
\]
where $a\in A^0\cup A^1$, $b\in B^0\cup B^1$, $r_A\in R_A^0\cup R_A^1$ and $r_B\in R_B^0\cup R_B^1$. 
\end{remark}

Now we construct $\{D_H\}$ as follows.  Let $\nabla^{T[\pi_H]}\colon \Gamma (TV_H)\to \Gamma (T^*[\pi_H]\otimes TV_H)$ be the the family of Levi-Civita connections along fibers of $\pi_H$, namely,  
\[
\nabla^{T[\pi_H]}=\iota_H^*\otimes q_H\circ \nabla^{TM}\circ q_H ,
\]
where $\iota_H \colon T[\pi_H]\to TV_H$ is the natural inclusion, $\nabla^{TM}$ is the Levi-Civita connection on $TM$ with respect to $g_J$, and $q_H\colon TV_H\to TV_H$ is the orthogonal projection to $T[\pi_H]$ with respect to $g_J$. $\nabla^{T[\pi_H]}$ induces the family of Hermitian connections on $\wedge_\C^\bullet TM_\C|_{V_H}$ along fibers of $\pi_H$, which is denoted by $\nabla^{\wedge_\C^\bullet TM_\C|_{V_H}}$. We define the family of Hermitian connections $\nabla^H$ on $W|_{V_H}$ along fibers of $\pi_H$ by
\[
\nabla^H:=\nabla^{\wedge_\C^\bullet TM_\C|_{V_H}}\otimes\id+\id\otimes\left( \iota_H^*\otimes \id \circ\nabla^L\right) \colon \Gamma (W|_{V_H})\to \Gamma (T^*[\pi_H]\otimes W|_{V_H}). 
\]
Then we define $D_H \colon \Gamma (W|_{V_H})\to \Gamma (W|_{V_H})$ by
\begin{equation}\label{D_H}
D_H:=c_{1,H}\circ p_H\circ \nabla^H, 
\end{equation}
where $p_H\colon T^*[\pi_H]\to T[\pi_H]$ is the isomorphism via $g_J$. By the construction it is clear that $\{D_H\}_{H\in A}$ has the following properties. 
\begin{proposition}\label{properties1}
Let $\{D_H\}_{H\in A}$ be the data defined by \eqref{D_H}. Then $\{D_H\}_{H\in A}$ satisfies the following properties:
\begin{enumerate}
\item $D_H\colon \Gamma (W|_{V_H})\to \Gamma (W|_{V_H})$ is a first order formally self-adjoint differential operator of degree-one.
\item $D_H$ contains only the derivatives along orbits of the $G_H$-action on $V_{H}$, i.e. $D_H$ commutes with multiplication of $G_H$-invariant smooth functions on $V_H$. 
\item The principal symbol $\sigma (D_H)$ of $D_H$ is given by $\sigma (D_H)=c\circ p_H\circ \iota_H^*\colon T^*V_H\to \End (W|_{V_H})$, where $\iota_H\colon T[\pi_H]\to TV_H$ is the natural inclusion, $p_H\colon T^*[\pi_H]\to T[\pi_H]$ is the isomorphism induced by the Riemannian metric and $c\colon T[\pi_H]\to \End (W|_{V_H})$ is the Clifford multiplication. 
\item For an orbit $\CO$ of the $G_H$-action on $V_H$ let $\tilde{u}\in \Gamma\left(T[\pi_H]^{\perp}|_{\CO}\right)^{G_H}$ be a $G_H$-equivariant section of $T[\pi_H]^{\perp}|_{\CO}$. 
$\tilde{u}$ acts on $W|_{\CO}$ by the Clifford multiplication $c(\tilde{u})$. Then $D_H$ and $c(\tilde{u})$ anti-commute each other, i.e.
\[
0=\{ D_H,c(\tilde{u}) \}:=D_H\circ c(\tilde{u})+c(\tilde{u})\circ D_H
\]
for all orbits $\CO$ and $\tilde{u}\in \Gamma\left(T[\pi_H]^{\perp}|_{\CO}\right)^{G_H}$. 
\end{enumerate}
\end{proposition}
The properties (1), (2), and (3) in Proposition~\ref{properties1} imply that $D_H$ is of Dirac-type when restricted to each fiber of $\pi_H$. 

We show that the above $\{D_H\}_{H\in A}$ satisfies the property (3)-(b) in Section~\ref{new}. 
\begin{proposition}\label{non_negativity}
Let $\{D_H\}_{H\in A}$ be the data constructed by \eqref{D_H}. If $V_H\cap V_{H'}\neq \emptyset$, then the anti-commutator $\{D_{H},D_{H'}\}$ is a non-negative operator over $V_H\cap V_{H'}$. 
\end{proposition}
It suffices to show the non-negativity of the anti-commutator along fibers. We consider the following model. 
\begin{itemize}
\item $E$ : Euclidean space 
\item $\Gamma$ : maximal lattice of $E$
\item $F:=E/\Gamma$ : flat torus 
\item $W\to F$ : $Cl(TF)$-module bundle 
\item $\nabla:\Gamma(W)\to \Gamma(TF\otimes W)$ : 
flat Clifford connection of $W$ 
\item $c:TF\otimes W\to W$ : Clifford action of $TF$ 
\item $A$ : finite set 
\item $\{E_{\alpha}\}_{\alpha\in A}$ : family of subspaces of $E$
\item $\{p_{\alpha}\}$ : family of orthogonal projections to $\{E_{\alpha}\}$ 
\item We assume $p_{\alpha}p_{\beta}=p_{\beta}p_{\alpha}$ for all 
$\alpha, \beta\in A$. 
\end{itemize}
Using the metric we have the identification $TF=T^*F=F\times E$. 
For a symmetric endomorphism $S : E\to E$
let $\hat S :F\times E \to F\times E$ be the induced bundle map on 
the (co)tangent bundle. 
We define a differential operator $D_S$ by the composition 
$$
D_S:=c\circ \hat S\circ\nabla : \Gamma(W)\to \Gamma(W). 
$$ 
Since $S$ is symmetric $D_S$ is a self-adjoint operator.  
%
Let $S_1$ and $S_2$ be symmetric endomorphisms which commute each other. 
By the direct computation using the orthonormal basis of $E$ consisting of 
simultaneously eigenvectors of $S_1$ and $S_2$ we have the following equality
$$
D_{S_1}\circ D_{S_2}+D_{S_2}\circ D_{S_1}=
2\nabla^*\circ\hat S_1\circ\hat S_2\circ\nabla, 
$$where 
$\nabla^*:\Gamma(TF\otimes W)\to \Gamma(W)$ 
is the adjoint operator of $\nabla$.  
%
%

In particular, by putting $S_1:=p_{\alpha}$ and $S_2:=p_{\beta}$ we have the following. 
\begin{lemma}\label{cor2}
$D_{p_{\alpha}}D_{p_{\beta}}+D_{p_{\beta}}D_{p_{\alpha}}=2D_{\alpha\beta}^2$, 
where $D_{\alpha\beta}$ is the self-adjoint operator 
$c\circ\hat p_{\alpha\beta}\circ\nabla$ 
defined by the projection $p_{\alpha\beta}$ to the 
intersection $E_{\alpha}\cap E_{\beta}$. 
\end{lemma}

\begin{proof}[Proof of Proposition~\ref{non_negativity}] 
By the third property of $\{V_H\}_{H\in A}$ in Proposition~\ref{rem good open covering} we have a family of tori at each point on $V$ which comes from a family of subspaces whose projections commute each other. Then the claim follows from Lemma~\ref{cor2}. \hfill \  
\end{proof}

If $V$ satisfies the assumption in Theorem~\ref{sufficient condition theorem} we show that $\{D_H\}_{H\in A}$ satisfies the condition~(3)-(a). 
\begin{proposition}\label{Hamiltonian acyclic}
Let $\{D_H\}_{H\in A}$ be as above. Suppose that for each $x\in V$ there exists an open neighborhood $V_x$ of $x$ such that for all $y\in V_x$ the restriction of $L$ on $G_x^{\perp}$-orbit $G_x^{\perp}\cdot y$ has no nontrivial parallel sections. Then, for each $H\in A$ and each orbit $\CO$ of the $G_H$-action on $V_H$ $D_H|_{\CO}$ has zero kernel. 
\end{proposition}

\begin{corollary}\label{aprioriestimate}
Under the assumption in Proposition~\ref{Hamiltonian acyclic}, for each orbit $\CO$ of the $G_H$-action on $V_H$ and arbitrary differential operator $Q$ of order at most $2$ along $\CO$, there exists a constant $C_{Q,\CO}$ such that the inequality 
$$
\left|\int_{\CO}(s_{\CO},Qs_{\CO})\right|\le C_{Q,\CO}\int_{\CO}\left|D_Hs_{\CO}\right|^2
$$
holds for all sections $s_{\CO}\in \Gamma (W|_{\CO})$. We can take $C_{Q,\CO}$ continuously with respect to $\CO$. Namely, $C_{Q,\CO}$ can be taken as a continuous function on $V_H/G_H$. 
\end{corollary}
%
%

Proposition~\ref{Hamiltonian acyclic} follows from the assumption and the following lemma.
\begin{lemma}\label{vanishing of cohomology}
Let $(E,\nabla^E)\to T$ be a flat Hermitian line bundle on an $n$-dimensional torus. If the degree zero cohomology $H^0(T;E)$ with local system $(E,\nabla^E)$ vanishes, then all cohomologies $H^\bullet (T;E)$ vanish. 
\end{lemma}
\begin{proof}
Fix a flat Riemannian metric on $T$. 
Let $(x_1,\cdots, x_n)$ be a coordinate on $T$ 
such that $\{\partial/\partial x_i\}$ 
forms a parallel orthonormal frame of the tangent bundle. 
Let $\nabla^{T}$ be the Levi-Civita connection of the flat metric. 
Then the induced connection 
$\nabla^W:=\nabla^{E}\otimes\id+\id\otimes\nabla^{T}$ 
on $W:=E\otimes\wedge^{\bullet}T^*T$ is a flat connection 
and $(W,\nabla^W)$ is a 
Dirac bundle in the sense of \cite{spingeom}. 
The general Bochner identity for Dirac bundles 
(Theorem~8.2 in \cite{spingeom}) implies that 
any harmonic section in $\Gamma(W)$ is parallel.  
If $\omega=\sum_{I}s_Idx_I\in \Gamma(W)$ is a harmonic section of degree $k$, 
where $I$ runs multi-indices $I=(i_1<\cdots<i_k)$ and 
$dx_I=dx_{i_1}\wedge\cdots\wedge dx_{i_k}$, then 
we have $0=\nabla^W\omega=\sum_I(\nabla^Es_I)dx_I$, 
and hence, $\nabla^Es_I=0$. 
Since $H^0(M;E)=0$ we have $s_I=0$ and $\omega=0$. \hfill \ 
\end{proof}
\begin{proof}[Proof of Proposition~\ref{Hamiltonian acyclic}]
By taking the union along $G$-orbits we may assume that each $V_x$ is $G$-invariant. Note that each $V_H$ is contained in $\cup_{G_x=H}V_x$. The restriction of $L$ on each $H^{\perp}$-orbit has no non-trivial parallel sections because of the property $V_H\subset\cup_{G_x=H}V_x$ and our assumption. It implies $\ker D_H=0$ by Lemma~\ref{vanishing of cohomology}. \hfill\
\end{proof}

\begin{remark}
In the proof of Proposition~\ref{properties1}, ~\ref{non_negativity}, and~\ref{Hamiltonian acyclic} the fact that the orbit of the $G_H$-action on $V_H$ is a flat torus is essentially used.  
\end{remark}


From the fourth property of $\{D_H\}_{H\in A}$ in Proposition~\ref{properties1} we can show that following lemma. 
\begin{lemma}\label{anti-commutator}
For each $H\in A$ the anti-commutator $\{D, D_H\}$ is a differential operator along orbits of the $G_H$-action on $V_H$. 
\end{lemma}
\begin{proof}
Recall that, for each $H\in A$ the principal symbol of horizontal direction $D^\text{hol}$ of $D$ with respect to $\pi_H$ anti-commutes not only with the symbol of $D_H$, but also with the whole operator $D_H$. The statement follows from this property. It is straightforward to check it using local description. 
Instead of giving the detail of the local calculation, however, 
we here give an alternative formal explanation for the above lemma. 
For $b \in V_H/G_H$ let ${\mathcal W}_b$ be
the sections of the restriction of $W$ on the fiber $\pi_H^{-1}(b)$.
Then ${\mathcal W}=\coprod {\mathcal W}_b$ is formally
an infinite dimensional vector bundle over $V_H/G_H$. 
We can regard $D_H$ as an endomorphism on ${\mathcal W}$.
Then $D_H$ is a order-zero differential operator
on ${\mathcal W}$ whose principal symbol is equal to 
$D_H$ itself.
Then, as a differential operator on ${\mathcal W}$, 
the anti-commutator 
$D_HD^\text{hol}+D^\text{hol}D_H$ is
an (at most) order-one operator whose principal symbol
is given by the anti-commutator between 
the Clifford action by $T\left(V_H/G_H\right)\cong \left(T[\pi_H]^{\perp}\right)^{G_H}$
and $D_H$. This principal symbol vanishes, which implies
that the anti-commutator is order-zero as a differential operator
on ${\mathcal W}$, i.e., it does not contain derivatives of
$V_H/G_H$-direction. \hfill \hfill \hfill \hfill \
\end{proof}

By applying Corollary~\ref{aprioriestimate} to $Q=\id$ and $Q=\{D, D_H\}$, we can show that $\{D_H\}_{H\in A}$ satisfies the condition (3)-(a) and (c) in Section~\ref{new}, respectively. 
\begin{proposition}\label{aprioriestimate2}
Under the assumption in Proposition~\ref{Hamiltonian acyclic}, for each $H\in A$ there exist positive constants $\delta_H$ and $C_H$ such that the inequalities 
\[
\delta_H\int_{\CO}\left| s_{\CO}\right|^2\le \int_{\CO}\left|D_Hs_{\CO}\right|^2,\quad 
\left|\int_{\CO}(s_{\CO},\{D, D_H\}s_{\CO})\right|\le C_H\int_{\CO}\left|D_Hs_{\CO}\right|^2
\]
hold for all sections $s_{\CO}\in \Gamma (W|_{\CO})$. 
\end{proposition}
\begin{proof}
We prove the first inequality. By applying Corollary~\ref{aprioriestimate} to $Q=\id$, for each orbit $\CO$ of the $G_H$-action on $V_H$ there exists a positive constant $C_{Q,\CO}$ depending continuously on $\CO$ such that the inequality in Corollary~\ref{aprioriestimate} holds for $Q=\id$. If necessary, by replacing $\{V_H\}_{H\in A}$ by its refinement as in Lemma~\ref{G-invariant refinement of good open covering}, we can assume that the inequality in Corollary~\ref{aprioriestimate} for $Q=\id$ still holds on a sufficiently small $G$-invariant neighborhood of $\overline{V_H}$. In particular, $C_{Q,\CO}$ takes the maximal value on the compact set $\overline{V_H}/G_H$ since $C_{Q,\CO}$ is a continuous function on $\overline{V_H}/G_H$. We denote by $C_H$ the maximal value of $C_{Q,\CO}$ on $\overline{V_H}/G_H$ and set $\delta_H:=\dfrac{1}{C_H}$. Then, $\delta_H$ is the required constant. The second inequality can be proved similarly. 
\
\end{proof}

Recall that $V$ is a $G$-invariant open subset of $M$ and $\{V_H\}_{H\in A}$ is the open covering of $V$ obtained in Lemma~\ref{good open covering}. We take and fix a $G$-invariant open neighborhood $V_\infty$ of $M\setminus V$. Set $\tilde{A}:=A\coprod \{\infty\}$. Then $\{V_H\}_{H\in \tilde{A}}$ is an open covering of $M$. Now we can show the following proposition. 

\begin{proposition}\label{partition of unity}
For the above data $M$, $V$, and $\{V_H\}_{H\in\tilde{A}}$ there exists a data $\{\rho_H,\epsilon_H\}_{H\in \tilde{A}}$ 
which satisfies the condition (4) in Section~\ref{new}. 
\end{proposition}
\begin{proof}
As we showed in the proof of Lemma~\ref{G-invariant refinement of good open covering} there exists a $G$-invariant open covering $\{V'_H\}_{H\in \tilde{A}}$ of $M$ such that $\overline{V'_H}\subset V_H$ for each $H\in \tilde{A}$. For such an open covering $\{V'_H\}_{H\in \tilde{A}}$ we can take a partition of unity $\{\phi_H\}_{H\in \tilde{A}}$ subordinate to $\{V_H\}_{H\in \tilde{A}}$ which satisfies $\phi_H>0$ on $\overline{V'_H}$ for each $H\in \tilde{A}$. 
For each $\phi_H$ we define $I(\phi_H)$ by 
\[
I(\phi_H)(x):=\int_G\phi_H(gx)dvol_G, 
\]
where $dvol_G$ is the Haar measure of $G$ with $\int_Gdvol_G=1$. Then we have a family of $G$-invariant functions $\{I(\phi_H)\}_{H\in \tilde{A}}$. Note that it is an another partition of unity subordinate to $\{V_H\}_{H\in \tilde{A}}$. We put $\rho_H:=I(\phi_H)/\sqrt{\sum_{H'}I(\phi_{H'})^2}$. We put $\epsilon_H:=\min_{x\in \overline{V'_H}}\rho_H(x)$. $\epsilon_H$ is a positive real number since $\phi_H>0$ on $\overline{V'_H}$ and $\overline{V'_H}$ is $G$-invariant. Then $\{\rho_{H},\epsilon_H\}_{H\in \tilde{A}}$ is the required data. \hfill\hfill\hfill\hfill\
\end{proof}

We show that $\{D_H\}$ and $\{\rho_H\}_{H\in A}$ satisfies the condition (5) in Section~\ref{new}. 
\begin{proposition}\label{elliptic}
For a family of non-negative numbers $\vec t=\{t_H\}_{H\in A}$ and the partition of unity $\{ \rho_H\}_{H\in A}$ obtained in Proposition~\ref{partition of unity} we define the operator $D_{\vec t}\colon \Gamma(W)\to \Gamma(W)$ by 
\[
D_{\vec t}:=D+\sum_{H\in A}t_H\rho_HD_H\rho_H. 
\]
Then, $D_{\vec t}$ is elliptic. 
\end{proposition}
\begin{proof}
With the identification of tangent bundles and cotangent bundles via the Riemannian metric, the principal symbol $\sigma_{\vec t}$ of $D_{\vec t}$ is written as 
\[
\sigma_{\vec t}=c\circ (\id_{TM}+\sum_{H\in A}t_H\rho_H p_H \rho_H), 
\]
where $p_H\colon TV_H\to T[\pi_H]$ is the orthogonal projection. We put 
\[
S:=\id_{TM}+\sum_{H\in A}t_H\rho_H p_H \rho_H .
\]
$S$ is a symmetric operator. In order to prove the proposition it is sufficient to show that for each $x\in M$ $S_x\colon T_xM\to T_xM$ is a positive operator. By definition, for $u\in T_xM$ we have 
\begin{equation*}
(S_xu,u)=\abs{u}^2+\sum_{H\in A}t_H\abs{\rho_H(x)p_H(u)}^2\ge 0,  
\end{equation*}
where we think $p_H (u)=0$ if $x\not\in V_H$.
Moreover, $(S_xu,u)=0$ if and only if $u=0$. This implies $S_x$ is positive. 
\end{proof}

\begin{proof}[Proof of Theorem~\ref{sufficient condition theorem}]
Theorem~\ref{sufficient condition theorem} follows from Proposition~\ref{rem good open covering}, Proposition~\ref{non_negativity}, Proposition~\ref{aprioriestimate2}, Proposition~\ref{partition of unity}, and Proposition~\ref{elliptic}.
\end{proof}


\section{Compatible fibration and acyclic compatible system}
\label{Compatible fibration and acyclic compatible system}

In this section we give a general framework which includes the  
construction in Section~\ref{sufficient condition} for Hamiltonian torus action . 
The framework does not relay on global torus action. 
For example it can be applied for locally toric Lagrangian fibration 
as in Section~\ref{four-dim ex}. 

\subsection{Compatible fibration}
Let $M$ be a smooth manifold. 
\begin{definition}\label{compatible fibration}
A {\it compatible fibration on $M$} is a collection of data 
$\{ \pi_{\alpha}\colon V_{\alpha }\to U_{\alpha}\mid \alpha\in A\}$ 
satisfying the following properties. 
\begin{enumerate}
\item $\{V_\alpha\}_{\alpha\in A}$ is an open covering of $M$. 
\item For each $\alpha \in A$, $U_\alpha$ is a manifold and $\pi_{\alpha}\colon V_{\alpha}\to U_{\alpha}$ is a fiber bundle whose fiber is a closed manifold. 
\item For each $\alpha$ and $\beta$, we have 
\[
V_{\alpha}\cap V_{\beta}=\pi_{\alpha}^{-1}(\pi_{\alpha}(V_{\alpha}\cap V_{\beta}))
=\pi_{\beta}^{-1}(\pi_{\beta}(V_{\alpha}\cap V_{\beta})). 
\]
\end{enumerate}
We often denote a compatible fibration on $M$ by $\{ \pi_\alpha\}_{\alpha \in A}$ for simplicity.  
\end{definition}


Let $\{ \pi_\alpha\}_{\alpha \in A}$ be a compatible fibration on $M$.

\begin{definition}\label{admissible subset}
A subset $C$ of $M$ is said to be {\it admissible} 
if $V_\alpha\cap C=\pi_\alpha^{-1}(\pi_\alpha(V_\alpha\cap C))$ for each $\alpha \in A$. 
\end{definition}
We can restrict a compatible fibration to an admissible open subset. 
\begin{proposition}
Let $C$ be a submanifold of $M$. Suppose $C$ is admissible. Then the data $\{ \pi_\alpha|_{C\cap V_\alpha}\colon C\cap V_\alpha \to \pi_\alpha(C\cap V_\alpha)\}_{\alpha \in A}$ is a compatible fibration on $C$. 
\end{proposition}


\begin{definition}\label{admissible function}
Let $f:M\to \R$ be a function. 
If 
$f$ is constant along fibers of $\pi_{\alpha}|_{V_{\alpha}}$ 
for all $\alpha\in A$, then we call $f$ an {\it admissible function}.
\end{definition}




The construction in Subsection~\ref{VH} (Proposition~\ref{rem good open covering}) is an example of 
compatible fibration, which is contained in a specific class of 
the compatible fibration. 
We give the precise definition of such a class. 

\begin{definition}
For $\alpha\in A$ and $x\in V_{\alpha}$, 
we define 
$A(\alpha;x)$ as follows. 
\[
A(\alpha;x):=\{\beta\in A \ | \ x\in V_{\alpha}\cap V_{\beta}, \ p_{\alpha\beta}^{\beta} : U_{\alpha\beta} \to \pi_{\beta}(V_{\alpha}\cap V_{\beta}) \mbox{ is a homeomorphism}\},
\]where $U_{\alpha\beta}$ is the push-out of $\pi_{\alpha}$ and $\pi_{\beta}$, i.e., 
$$
U_{\alpha\beta}=\{(b_{\alpha},b_{\beta}) \ | \ b_{\alpha}=\pi_{\alpha}(x), 
b_{\beta}=\pi_{\beta}(x) \ (x\in V_{\alpha}\cap V_{\beta})\} 
$$
and $p_{\alpha\beta}^{\beta} : U_{\alpha\beta} \to \pi_{\beta}(V_{\alpha}\cap V_{\beta}) $ is the natural map. 
\end{definition}

\begin{remark}\label{relation}
If $\beta\in A(\alpha, x)$ for $x\in V_{\alpha}\cap V_{\beta}$, then we have $\pi_{\beta}^{-1}\pi_{\beta}(x)\subset\pi_{\alpha}^{-1}\pi_{\alpha}(x)$. 
Conversely if $x\in V_{\alpha}\cap V_{\beta}$ satisfies 
$\pi_{\beta}^{-1}\pi_{\beta}(x)\subset\pi_{\alpha}^{-1}\pi_{\alpha}(x)$, 
then we have $\beta\in A(\alpha,x)$. 
\end{remark}

\begin{definition}[Good compatible fibration]\label{Good compatible fibration}
If a compatible fibration 
$\{\pi_{\alpha}:V_{\alpha}\to U_{\alpha}\}$ over $M$ 
satisfies the following condition  (4), then we call 
$\{\pi_{\alpha}:V_{\alpha}\to U_{\alpha}\}$ a {\it good compatible fibration}. 
\begin{itemize}
\item[(4)] If $V_{\alpha}\cap V_{\beta}\neq \emptyset$, then 
we have $\alpha\in A(\beta;x)$ or $\beta\in A(\alpha;x)$ 
for all $x\in V_{\alpha}\cap V_{\beta}$. 
\end{itemize} 
\end{definition}

Though good compatible fibrations have several nice properties, note that the product of two good compatible fibrations, which is defined in a natural way as in Section~\ref{product formula of local indices} is not a good compatible fibration in general. On the other hand 
the product of two compatible fibrations is also a compatible fibration.


Now we define an appropriate notion of Riemannian metric 
for a compatible fibration. 
For the projection $\pi$ of a fiber bundle over a manifold we denote by $T[\pi]$ the vector bundle consisting of tangent vectors along fibers of $\pi$. 
%

\begin{definition}\label{compatible metric}
A {\it compatible Riemannian metric of a compatible fibration} is a collection of Riemannian metrics of $M$ and $\{U_\alpha\}_{\alpha\in A}$ such that $\pi_\alpha\colon V_\alpha\to U_\alpha$ is a Riemannian submersion with respect to the Riemannian metrics for each $\alpha \in A$.  
\end{definition}

As we showed in 
Proposition~\ref{rem good open covering} any smooth manifold equipped with  
torus action has a structure of good compatible fibration. 
In this case any torus invariant metric 
gives a compatible Riemannian metric 

\subsection{Acyclic compatible system}
Let $\{ \pi_{\alpha}\}$ be a compatible fibration on $M$ with compatible Riemannian metric and $W\to M$ a Clifford module bundle. 
\begin{definition}[Compatible system of Dirac-type operators]\label{compatible acyclic system}
A {\it compatible system of Dirac-type operators} on $(\{ \pi_{\alpha}\}, W)$ 
is a data $\{ D_{\alpha}\}$ satisfying the following properties. 
\begin{enumerate}
\item $D_{\alpha}\colon \Gamma (W|_{V_{\alpha}})\to \Gamma (W|_{V_{\alpha}})$ is an order-one formally self-adjoint differential operator of degree-one.
\item $D_{\alpha}$ contains only the derivatives along fibers of $\pi_{\alpha}\colon V_{\alpha}\to U_{\alpha}$, i.e. $D_{\alpha}$ commutes with multiplication of the pull-back of smooth functions on $U_{\alpha}$. 
\item The principal symbol $\sigma (D_{\alpha})$ of $D_{\alpha}$ is given by $\sigma (D_{\alpha})=c\circ p_{\alpha}\circ \iota_{\alpha}^*\colon T^*V_{\alpha}\to \End (W|_{V_{\alpha}})$, where $\iota_{\alpha}\colon T[\pi_{\alpha}]\to TV_{\alpha}$ is the natural inclusion, $p_{\alpha}\colon T^*[\pi_{\alpha}]\to T[\pi_{\alpha}]$ is the isomorphism induced by the Riemannian metric and $c\colon T[\pi_{\alpha}]\to \End (W|_{V_{\alpha}})$ is the Clifford multiplication. 
\item For $b\in U_{\alpha}$ and $u\in T_bU_{\alpha}$, let $\tilde{u}\in \Gamma(TV_{\alpha}|_{\pi^{-1}_{\alpha}(b)})$ 
be the horizontal lift of $u$ with respect to the Riemannian metric of $M$ restricted to $V_\alpha$. 
$\tilde{u}$ acts on $W|_{\pi^{-1}_{\alpha}(b)}$ by the Clifford multiplication $c(\tilde{u})$. Then $D_{\alpha}$ and $c(\tilde{u})$ anti-commute each other, i.e.
\[
0=\{ D_{\alpha},c(\tilde{u}) \}:=
D_{\alpha}\circ c(\tilde{u})+c(\tilde{u})\circ D_{\alpha}
\]
for all $b\in U_{\alpha}$ and $u\in T_bU_{\alpha}$. 
\end{enumerate}
\end{definition}
The properties (1), (2), and (3) in Definition~\ref{compatible acyclic system}
imply that $D_{\alpha}$ is of Dirac-type when restricted to 
each fiber of $\pi_{\alpha}$. 
We call a compatible system of Dirac-type operators 
$\{D_{\alpha}\}$ a compatible system for short. 



\begin{definition}[Acyclic compatible system]\label{compatible strongly acyclic system}
A compatible system $\{D_{\alpha}\}$  is {\it acyclic} if 
it satisfies the following conditions. 
\begin{enumerate}
\item For each $\alpha\in A$ and $b\in U_\alpha$ $D_\alpha|_{\pi_\alpha^{-1}(b)}$ has zero kernel. 
\item If $V_{\alpha}\cap V_{\beta}\neq \emptyset$, then the anti-commutator 
$\{D_{\alpha},D_{\beta}\}$ is a non-negative operator over $V_{\alpha}\cap V_{\beta}$. 
\end{enumerate}
\end{definition}

\begin{example}
The data $\{D_H\}_{H\in A}$ constructed in Subsection~\ref{DH} 
is an acyclic compatible system.  See also Proposition~\ref{properties1}.  
\end{example}

Acyclicity of the compatible system implies the 
vanishing of the kernel of the weighted averages of $D_{\alpha}$'s.

\begin{lemma}\label{strongly -> acyclic}
If a compatible system  $\{D_{\alpha}\}$ is acyclic , then 
the operator 
$$
\sum_{\beta \in A(\alpha;x)}t_{\beta}D_{\beta}\colon 
\Gamma (W|_{\pi_{\alpha}^{-1}(\pi_{\alpha}(x))})\to 
\Gamma(W|_{\pi_{\alpha}^{-1}(\pi_{\alpha}(x))})
$$ has zero kernel
for all $\alpha\in A$, $x\in V_{\alpha}$ and 
any family of non-negative numbers 
$\{ t_{\beta}\}_{\beta \in A(\alpha;x)}$ satisfying 
$t_{\beta}>0$ for some $\beta$. 
Note that the above operator is well-defined 
as a differential operator on $\pi_{\alpha}^{-1}(\pi_{\alpha}(x))$ because of 
Remark~\ref{relation}.
\end{lemma}
\begin{proof}
If $\{D_{\alpha}\}$ is an acyclic compatible system, 
then we have 
$$ 
\left(D_{\alpha}+\sum_{\beta\in A(\alpha;x)}
\tau_{\beta}D_{\beta}\right)^2=
D_{\alpha}^2+\sum\tau_{\beta}\{D_{\alpha},D_{\beta}\}+
\left(\sum_{\beta}\tau_{\beta}D_{\beta}\right)^2\ge D_{\alpha}^2
$$for any family of non-negative numbers 
$\{\tau_{\beta}\}$. 
Suppose $\left(\sum_{\beta\in A(\alpha;x)}t_{\beta}D_{\beta}\right)s=0$ for 
$s\in \Gamma(W|_{\pi_{\alpha}^{-1}(\pi_{\alpha}(x))})$. 
Take $\alpha_0\in A(\alpha;x)$ so that $t_{\alpha_0}$ is not $0$. 
Then we have 
$$
\left(D_{\alpha_0}+\sum_{\beta\in A(\alpha;x)\setminus\{\alpha_0\}}
(t_{\beta}/t_{\alpha_0})D_{\beta}\right)s=0
$$and $s=0$ by the above inequality and the first condition 
in Definition~\ref{compatible strongly acyclic system}.  
\end{proof}

\begin{example}
Let $M$ be $\R\times S^1$ with the standard Riemannian metric and $(t,\theta)$ its standard coordinate. We introduce a compatible fibration on $M$ by 
\[
\begin{split}
\pi_{\alpha}\colon V_{\alpha}&:=(-\infty ,1)\times S^1\to U_{\alpha}:=(-\infty ,1)\\ 
\pi_{\beta}\colon V_{\beta}&:=(-1,\infty)\times S^1\to U_{\beta}:=(-1,\infty ).
\end{split}
\] 

Let $W$ be the trivial rank $2$ Hermitian vector bundle $M\times \C^2$ on $M$.
We denote
$$
c_0:=
\begin{pmatrix}
0 & 1 \\
-1 & 0
\end{pmatrix},\quad
c_1:=
\begin{pmatrix}
0 & \sqrt{-1} \\
\sqrt{-1} & 0
\end{pmatrix},\quad
\epsilon_0=
\begin{pmatrix}
1 & 0 \\
0 & -1
\end{pmatrix}
.
$$
We introduce the obvious ${\Z/2}$-grading structure on $W$ by the eigen-decomposition for $\epsilon_0$,
and define the Clifford multiplication of $Cl(TM)$ on $W$ by
\[
c(\partial_{\theta})=c_0,\qquad
c(\partial_t)=c_1
\]
For smooth functions $f_\alpha\colon V_{\alpha}\to \R$,  $f_{\beta}:V_{\beta}\to \R$, let $D_\alpha$, $D_\beta$ be differential operators on $\Gamma (W|_{V_\alpha})$, $\Gamma(W|_{V_\beta})$ which are defined by 
\[
D_{\alpha}:=c_0\left(\partial_{\theta}+\sqrt{-1} f_{\alpha}(t,\theta)\right),\quad
D_{\beta}:=c_0\left(\partial_{\theta}+\sqrt{-1} f_{\beta}(t,\theta )\right). 
\]
They are order-one formally self-adjoint differential operators of degree-one. We give a sufficient condition for the data $(D_\alpha, D_\beta )$ to be an acyclic compatible system.

On $V_\alpha \cap V_\beta$
we write
$f_\alpha=f+h$ and $f_\beta=f-h$ 
using $f:=(f_\alpha+f_\beta)/2$ and $h:=(f_\alpha-f_\beta)/2$. We also write $D=(D_\alpha +D_\beta)/2=c_0(\partial_\theta+ \sqrt{-1} f)$ and $V=(D_\alpha-D_\beta)/2=c_0 \sqrt{-1} h$ so that
$$
\{ D_\alpha, D_\beta\}=  2 D^2 -2 V^2=2(D^2 -h^2).
$$
The first eigenvalue of $D^2$ is estimated from below as follows.
Let
$$
U=U(t,\theta):=\exp(\sqrt{-1} \epsilon_0 
\int_0^\theta f(t,\theta') d\theta')
$$
be the unitary isomorphism from $W|_{V_\alpha \cap V_\beta}$ to the bundle $W'$ defined by
$$
W':=((-1,\infty) \times \R \times \C^2)/\sim,
\qquad (t,\theta, v) 
\cong (t, \theta+n, U(t,2\pi)^n v)
$$
over $V_\alpha \cap V_\beta=(-1,1) \times S^1$.
For each $t \in (-1,1)$ we restrict the operators
on $\{t\} \times S^1$.
Since $D=c_0 U^{-1} \partial_\theta U$, the minimum of the absolute values of the eigenvalues of $D$ is the same as that of $\partial_\theta$ on the flat bundle $W'|_{\{t\} \times S^1}$, which is equal to $\min_{n \in \Z}\{|n+\phi(t)|\} $,
where
$$
\phi(t):=\frac{1}{2\pi}\int_0^{2\pi} f(t,\theta')d\theta'.
$$
It implies that $\{D_\alpha, D_\beta\}$ is non-negative if
$$
\mbox{\rm (1)}\qquad\qquad
|h(t,\theta)| \leq \min_{n \in \Z}\{|n+\phi(t)|\}
$$
for every $(t,\theta) \in V_\alpha \cap V_\beta
=(-1,1) \times S^1$.

The condition for $D_\alpha|_{\pi^{-1}_\alpha(t)}$to have zero kernel is
$$
\mbox{\rm (2)}\qquad\qquad
\int_0^{2\pi} f_\alpha(t,\theta') d\theta' \nin \Z.
$$
for $t \in (-1,\infty)$. Similarly we need
$$
\mbox{\rm (3)}\qquad\qquad
\int_0^{2\pi} f_\beta(t,\theta') d\theta' \nin \Z.
$$
for $t \in (-\infty, 1)$.
These properties (1), (2) and (3) give
the required condition.
\end{example}

\begin{example}\label{T^nR^n}
For non-negative integers $m$ and $n$ satisfying $n\le m$ let $M$ be $\R^{2m-n}\times T^n$, where we regard $T^n$ as $\R^n/(2\pi\Z)^n$. 
Let 
$\{V_{\alpha}'\}_{\alpha\in A}$ be a finite open covering of $\R^{2m-n}$, and $\{ R_\alpha\}_{\alpha \in A}$ a family of subspaces of $\R^n$ spanned by rational vectors. 
Let $p_{\alpha}$ be the orthogonal projection of $R_{\alpha}$ to $\R^{2m-2}$. 
We assume that $p_{\alpha}p_{\beta}=p_{\beta}p_{\alpha}$ for each 
$\alpha, \beta\in A$. 
We put $V_{\alpha}:=V_{\alpha}'\times T^n$ and $T_{\alpha}:=R_{\alpha}/R_{\alpha}\cap(2\pi \Z)^n$. Define $U_{\alpha}$ to be $V_{\alpha}'\times T^n/T_{\alpha}$ and $\pi_{\alpha}:V_{\alpha}\to U_{\alpha}$ to be the natural projection. Then these data define a compatible fibration on $M$. 

Let $g$ be the standard product metric on $M$ and 
$J$ the almost complex structure on $M$ which is defined by
\[
J(\partial_{y_i})=
\begin{cases}
\partial_{\theta_i} & 1\le i\le n\\
\partial_{y_{m-n+i}} & n+1\le i\le m\\
-\partial_{y_{i-m+n}} & m+1\le i\le 2m-n 
\end{cases}
\]
for $x=(y,\theta)\in M$. 
Note that since $g$ is invariant under $J$, $(g,J)$ defines the Hermitian metric on $M$. 
By using the horizontal lift $\pi_\alpha^*TU_\alpha\to TV_\alpha$ with respect to $g$, it is obvious that $\{ \pi_\alpha\}$ is equipped with a compatible Riemannian metric.

Take a Hermitian line bundle $(L,\nabla^L)$ with Hermitian connection on $M$ whose restriction to $\pi_\alpha^{-1}(b)$ is a flat connection for each $\alpha\in A$ and $b\in U_\alpha$. We impose the following assumption. 
\begin{assumption}\label{nontrivial holonomy}
For all $\alpha$ and $b\in U_{\alpha}$ 
the restriction $\nabla^L|_{\pi^{-1}_{\alpha}(b)}$ is not trivially flat connection,  
i.e., its holonomy representation is non-trivial. 
\end{assumption}

We define a $\Z/2$-graded Clifford module bundle $W$ by 
\[
W:=\wedge^\bullet_\C TM_{\C}\otimes L, 
\]
where $TM_\C$ is the tangent bundle regarded as a complex vector bundle with the almost complex structure $J$. 
We can construct an compatible system $\{D_{\alpha}\}_{\alpha\in A}$ 
as in the same way in Subsection~\ref{D_H}. 

\begin{proposition}
$\{D_{\alpha}\}_{\alpha\in A}$ is an acyclic compatible system.  
\end{proposition}
\begin{proof}
By the construction it is obvious that $\{D_\alpha\}_{\alpha}$ satisfies the condition~(1), (2), and (3) in Definition~\ref{compatible acyclic system}. The condition~(4) in Definition~\ref{compatible acyclic system} follows from the fact that $g$ restricted to each fiber of $\pi_\alpha$ is flat. 
Assumption~\ref{nontrivial holonomy} and Lemma~\ref{vanishing of cohomology} imply that 
the kernel of $D_\alpha|_{\pi_\alpha^{-1}(b)}$ vanishes 
for each $\alpha$ and $b\in U_\alpha$. 
Lemma~\ref{cor2} implies that the anti-commutator $\{D_{\alpha},D_{\beta}\}$ 
is non-negative.
\end{proof} 
\end{example}

\section{Local index and localization formula} 
\label{Main theorem}

In this section we give the definition of 
local index for acyclic compatible system.  
We first give the definition under the assumption 
of local torus action. 
Note that the torus action itself is not 
essential to define the local index. 
To give the precise definition for general 
case (without torus action) we need some more definitions and settings, 
and we avoid it for simplicity in this paper. 
We give several remarks for the general case in the final subsection.  

\subsection{Definition of local index}

We consider the following setting. 
Let $M$ be a Riemannian manifold which is not necessarily compact 
and $W$ a $\Z/2$-graded $Cl(TM)$-module bundle with the Clifford 
multiplication $c$. 
Suppose that there exists an open subset $V$ of $M$ 
whose complement $M\setminus V$ is compact. 
Let $G$ be a compact torus  which acts on $V$ in an isometric way, 
and the action lifts to $W|_V$ so that it commutes with $c$. 

As we proved in Lemma~\ref{good open covering} and Proposition~\ref{rem good open covering} there exists an open covering $\{V_{\alpha}\}_{\alpha\in A}$ of $V$ and a structure of 
a good compatible fibration $\{\pi_{\alpha}\}$ on $V=\cup_{\alpha\in A} V_{\alpha}$. 
Moreover the metric on $M$ is a compatible Riemannian metric in the sense of 
Definition~\ref{compatible metric}. 
We take a smooth function $f:M\to \R$ which is $G$-invariant on the end of $V$ 
and a regular value $c\gg 1 $ of $f$ such that $f^{-1}((-\infty, c])$ is compact 
and contains $M\setminus V$. 
Then we consider the complete manifold 
$\hat M=f^{-1}((-\infty, c])\cup \left( f^{-1}(c)\times [0,\infty)\right)$ 
with the cylindrical end $\hat V=(f^{-1}((-\infty, c])\cap V) \left( f^{-1}(c)\times [0,\infty)\right)$ equipped with natural $G$-action. 
We can also construct a $\Z/2$-graded Clifford module bundle $\hat W\to \hat M$. See \cite{Fujita-Furuta-Yoshida} for example for the 
construction of the Clifford action on $\hat W$. 
Since $f^{-1}(c)$ is an admissible subset in the sense of Definition~\ref{admissible subset} the compatible fibration $\{\pi_{\alpha}\}_{\alpha\in A}$ 
can be extended to a compatible fibration $\{\hat\pi\}_{\alpha\in A}$ 
on $\hat V=\cup_{\alpha\in A}\hat V_{\alpha}$. 

Suppose that there exists a compatible system $\{D_{\alpha}\}_{\alpha\in A}$ on  the compatible fibration $\{\pi_{\alpha}\}_{\alpha}$. 
Since $D_{\alpha}$ and the anti-commutator $D_{\alpha}D_{\beta}+D_{\beta}D_{\alpha}$ are operators along $G$-orbit for all $\alpha,\beta\in A $, 
we have the following. 

\begin{proposition}
The compatible system $\{D_{\alpha}\}_{\alpha}$ 
can be extended to a compatible system $\{\hat D_{\alpha}\}_{\alpha}$ 
on $\{\hat\pi_{\alpha}\}_{\alpha}$, which has translationally invariance on the 
end. 
Moreover if $\{D_{\alpha}\}_{\alpha}$ is acyclic, then 
$\{\hat D_{\alpha}\}_{\alpha}$ is also acyclic. 
\end{proposition}

\begin{theorem}\label{acyclic->assumption}
If the compatible system $\{D_{\alpha}\}_{\alpha\in A}$ is acyclic, then, 
for a translationally invariant Dirac-type operator $D$ on $\hat M$, 
there exists a data  $(V_\infty, \rho_\infty, \epsilon_\infty)$ and  $\{\hat V_\alpha , \hat D_\alpha,  C_\alpha, \delta_\alpha,  \rho_\alpha, \epsilon_\alpha \}_{\alpha \in A}$ satisfying the conditions in Section~\ref{new}. 
In particular $D_t=D+\sum_{\alpha\in A}\rho_{\alpha}\hat D_{\alpha}\rho_{\alpha}$ 
satisfies Assumption~\ref{assumption for operator} for any $t \gg 0$. 
\end{theorem}

\begin{proof}
The proof is similar for the proofs in Subsection~\ref{DH}. 
In fact $\{C_{\alpha}, \delta_{\alpha}\}$ can be constructed as in 
Proposition~\ref{aprioriestimate2}, and
$\{\rho_{\alpha}, \epsilon_{\alpha}\}$ can be constructed as in 
Proposition~\ref{partition of unity}. 
\
\end{proof}

We can define the index $\ind(\hat M,\hat V,\hat W)$ as in 
Definition~\ref{def. of abstract local index} 
for an acyclic compatible system $\{\hat D_{\alpha}\}_{\alpha\in A}$. 

\begin{proposition}
The index $\ind(\hat M,\hat V,\hat W)$ does not depend on the 
choice of the completion $\hat M$. 
\end{proposition}
\begin{proof}
We show that for two functions $f_i:M\to \R$ and their regular values $c_i$ \ ($i=1,2$) 
the indices for $\hat M_i=f_i^{-1}((-\infty, c_i])\cup \left( f_i^{-1}(c_i)\times [0,\infty)\right)$ are equal. 
Without loss of generality we may assume that 
$f_1^{-1}((-\infty,c_1])\subset f_2^{-1}((-\infty, c_2])$, 
and we show that $\ind(\hat M_1, \hat V_1, \hat W_1)=
\ind(\hat M_2, \hat V_2, \hat W_2)$. 
In this case we have the desired equality 
by the sum formula (Lemma~\ref{sum formula}) and the vanishing lemma (Lemma~\ref{Assumption in FFY}). 
\
\end{proof}

Moreover since $\ind(\hat M, \hat V, \hat W)$ is invariant 
under the continuous deformation of the give data, 
we have the following.  

\begin{proposition}
$\ind(\hat M, \hat V, \hat W)$ does not depend on the choice 
of $\{\rho_{\infty}, \epsilon_{\infty}, C_{\alpha}$, $\delta_{\alpha}$, $\rho_{\alpha}, \epsilon_{\alpha}\}_{\alpha\in A}$. 
\end{proposition}

\begin{definition}\label{def. of local the index}
We define the {\it local index} $\ind(M,V,W)$ 
to be the index $\ind(\hat M,\hat V,\hat W)$.  
\end{definition}

\begin{remark}
The local index $\ind(M,V,W)$ depends on the local torus action on $V$ 
and the choice of open covering $\{V_{\alpha}\}_{\alpha\in A}$. 
\end{remark}

\subsection{Localization formula}
We consider the setting as in the previous subsection. 
Namely let $M$ be a Riemannian manifold and $W$ a $\Z/2$-graded 
$Cl(TM)$-module bundle. Suppose that there exists an open subset 
$V$ of $M$ with the compact complement $M\setminus V$ 
and on which a compact torus $G$ acts in an isometric way. 
We fix an open covering $\{V_{\alpha}\}_{\alpha}$ of $V$ 
so that we have a compatible fibration $\{\pi_{\alpha}\}_{\alpha}$ 
on $V$. 
Moreover suppose that there exists an acyclic compatible system $\{D_{\alpha}\}_{\alpha}$ on the compatible fibration $\{\pi_{\alpha}\}_{\alpha}$. 
We have the local index $\ind(M,V,W)$. 
The well-definedness of $\ind(M,V,W)$ implies the following 
excision formulas. 
\begin{theorem}[Excision formula 1]\label{excision1}
Let $M'$ be an open neighborhood of $M\setminus V$ such that $V\cap M'$ is $G$-invariant. Then we have 
$$
\ind(M,V,W)=\ind(M', V\cap M', W|_{M'}). 
$$
\end{theorem}

\begin{theorem}[Excision formula 2]\label{excisoin2}
Let $V'$ be a $G$-invariant open subset of $V$ such that $M\setminus V'$ is compact. Then we have 
$$
\ind(M,V,W)=\ind(M, V', W). 
$$
\end{theorem}

Note that if $M$ is closed, i.e., compact without boundary, then 
the local index is equal to the usual index of the Dirac-type operator,  
$\ind(M,W):=\ind(D)$. 
As a corollary of the excision formula 
we have the following localization formula 
of the index of Dirac-type operator. 

\begin{theorem}[Localization formula of index]\label{localization formula}
Suppose that $M$ is closed. 
Let $\cup_{i=1}^mO_i$ be an open neighborhood of $M\setminus V$ such that 
$O_i\cap V$ is $G$-invariant for each $i$ and $O_i\cap O_j=\emptyset$ if $i\neq j$. Then we have 
$$
\ind(M,W)=\sum_{i=1}^m\ind(O_i, O_i\cap V, W|_{O_i\cap V}). 
$$
\end{theorem}

\begin{remark}
The notion of acyclic compatible system 
arising from Hamiltonian torus action has 
close relation with existence of 
parallel sections of prequantizing line bundle on orbits. 
In fact the assumption of 
Proposition~\ref{Hamiltonian acyclic} implies that
there is no global 
non-zero parallel section on each orbit, 
and hence, it gives an acyclic compatible system. 
However the converse does not hold in general.
On the other hand 
the moment map image of a orbit 
with a global non-zero parallel section on it 
is a lattice point in the dual of the Lie algebra.
In this case it may be expected
that there would exist an acyclic compatible system
on the complement of the inverse image of lattice points.
If it were the case, we would have a localization
of the Riemann-Roch number as a sum of  contributions
from lattice points.
Proposition~\ref{Hamiltonian acyclic}, however,
does not realize this expectation.
Actually the complement of the inverse image of
the lattice points does not satisfy the assumption
of Proposition~\ref{Hamiltonian acyclic} in general.
We would need to delete neighborhoods of some extra points 
from the complement to make it satisfy the assumption.
In fact it is not hard to check that
the contributions from the extra points to the
Riemann-Roch number turn out to be zero, for instance,
in the case of toric action, which implies that
the localization to the lattice points is practically realized.
\end{remark}

\subsection{Technical comments for general case}\label{Technical comments}
Let $\{\pi_{\alpha}\}_{\alpha}$ be a compatible fibration 
on $M$. 
To have the full data in Section~\ref{new} using 
the compatible system on $\{\pi_{\alpha}\}_{\alpha}$ without torus action, 
we need several technical assumptions on it.  
The assumptions are satisfied in the case of 
locally toric Lagrangian fibrations considered in Section~\ref{four-dim ex}. 

When we construct a 
smooth function which is used to construct a completion of $M$ and a compatible fibration on it,  
it is convenient to use an operation $I:C^{\infty}(M)\to C^{\infty}(M)$ 
with the following properties. 

\begin{enumerate}
\item $I(f)$ is an admissible function (Definition~\ref{admissible function}) for all $f\in C^{\infty}(M)$. 
\item If $f$ is a constant function, then $I(f)$ is also a constant function.  
\item If $f$ is a non-negative function, then $I(f)$ is so.  
\item 
Let $f:M\to \R$ be a smooth function. 
If $supp f$ is contained in $V_{\alpha}$ for some $\alpha\in A$, 
then $supp I(f)$ is also contained in $V_{\alpha}$.
\end{enumerate}

We call such an operation an {\it averaging operation} 
for the compatible fibration  $\{\pi_{\alpha}\}_{\alpha}$. 
For the compatible fibration defined by local torus action, 
we have an averaging operation by using the averaging by the group action. 
If we have an averaging operation we can construct 
an admissible partition of unity from given partition of unity. 
When we construct an averaging operation it is convenient to use an open covering $\{V_{\alpha}'\}_{\alpha}$ with the following properties for all $\alpha, \beta\in A$. 

\begin{itemize}
\item $V_{\alpha}'\cap V_{\beta}=
\pi_{\alpha}^{-1}\pi_{\alpha}(V_{\alpha}'\cap V_{\beta})=\pi_{\beta}^{-1}\pi_{\beta}(V_{\alpha}'\cap V_{\beta})$.  
\item $\overline{V_{\alpha}'}\subset V_{\alpha}$. 
\end{itemize}

Under the assumption of the existence of such $\{V_{\alpha}'\}$, 
we can construct an averaging operation (with some more technical conditions) 
for a good compatible fibration. 
In fact if $\{\pi_{\alpha}\}_{\alpha\in A}$ is a good compatible fibration, 
then $A$ is equipped with a partial order with respect to the inclusion 
relation between fibers, and for a given smooth function $f$ 
we can average $f$ according to the order of $A$. 

\section{Product formula of local indices}\label{product formula of local indices}
In this section we formulate the product of 
acyclic compatible systems. 
Once we have an appropriate formulation of the product, 
then we obtain the product formula of local indices of 
the acyclic compatible systems by results in Section~3. 

\subsection{Product of compatible fibrations}
\label{product of compatible fibrations}

In this subsection we formulate a product of compatible fibrations.  
The product is defined for the following collection of data for $i=0,1$ 
which satisfy Assumption~\ref{assumption:bundle}. 

\begin{enumerate}
\item $M_i$ : a manifold. 
\item $V_i$ : an open set of $M_i$. 
\item $\{\pi_{i,\alpha}:V_{i,\alpha}\to U_{i,\alpha}\ | \ \alpha\in A_i\}$ : 
a compatible fibration on $V_i$. 
\item $K$ : a compact Lie group which acts smoothly on $M_1$. 
\item $\pi_P:P\to M_0$ : a principal $K$-bundle over $M_0$. 
\end{enumerate}


\begin{assumption}\label{assumption:bundle}
\begin{itemize}
\item[(1)] 
$V_1$ is $K$-invariant and 
the fibrations $V_{1,\alpha}\to U_{1,\alpha}$ 
are $K$-equivariant fiber bundles for all $\alpha\in A_1$.

\item[(2)]  
There exist principal $K$-bundles 
$P_{\alpha}\to U_{0,\alpha}$ and bundle maps $P|_{V_{0,\alpha}}\to P_{\alpha}$  for all $\alpha\in A_0$. 
%
\end{itemize}
\end{assumption}

For later convenience we take an open neighborhood $V_{i,\infty}$ 
of $M_i\setminus V_i$ and consider the trivial fiber bundle structure 
$\pi_{i,\infty}: V_{i,\infty}\to V_{i,\infty}$.
In other words we consider a compatible fibration 
$\{\pi_{i,\alpha}:V_{i,\alpha}\to U_{i,\alpha} \ | \ \alpha\in A_i\cup\{\infty\}\}$ on 
$M_i=V_{i,\infty}\cup(\cup_{\alpha}V_{i,\alpha})$. 
Let $M$ be 
the quotient manifold by the diagonal action of $K$ on $P\times M_1$. 
Then the natural map $\pi:M\to M_0$ is a 
fiber bundle whose fiber is equal to $M_1$. 
To define a structure of compatible fibration on $M$ 
we first prepare several notations for $i=0,1$. 
\begin{itemize}

\item $\tilde A_i:=A_i\cup \{\infty\}$. 
\item 
$\tilde A:=\tilde A_0\times \tilde A_1$. 
\item 
$A:=\tilde A\setminus (\infty,\infty)$. 
\item $V_{\alpha_0,\alpha_1}:=P|_{V_{0,\alpha_0}}\times_KV_{1,\alpha_1}$ 
for $\alpha_i\in \tilde A_i$. 
\item $U_{\alpha_0,\alpha_1}:=P_{\alpha_0}\times_KU_{1,\alpha_1}$ 
for $\alpha_i\in \tilde A_i$. 
\item
$V:=\bigcup_{(\alpha_0,\alpha_1)\in A}V_{\alpha_0,\alpha_1}$. 




\end{itemize}

Then we have the following from the construction of $V_{\alpha_0,\alpha_1}$ and $U_{\alpha_0,\alpha_1}$. 

\begin{proposition}
For $(\alpha_0,\alpha_1) \in A$ let $\pi_{\alpha_0,\alpha_1}:V_{\alpha_0,\alpha_1}\to U_{\alpha_0,\alpha_1}$ be the induced natural map. A collection of data 
$$
\{\pi_{\alpha_0,\alpha_1}:V_{\alpha_0,\alpha_1}\to U_{\alpha_0,\alpha_1} \ 
| \ (\alpha_0,\alpha_1) \in A\}
$$ is a compatible fibration on 
$V=\bigcup_{(\alpha_0,\alpha_1)\in A}V_{\alpha_0,\alpha_1}$.
\end{proposition}


\subsection{Product of acyclic compatible systems}
\label{product of acyclic compatible systems}

In this subsection we define a product of acyclic compatible systems. 
To define the product we consider the following data together with the 
data (1), (2), (3), (4) and (5) in Subsection~\ref{product of compatible fibrations}. 

\begin{itemize}
\item[(6)] a compatible Riemannian metric on $M_i$. 
\item[(7)] $W_i$ : a $Cl(TM_i)$-module bundle over $M_i$. 
\item[(8)] $\{D_{i,\alpha} \ | \ \alpha\in A_i \}$ : 
acyclic compatible system over $V_i=\cup_{\alpha\in A_i}V_{i,\alpha}$.
\end{itemize}

Together with Assumption~\ref{assumption:bundle} we impose the following assumption.

\begin{assumption}\label{assumption:G-equiv}
\begin{enumerate}
\item The metric on $M_1$ is $K$-invariant, and 
$W_1\to M_1$ and $\{D_{1,\alpha}\}$ are $K$-equivariant. 
\item For $i=0,1$ we assume the following. 
\begin{itemize}
\item A compact torus $G_i$ acts on $V_i$ in an isometric way. 
\item $G_i$-action lifts to $W_i|_{V_i}$ so that it commutes with the 
Clifford action. 
\item The compatible fibration $\{\pi_{i, \alpha}\}_{\alpha}$ is 
defined by $G_i$-action. 
\item $G_0$-action on $V_0$ lifts to $P|_{V_0}$. 
\item $K$-action on $V_1$ commutes with $G_1$-action. 
\end{itemize}
\end{enumerate}
\end{assumption} 

Though the condition (2) in Assumption~\ref{assumption:G-equiv} is not essential for us, we assume it for simplicity. 

From the Assumption~\ref{assumption:bundle}~(2). the restrictions of $P$ 
at each fibers of $\pi_{0,\alpha}$ are trivial. 
Moreover we have the following.

\begin{lemma}\label{connection on P}
There exists a connection on $P$ 
which is trivially flat over each fiber of $\pi_{0,\alpha}$ 
for all $\alpha\in A_0$. 
\end{lemma}
\begin{proof}
Take connections $\bar\nabla_\alpha$ 
for each $P_\alpha\to U_{0,\alpha}$. 
Let $\nabla_\alpha$ be the pull-back connections 
of them to $P|_{V_{0,\alpha}}\to V_{0,\alpha}$ by $\pi_{0,\alpha}$. 
Define a connection $\nabla$ on $P$ by patching $\{\nabla_\alpha\}_{\alpha}$ by 
an admissible partition of unity $\{\rho_\alpha^2\}_{\alpha}$, 
which satisfies the required property. \
\end{proof}

Using this connection on $P$ and the compatible metric on $M_0$ 
we have the metric on $P$, and hence the metric on $M$. 
Moreover since the connection is trivial along fibers of $\{\pi_{0,\alpha}\}$ 
it induces a family of connections of $\{P_{\alpha}\}$, 
and hence a family of metrics on them. 
Combining them with the $K$-invariant compatible metric on $M_1$ 
we have a family of metrics 
of $\{U_{\alpha_0\alpha_1}\}$ and so on. 
It defines a compatible metric of a compatible fibration 
$M=\cup_{(\alpha_0,\alpha_1)}V_{\alpha_0,\alpha_1}$.  
We put $\tilde W_0:=\pi^*W_0=\pi_P^*W_0\times_K M_1$, 
$\tilde W_1:=P\times_KW_1$ and 
$W:=\tilde W_0\otimes \tilde W_1$. 
Then $W\to M$ is a Clifford module bundle of 
$M=\cup_{(\alpha_0,\alpha_1)}V_{\alpha_0,\alpha_1}$ 
with respect to the above induced compatible metric. 
Now we define differential operators 
$\tilde D_{0,\alpha_0}$ and $\tilde D_{1,\alpha_1}$ 
for each $\alpha_0\in A_{0}$ and $\alpha_1\in A_1$ which act on 
$\Gamma(\tilde W_0|_{V_{\alpha_0,\alpha_1}})$ and 
$\Gamma(\tilde W_1|_{V_{\alpha_0,\alpha_1}})$ respectively. 
The operator $\tilde D_{1,\alpha_1}$ is the one induced from 
the $K$-equivariant operator $D_{1,\alpha_1}$ on 
$\Gamma(W_1|_{V_{1,\alpha_1}})$. 
On the other hand $\tilde D_{0,\alpha_0}$ is the operator defined as follows: 
Since $D_{0,\alpha_0}$ is a differential operator along fibers of 
$\pi_{0,\alpha_0}$ and 
$P$ is trivial at each fiber of $\pi_{0,\alpha_0}$,  
we can define the operator acting on the restriction 
$\Gamma(\pi^*_PW_0|_{\rm fiber}\times_G M_1)$ 
using $D_{0,\alpha_0}|_{\rm fiber}$ and 
a trivialization of $P|_{\rm fiber}$. 
Since such  operators along fibers do not depend on trivialization 
we have a differential operator $\tilde D_{0,\alpha_0}$ acting on 
$\Gamma(\tilde W_0|_{V_{\alpha_0,\alpha_1}})$. 
Using these operators we define an operator acting on $\Gamma(W|_V)$ 
by $D_{\alpha_0,\alpha_1}:=\tilde D_{0,\alpha_0}\otimes \id_{\tilde W_1}+
\epsilon_{\tilde W_0}\otimes\tilde D_{1,\alpha_1}$, 
where $\epsilon_{\tilde W_0}$ is a map on $\tilde W_0$ 
defined by $\epsilon_{\tilde W_0}(v):=(-1)^{\deg v}v$. 
For later convenience we put 
$\tilde D_{\alpha_0,\infty}=\tilde D_{\infty,\alpha_1}=0$. 

\begin{proposition}\label{acyclic system on the product}
A collection of differential operators 
$\{D_{\alpha_0,\alpha_1} \ | \ (\alpha_0,\alpha_1)\in A\}$ 
is an acyclic compatible system on $(V,W|_{V})$. 
\end{proposition}

\begin{proof}
Since $\tilde D_{0,\alpha_0}\otimes \id_{\tilde W_1}$ and 
$\epsilon_{\tilde W_{0}}\otimes \tilde D_{1,\alpha_1}$ anti-commute each other 
we have 
$$
D_{\alpha_0,\alpha_1}^2
=\tilde D_{0,\alpha_0}^2
\otimes \id_{\tilde W_1} 
+\id_{\tilde W_0}\otimes\tilde D_{1,\alpha_1}^2 
$$
and the equality among anti-commutators 
$$
\{D_{\alpha_0,\alpha_1}, D_{\alpha_0',\alpha_1'}\}=
\{\tilde D_{0,\alpha_0}, \tilde D_{0,\alpha_0'}\}\otimes\id_{\tilde W_1} +
 \id_{\tilde W_0}\otimes \{\tilde D_{1,\alpha_1}, \tilde D_{1,\alpha_1'}\}. 
$$
These equalities imply that if 
$\{D_{i,\alpha_i}\}_{\alpha_i}$ are acyclic compatible systems, 
then $\{D_{\alpha_0,\alpha_1}\}_{(\alpha_0,\alpha_1)}$ is so. \
\end{proof}

\subsection{Product formula of local indices}

In this subsection we consider the setting in Subsection~\ref{product of compatible fibrations} and \ref{product of acyclic compatible systems}.

To apply the construction in Section~\ref{Main theorem}, 
we deform the end of $M_i$ and $M$ as in the following way: 
we deform the end $V_i$ of $M_i$ into $\hat M_i$ 
so that $\hat M_i$ has the cylindrical end, 
and, we can assume that the acyclic compatible systems and admissible partitions of unity on $\hat M_i$ is translationally invariant. 
We can also assume that the deformation $\hat M_1$ 
for $M_1$ is $G$-equivariant. 
In addition we can deform $P$ 
into $\hat P$ together with its connection so that it is 
translationally invariant on the end.  
Then the metric on the product $\hat M=\hat P\times _G\hat M_1$ is complete. 
Hereafter we denote the completion $\hat M_i$ (resp. $\hat M$, $\hat P$) 
by the original letter $M_i$ (resp. $M$, $P$) for simplicity.

Let $\{\rho_{i,\alpha}^2\}_{\alpha\in\tilde A_i}$ 
be admissible partitions of unity of $M_i$. 
We can assume that $\{\rho_{1,\alpha}^2\}$ is $K$-invariant. 
Using these partitions of unity we have an admissible partition of unity 
$\{\rho_{\alpha_0,\alpha_1}^2\}_{(\alpha_0,\alpha_1)\in\tilde A}$ 
on $M=\cup_{(\alpha_0,\alpha_1)}V_{\alpha_0,\alpha_1}$ which is defined by 
$\rho_{\alpha_0,\alpha_1}([u,y]):=
\rho_{0,\alpha_0}(\pi(u))\rho_{1,\alpha_1}(y)$ for $[u,y]\in M$. 

For any translationally invariant Dirac-type operators $D_i$ on $\Gamma(W_i)$, 
using a local trivialization of $P$ 
we have their lifts $\tilde D_0\otimes \id_{\tilde W_1}$ and  
$\epsilon_{\tilde W_0}\otimes \tilde D_1$ on $\Gamma(W)$ as in Subsection~3.6. 
Note that $D:=\tilde D_0\otimes \id_{\tilde W_1}+\epsilon_{\tilde W_0}\otimes \tilde D_1$ 
is a translationally invariant Dirac-type operator on $\Gamma(W)$. 

Because of Lemma~\ref{acyclic system on the product} 
if we take a positive number $t$ large enough, then 
the inequality in Proposition~\ref{KeyProp-local} 
holds for deformed operators $D_{i,t}$ and $D_t$ on $M_i$ and $M$. 
On the other hand we have a decomposition $D_t=D_t^B+D_t^F$, where 
\begin{eqnarray*}
D_t^B&:=&\left(\tilde D_0+t\sum_{\alpha_0}
\left(\sum_{\alpha_1}\rho_{\alpha_1}^2\right)
\pi^*\rho_{0,\alpha_0}\tilde D_{0,\alpha_0}\pi^*\rho_{0,\alpha_0}\right)
\otimes \id_{\tilde W_1} \\
&=&\left(\tilde D_0+t\sum_{\alpha_0}
\pi^*\rho_{0,\alpha_0}\tilde D_{0,\alpha_0}\pi^*\rho_{0,\alpha_0}\right)
\otimes \id_{\tilde W_1}  \\ 
D_t^F&:=&\epsilon_{\tilde W_0}\otimes\left(\tilde D_1+t\sum_{\alpha_1}
\left(\sum_{\alpha_0}\pi^*\rho_{\alpha_0}^2\right)
\rho_{1,\alpha_1}\tilde D_{1,\alpha_1}\rho_{1,\alpha_1}\right)\\ 
&=&\epsilon_{\tilde W_0}\otimes\left(\tilde D_1+t\sum_{\alpha_1}
\rho_{1,\alpha_1}\tilde D_{1,\alpha_1}\rho_{1,\alpha_1}\right). 
\end{eqnarray*}
Note that $D_t^F$ is a differential operator along fibers of $\pi:M\to M_0$. 
They anti-commutes each other. Namely,   
\begin{lemma}\label{anticommutator}
$D_t^BD_t^F+D_t^FD_t^B=0$.
\end{lemma}

Moreover we have the following as in the same in 
Theorem~\ref{acyclic->assumption}.  
\begin{lemma}\label{product assumption}
If  $t$ is large enough, then $D_t$ and $D_t^F$ 
satisfy Assumption~\ref{assumption product}. 
\end{lemma}

By Lemma~\ref{product assumption} we can define the local indices 
$\ind(M_1,V_1,W_1)$ and $\ind(M,V,W)$. 
Note that since all the data for $\ind(M_1,V_1,W_1)$ 
is $K$-equivariant the index can be 
defined as the $K$-equivariant index $\ind_K(M_1,V_1,W_1)\in R(K)$. 
When we write 
$\ker D_{1,t}
=E^0\oplus E^1$
as the $K$-equivariant $\Z/2$-graded vector space, 
$K$-equivariant local index can be 
written as $\ind_K(M_1,V_1,W_1)=[E^0]-[E^1]\in R(K)$. 
Let $\underline{E}^i$ be the vector bundle over 
$M_0$ defined by $\underline{E}^i=P\times_KE^i$. 
Then the acyclic compatible system on $(M_0,V_0,W_0)$ 
induces another acyclic compatible systems on $(M_0,V_0,W_0\otimes \underline{E}^i)$ via 
$\{D_{0,\alpha}\otimes\id_{E^i}\}$ for $i=0,1$. 
Lemma~\ref{anticommutator}, Lemma~\ref{product assumption} and 
the product formula in Section~3 imply the following 
product formula of local indices. 

\begin{theorem}\label{product}
We have the following product formula. 
$$
\ind(M_0,V_0,W_0\otimes \underline{E}^0)-
\ind(M_0,V_0,W_0\otimes \underline{E}^1)
=\ind(M,V,W)\in\Z. 
$$
\end{theorem}

\section{Four-dimensional case}\label{four-dim ex}
In this section we apply the localization formula (Theorem~\ref{localization formula}), the product formula (Theorem~\ref{product}), and Theorem~\ref{calculation} to show that for a four-dimensional closed locally toric Lagrangian fibration the Riemann-Roch number is equal to the number of Bohr-Sommerfeld fibers (Theorem~\ref{RR=BSforlocallytoricLagrangianfibrations}).  
\subsection{Local indices for elliptic singularities}
A critical point of a $2n$-dimensional singular Lagrangian fibration $\mu \colon (M,\omega)\to B$ is called a {\it nondegenerate elliptic singular point of rank $k\ (\le n)$} if there exists a symplectic coordinates $x_1,\ldots ,x_n,y_1,\ldots ,y_n$ such that in these coordinate, $\mu$ is written as $\mu=(x_1, \ldots, x_k,x_{k+1}^2+y_{k+1}^2,\ldots,x_n^2+y_n^2)$. See \cite{Williamson,Symington,Hamilton}. In this subsection we calculate local indices for elliptic singularities in four-dimensional case. 

\subsubsection{Definition of $RR_0(a_1,a_2)$}\label{singularBSinD}
Let $D:=\{ z\in \C\mid \abs{z}<1 \}$ be the unit open disc in $\C$. Let $X_0$ be the product of two copies of $D$ with symplectic structure 
\[
\omega_0:=\frac{\sqrt{-1}}{2\pi}\sum_{k=1}^2dz_k\wedge d\overline{z}_k,
\]
and $(L_0,\nabla^{L_0})$ a prequantizing line bundle on $(X_0,\omega_0)$. 

Let us consider the structure of a singular Lagrangian fibration $\mu_0 \colon (X_0,\omega_0 )\to [0,1)\times [0,1)$ on $X_0$ which is defined by 
\[
\mu_0(z):=\left( \abs{z_1}^2,\abs{z_2}^2\right) . 
\]
We put the following assumption. 
\begin{assumption}\label{assumption0}
The cohomology groups $H^*\left( \mu^{-1}_0(b);(L_0,\nabla^{L_0})|_{\mu_0^{-1}(b)}\right)$ vanish for all points $b\in [0,1)\times [0,1)$ except for $b=(0,0)$. 
\end{assumption}

Let $a_1$ and $a_2\in \Z$ be arbitrary integers. We define a good compatible fibration on $X_0\setminus \{ (0,0)\}$ consisting of three quotient maps of the torus actions by 
\[
\begin{split}
&\pi^0_0\colon V^0_0:=X_0\cap \left( \C^*\times \C^*\right)\to U^0_0:=V^0_0/T^2 ,\\
&\pi^0_1\colon V^0_1:=\{ (z_1,z_2)\in X_0\mid \abs{z_1}>\abs{z_2}, \abs{z_1}^2+a_1\abs{z_2}^2\notin \Z\}\to U^0_1:=V^0_1/S^1 ,\\
&\pi^0_2\colon V^0_2:=\{ (z_1,z_2)\in X_0\mid \abs{z_1}<\abs{z_2}, a_2\abs{z_1}^2+\abs{z_2}^2\notin \Z\}\to U^0_2:=V^0_2/S^1 ,  
\end{split}
\]
where the $T^2$-action on $V^0_0$ is the standard one, the $S^1$-action on $V^0_1$ is defined by
\begin{equation}\label{action1}
t(z_1,z_2):=(tz_1,t^{a_1}z_2),
\end{equation}
and the $S^1$-action on $V^0_2$ is defined by 
\begin{equation}\label{action2}
t(z_1,z_2):=(t^{a_2}z_1,tz_2). 
\end{equation}

We take and fix an arbitrary Hermitian structure $(g_0,J_0)$ invariant under the standard $T^2$-action on $X_0$ and compatible with $\omega_0$. Since $g_0$ is $T^2$-invariant $g_0$ induces a compatible Riemannian metric of this compatible fibration. 

Let $W_0$ be the Hermitian vector bundle on $X_0$ which is defined by
\[
W_0:=\wedge_\C^\bullet (TX_0)_\C\otimes L_0. 
\]
$W_0$ is a $\Z_2$-graded Clifford module bundle with respect to the Clifford module structure \eqref{Clifford}. 
We take a compatible system $\{ D_i\}_{i=0,1,2}$ to be the family of de Rham operators along fibers of $\pi^0_i$ ($i=0, 1 ,2$) which is defined by the same way as in Subsection~\ref{DH}. By Assumption~\ref{assumption0} and the definition of $V^0_i$ ($i=1, 2$) the kernel of all $D_i$ vanish. We should notice that for $z\in V^0_i$ the cohomology groups $H^*\left(\left(\pi^0_i\right)^{-1}\left(\pi^0_i(z)\right);\left(L_0,\nabla^{L_0}\right)|_{\left(\pi^0_i\right)^{-1}\left(\pi^0_i(z)\right)}\right)$ vanish if and only if $z$ satisfies the equality $\abs{z_1}^2+a_1\abs{z_2}^2\notin \Z$ for $i=1$ and $a_2\abs{z_1}^2+\abs{z_2}^2\notin \Z$ for $i=2$. Hence $\{D_i\}$ is acyclic.

\begin{definition}
We define $RR_0(a_1,a_2)$ to be the local index in the sense of Definition~\ref{def. of local the index} for the above data and any Dirac-type operator on $W_0$. 
\end{definition}
\begin{remark}
$RR_0(a_1,a_2)$ does not depend on the choice of a compatible Hermitian structure $(g_0,J_0)$ since it is deformation invariant. 
\end{remark}

\subsubsection{Definition of $RR_1(a_+,a_-)$}\label{singularnonBS}
Let $X_1:=(0,1)\times S^1\times D$ be the product of $(0,1)\times S^1$ and $D$ with symplectic structure
\[
\omega_1:=dr\wedge d\theta +\frac{\sqrt{-1}}{2\pi}dz\wedge d\bar{z} 
\]
for $(r,e^{2\pi\sqrt{-1}\theta},z)\in X_1$, and $(L_1,\nabla^{L_1})$ a prequantizing line bundle on $(X_1,\omega_1)$. 

Let us consider the structure of singular Lagrangian fibration $\mu_1\colon (X_1,\omega_1)\to (0,1)\times [0,1)$ which is defined by
\[
\mu_1(r,u,z):=\left( r,\abs{z}^2\right). 
\]
We put the following assumption. 
\begin{assumption}\label{assumption1}
For all points $b\in (0,1)\times [0,1)$, $H^*(\mu^{-1}_1(b);(L_1,\nabla^{L_1})|_{\mu_1^{-1}(b)})$ vanish. 
\end{assumption}

Let $a_+$ and $a_-\in \Z$ be arbitrary integers. We take an element $r_1\in (0,1)$ and fix it. Then, we define a good compatible fibration on $X_1\setminus \mu_1^{-1}(r_1,0)$ consisting of three quotient maps of the torus actions by 
\[
\begin{split}
&\pi^1_0\colon V^1_0:=(0,1)\times S^1\times \left( D\setminus \{0\}\right)\to U^1_0:=V^1_0/T^2 ,\\
&\pi^1_1\colon V^1_1:=\{(r,u,z)\in (r_1,1)\times S^1\times D\mid r+a_+\abs{z}^2\notin\Z\} \to U^1_1:=V^1_1/S^1,\\
&\pi^1_2\colon V^1_2:=\{(r,u,z)\in (0,r_1)\times S^1\times D\mid r+a_-\abs{z}^2\notin \Z\} \to U^1_2:=V^1_2/S^1 , 
\end{split}
\]
where the $T^2$-action on $V^1_0$ is defined by
\[
t(r,u,z):=(r,t_1u,t_2z),
\]
the $S^1$-action on $V^1_1$ is defined by
\[
t(r,u,z):=(r,tu,t^{a_+}z),
\]
and the $S^1$-action on $V^1_2$ is defined by 
\[
t(r,u,z):=(r,tu,t^{a_-}z). 
\]

We take an arbitrary Hermitian structure $(g_1,J_1)$ which is invariant under the standard $T^2$-action on $X_1$ and compatible with $\omega_1$ and fix it. We define the $\Z_2$-graded Clifford module bundle $W_1$ and the acyclic compatible system in the same way as in Section~\ref{singularBSinD}. 
\begin{definition}
We define $RR_1(a_+,a_-)$ to be the local index in the sense of Definition~\ref{def. of local the index} for the above data and any Dirac-type operator on $W_1$. 
\end{definition}
\begin{remark}
$RR_1(a_+,a_-)$ does not depend on the choice of a compatible Hermitian structure $(g_1,J_1)$ since it is deformation invariant. 
\end{remark}

\subsubsection{Computation}
First we can show the following lemma. 
\begin{lemma}\label{relation1}
For integers $a, b, c\in \Z$ we have
\[
RR_0(a,b)=RR_0(b,a),\quad
RR_1(a,b)=RR_1(a+c,b+c).
\]
\end{lemma}
\begin{proof}
We prove the latter equation. The proof of the former equation is similar.  Let $\varphi \colon (0,1)\times S^1\times D\to (0,1)\times S^1\times D$ be the diffeomorphism which is defined by
\[
\varphi (r,u,z)=(r,u,u^cz). 
\]
On the target space of $\varphi$ we consider the same compatible fibration as $\{ \pi^1_i\}_{i=0, 1 ,2}$ except that $a$ and $b$ are replaced by $a+c$ and $b+c$, respectively. Then $\varphi$ induces an isomorphism between compatible fibrations. 

As the other data on the target space of $\varphi$ we consider the data which are induced from those on the source space by $\varphi^{-1}$. Then the local index for the induced data on the target space is nothing but $RR_1(a,b)$. 

On the other hand, the data $(\varphi^{-1})^*\omega_1$ and $(\varphi^{-1})^*\nabla^{L_1}$ can be deformed to $\omega_1$ and $\nabla^{L_1}$ by linear deformations. Since the local index is invariant under continuous deformation this implies that the latter equation. \
\end{proof}

Moreover, we can also show the following lemma by Theorem~\ref{excision1}. 
\begin{lemma}\label{relation2}
\[
RR_0(a,b) = RR_0(a',b) + RR_1(a',a),
\qquad
RR_1(a,c) = RR_1(a,b) + RR_1(b,c).
\]
\end{lemma}
Then we can calculate $RR_0(a_1,a_2)$ and $RR_1(a_+,a_-)$. 
\begin{theorem}\label{calculation}
$$RR_0(a_1,a_2)=1,\qquad RR_1(a_+,a_-)=0.$$
\end{theorem}
\begin{proof}
We show $RR_0(0,1)=1$ and $RR_0(0,0)=1$. Then the theorem follows from these equalities and Lemma~\ref{relation1} and ~\ref{relation2}. 

First we show $RR_0(0,1)=1$. Let us consider the standard toric action on $\CP^2$ with hyperplane bundle as a prequantizing line bundle. We adopt the moment map $\mu$ of this action as a singular Lagrangian fibration. The image $B$ of $\mu$ is the triangle in $\R^2$ with vertices $(0,0)$, $(1,0)$, $(0,1)$, and $\mu$ has three Bohr-Sommerfeld fibers which corresponds one-to-one to three fixed points $[1:0:0],[0:1:0],[0:0:1]$ of the toric action. 

We construct a compatible fibration on $\C P^2\setminus \{[1:0:0],[0:1:0],[0:0:1]\}$. For each $k \in \Z/3$ let $V_k$ be a pairwise disjoint $T^2$-invariant open neighborhood of $\{[z_0:z_1:z_2]\in \C P^2\mid z_k=0\}\setminus \{[1:0:0],[0:1:0],[0:0:1]\}$, and $G_k$ the stabilizer of $\{[z_0:z_1:z_2]\in \C P^2\mid z_k=0\}\setminus \{[1:0:0],[0:1:0],[0:0:1]\}$. Each $G_k$ is a circle subgroup in $T^2$ and $G_{k-1}$ acts on $V_k$ freely. Then we put $U_k:=V_k/G_{k-1}$ and define $\pi_k\colon V_k\to U_k$ to be the quotient map. We also put $V_4:=U_4:=B\setminus \partial B$ and define $\pi_4\colon V_4\to U_4$ to be the identity map. These data define a good compatible fibration on $\C P^2\setminus \{[1:0:0],[0:1:0],[0:0:1]\}$. 

The $\Z_2$-graded Clifford module bundle and the acyclic compatible system are defined by the same way as in Section~\ref{singularBSinD}. 

Then by Theorem~\ref{localization formula} the Riemann-Roch number is localized at $[1:0:0],[0:1:0],[0:0:1]$, and the contribution of each fixed point is equal to $RR_0(0,1)$. 

On the other hand it is well-known that the Riemann-Roch number of $\C P^2$ is $3$. Thus we obtain $RR_0(0,1)=1$. 

Next we show $RR_0(0,0)=1$. It is a direct consequence of the product formula Theorem~\ref{product} and the fact $[D^+]=1$(see \cite[Theorem~6.7]{Fujita-Furuta-Yoshida}). 

We can also show $RR_0(0,0)=1$ in the following way. We consider $\C P^1 \times \C P^1$ with standard toric action. The image of the moment map is a square. By the similar construction as above the Riemann-Roch number is localized at four vertices and the contribution of any vertex is $RR_0(0,0)$. On the other hand the Riemann-Roch number of $\C P^1 \times \C P^1$ is four. This implies $RR_0(0,0)=1$. \
\end{proof}

\subsection{Application to locally toric Lagrangian fibrations}

\subsubsection{Locally toric Lagrangian fibrations}
Let $\omega_{\C^n}$ be the standard symplectic structure on $\C^n$ 
\begin{equation*}\label{om_C}
\omega_{\C^n}:=\frac{\sqrt{-1}}{2\pi}\sum_{k=1}^ndz_k\wedge d\overline{z}_k.
\end{equation*}
The standard action of $T^n$ on $\C^n$ preserves $\omega_{\C^n}$ and the map $\mu_{\C^n}\colon \C^n\to \R^n$ which is defined by 
\begin{equation*}\label{std-moment}
\mu_{\C^n}(z):=\left(\abs{z_1}^2, \ldots ,\abs{z_n}^2\right)
\end{equation*}
for $z=(z_1, \ldots ,z_n)\in \C^n$ is a moment map of the standard $T^n$-action. Note that the image of $\mu_{\C^n}$ is the $n$-dimensional standard positive cone 
\begin{equation*}\label{positive-cone}
\R^n_+:=\{ r =(r_1, \ldots ,r_n)\in \R^n \colon r_i\ge 0\ i=1, \ldots ,n\} . 
\end{equation*}

Let $(M,\omega)$ be a $2n$-dimensional symplectic manifold and $B$ an $n$-dimensional manifold with corners. 
\begin{definition}[\cite{Hamilton,Yoshida}]\label{locallytoricLagrangianfibration}
A map $\mu \colon (M, \omega )\to B$ is called a {\it locally toric Lagrangian fibration} if there exists a system $\{ (U_{\alpha},\varphi^B_{\alpha})\}$ of coordinate neighborhoods of $B$ modeled on $\R^n_+$, and for each $\alpha$ there exists a symplectomorphism $\varphi^M_{\alpha}\colon (\mu^{-1}(U_{\alpha}), \omega )\to (\mu_{\C^n}^{-1}(\varphi^B_{\alpha}(U_{\alpha})), \omega_{\C^n})$ such that $\mu_{\C^n}\circ \varphi^M_{\alpha}=\varphi^B_{\alpha}\circ \mu$. 
\end{definition}

Note that a locally toric Lagrangian fibration is a singular Lagrangian fibration that allows only elliptic singularities. 

By the definition of a manifold with corners, $B$ is equipped with a natural stratification. We denote by ${\mathcal S}^{(k)}B$ the $k$-dimensional part of $B$, namely, ${\mathcal S}^{(k)}B$ consists of those points which have exactly $k$ nonzero components in a local coordinate system. Then, it is easy to see that the fiber of $\mu$ at a point in ${\mathcal S}^{(k)}B$ is a $k$-dimensional torus. In particular, all fibers of $\mu$ are smooth.  

\begin{example}[Projective toric variety]
The moment map of a nonsingular projective toric variety is a locally toric Lagrangian fibration. 
\end{example}
\begin{example}[Non toric example]\label{nontoricex}
Let $c\in \N$ be a positive integer. We consider the diagonal Hamiltonian $S^1$-action on $(\C^2,\omega_{\C^2})$ with moment map 
\[
\Phi(z):=\norm{z}^2-c. 
\]
It is well-known that the symplectic quotient $\left(\Phi^{-1}(0),\omega_{\C^2}|_{\Phi^{-1}(0)}\right)/S^1$ is $\C P^1$ with $c$ times Fubini-Study form $\omega_{FS}$. In the rest of this example we identify $(\C P^1,c\omega_{FS})$ with $\left(\Phi^{-1}(0),\omega_{\C^2}|_{\Phi^{-1}(0)}\right)/S^1$. 

Let $\tilde{\mu}\colon (\tilde{M},\tilde{\omega} )\to \tilde{B}$ be the singular Lagrangian fibration which is defined by
\[
\begin{split}
&(\tilde{M},\tilde{\omega} ):= (\R\times S^1\times \C P^1, dr \wedge d\theta \oplus c\omega_{FS}),\\
&\tilde{B}:=\R\times [0,c],\\
&\tilde{\mu}(r,u,[z_0:z_1]):=(r,\abs{z_1}^2), 
\end{split}
\]
where we use the coordinate $(r,e^{2\pi \sqrt{-1}\theta})\in \R \times S^1$. For a negative integer $a\in \Z$ ($a<0$) and a positive integer $b\in \N$, we define the $\Z$-actions on $\tilde{M}$ and $\tilde{B}$ by
\begin{align}
&n(r,u,[z_0:z_1]):=\left( r+n(-a\abs{z_1}^2+b),u,[z_0:u^{na}z_1]\right)\label{ZonM},\\
&n(r_1,r_2):=(r_1+n(-ar_2+b), r_2)\label{ZonB}. 
\end{align}
It is easy to see that \eqref{ZonM} and \eqref{ZonB} are free $\Z$-actions and \eqref{ZonM} preserves $\tilde{\omega}$. Then we put 
\[
\begin{split}
&(M,\omega):= (\tilde{M},\tilde{\omega} )/\Z ,\\
&B:=\tilde{B}/\Z . 
\end{split}
\]
It is also easy to see that $\tilde{\mu}$ is equivariant with respect to \eqref{ZonM} and \eqref{ZonB}. Hence $\tilde{\mu}$ induces the map from $M$ to $B$ which we denote by $\mu\colon (M,\omega)\to B$. By construction, $B$ is a cylinder and $\mu$ is a locally toric Lagrangian fibration which has singular fibers on $\partial B$. 
\end{example}

Let $\mu \colon (M^{2n}, \omega )\to B$ be a locally toric Lagrangian fibration. By definition, for each $\alpha$ there is a symplectomorphism $\varphi^M_\alpha\colon \mu^{-1}(U_\alpha)\to \mu_{\C^n}^{-1}(\varphi^B_\alpha(U_\alpha))$, and $\mu_{\C^n}^{-1}(\varphi^B_\alpha(U_\alpha))$ has a $T^n$-action which is obtained by restricting the standard $T^n$-action on $\C^n$. Then, it is known by \cite[Proposition~3.13]{Yoshida} that on each nonempty overlap $U_\alpha\cap U_\beta$ there exists an automorphism $\rho_{\alpha\beta}\in \Aut (T^n)$ of $T^n$ such that $\varphi^M_{\alpha\beta}:=\varphi^M_\alpha \circ (\varphi^M_\beta)^{-1}$ is $\rho_{\alpha \beta}$-equivariant, namely, 
\[
\varphi^M_{\alpha\beta}(tx)=\rho_{\alpha\beta}(t)\varphi^M_{\alpha\beta}(x)
\]
for $t\in T^2$ and $x\in \mu_{\C^n}^{-1}(\varphi^B_\beta(U_\alpha\cap U_\beta))$. Moreover,we can show that $\rho_{\alpha\beta}$'s form a \v{C}ech one-cocycle $\{ \rho_{\alpha\beta}\}$ on $\{U_\alpha\}$ with coefficients in $\Aut(T^n)$. Hence it defines an element $[\{ \rho_{\alpha\beta}\}]$ in the \v{C}ech cohomology $H^1(B;\Aut (T^n))$. Then we have the following lemma. 
\begin{lemma}[\cite{Yoshida}]\label{obstruction}
The \v{C}ech cohomology class  $[\{ \rho_{\alpha\beta}\}]$ is the obstruction class in order that the $T^n$-actions on $\mu_{\C^n}^{-1}(\varphi^B_\alpha(U_\alpha))$ for all $\alpha$ can be patched together to obtain a global $T^n$-action on $M$. 
\end{lemma}
For more detail see \cite{Yoshida}. 

Let $q_B\colon \tilde{B}\to B$ be the universal covering of $B$. Since the \v{C}ech cohomology $H^1(B;\Aut (T^n))$ is identified with the moduli space of representations of the fundamental group $\pi_1(B)$ of $B$ to $\Aut (T^n)$, the fiber product $q_B^*M:=\{ (\tilde{b},x)\in \tilde{B}\times M\mid q_B(\tilde{b})=\mu(x)\}$ admits a $T^n$-action. 

We take a representative $\rho \colon \pi_1(B)\to \Aut(T^n)$ of the equivalence class of representations corresponding to $[\{ \rho_{\alpha\beta}\}]$. Then the $T^n$-action on $q_B^*M$ can be written explicitly. See \cite[Lemma~3.1]{Yoshida2} for the explicit description. 

On the other hand, by the construction, $\pi_1(B)$ acts on $q_B^*M$ from the left by the inverse of the deck transformation, and it is shown that the $T^n$-action and the $\pi_1(B)$-action satisfy the following relationship
\begin{equation}\label{semidirect}
t(a\tilde{x})=a\left(\rho(a^{-1})(t)\tilde{x}\right)
\end{equation}
for $t\in T^n$, $a\in \pi_1(B)$, and $\tilde{x}\in q_B^*M$. Let $T^n\rtimes_\rho \pi_1(B)$ be the semidirect product of $T^n$ and $\pi_1(B)$ with respect to $\rho$. Then, \eqref{semidirect} implies that these actions form an action of $T^n\rtimes_\rho \pi_1(B)$ on $q_B^*M$. For more details see \cite{Yoshida2}. 

Let $q_M\colon q_B^*M\to M$ be the natural projection. Note that $q_M^*\omega$ is $T^n\rtimes_\rho\pi_1(B)$-invariant since $\omega$ is invariant under the $T^n$-action on $\mu^{-1}(U_\alpha)$ induced by the standard $T^n$-action on $\C^n$ for each $\alpha$. Now we show the following lemma. 
\begin{lemma}\label{invariantHermitianstructureonpi*M}
There exists a Hermitian structure $(\tilde{g},\tilde{J})$ on $q_B^*M$ compatible with $q_M^*\omega$ which is invariant under the action of $T^n\rtimes_\rho \pi_1(B)$. 
\end{lemma}
\begin{proof}
It is sufficient to show that the existence of an invariant Riemannian metric. Let $g'$ be a Riemannian metric on $M$. We define the Riemannian metric $\tilde{g}$ on $q_B^*M$ by 
\[
\tilde{g}_{\tilde{x}}(u,v):=\int_{T^n}\left(\varphi^*_t(q_M^*g')\right)_{\tilde{x}}(u,v)dt, 
\]
where $\varphi_t$ implies the $T^n$-action for $t\in T^n$. It is sufficient to show that $\tilde{g}$ is $\pi_1(B)$-invariant. For $a \in \pi_1(B)$ we denote the $\pi_1(B)$-action by $\phi_a$. Then we have
\begin{align}
(\phi_a^*\tilde{g})_{\tilde{x}}(u,v)&=\int_{T^n}\left( \phi_a^*\left( \varphi_t^*(q_M^*g')\right)\right)_{\tilde{x}}(u,v)dt\nonumber \\
&=\int_{T^n}\left( \varphi_{\rho(a^{-1})(t)}^*\left( \phi_a^*(q_M^*g')\right)\right)_{\tilde{x}}(u,v)dt\nonumber \\
&=\int_{T^n}\left( \varphi_{\rho(a^{-1})(t)}^*(q_M^*g')\right)_{\tilde{x}}(u,v)dt\nonumber\\
&=\det\rho(a^{-1}) \int_{T^n}\left( \varphi_{\rho(a^{-1})(t)}^*(q_M^*g')\right)_{\tilde{x}}(u,v)\rho(a^{-1})^*dt\nonumber \\
&=\int_{T^n}\left( \varphi_t^*(q_M^*g')\right)_{\tilde{x}}(u,v)dt. \nonumber \\
&=\tilde{g}_{\tilde{x}}(u,v). \nonumber 
\end{align}
Here we remark that $\det \rho(a^{-1})=\pm 1$ since $\rho(a^{-1})\in \Aut (T^n)$. \
\end{proof}

\begin{corollary}[the existence of an invariant Hermitian structure]\label{invariantHermitianstructure}
There exists a Hermitian structure $(g,J)$ on $M$ compatible with $\omega$ such that on each $\mu^{-1}(U_\alpha)$ $(g,J)$ is invariant under the $T^n$-action on $\mu^{-1}(U_\alpha)$ which is induced from the $T^n$-action on $\mu_{\C^n}^{-1}(\varphi^B_\alpha(U_\alpha))$ with the identification $\varphi^M_\alpha$. 
\end{corollary}
\begin{proof}
By Lemma~\ref{invariantHermitianstructureonpi*M} there is a $T^n\rtimes_\rho \pi_1(B)$-invariant Hermitian structure $(\tilde{g},\tilde{J})$ on $q_B^*M$ compatible with $q_M^*\omega$. In particular, since $(\tilde{g},\tilde{J})$ is $\pi_1(B)$-invariant, $(\tilde{g},\tilde{J})$ induces an $\omega$-compatible Hermitian structure on $M$ which is denoted by $(g,J)$. Then, $(g,J)$ is the required one. \
\end{proof}

\begin{lemma}[The existence of an averaging operation]\label{averagingoperationonM}
Suppose that there exists a compatible fibration $\{ \pi_\alpha\colon V_\alpha \to U_\alpha\}$ on $M$ such that for each $\alpha$ a fiber of $\pi_\alpha$ is contained in that of $\mu$, namely, $\pi_\alpha^{-1}\pi_\alpha(x)\subset \mu^{-1}\mu(x)$ for $x\in V_\alpha$. There exists an averaging operation $I\colon C^\infty(M)\to C^\infty(M)$ with respect to $\{ \pi_\alpha\colon V_\alpha \to U_\alpha\}$. 
\end{lemma}
\begin{proof}
For $f\in C^\infty(M)$ let $\tilde{f}\in C^\infty(q_B^*M)$ be the function on $q_B^*M$ which is defined by 
\[
\tilde{f}(\tilde{x}):=\int_{T^n}(f\circ q_M)(t\tilde{x})dt. 
\]
Then, by the similar way to that in the proof of Lemma~\ref{invariantHermitianstructureonpi*M}, we can show that $\tilde{f}$ is $T^n\rtimes_\rho \pi_1(B)$-invariant. Hence it descends to the function on $M$. We denote it by $I(f)$. Then, it is clear that $I(f)$ satisfies the properties in Subsection~\ref{Technical comments}. \
\end{proof}

\subsubsection{Bohr-Sommerfeld fibers and the Riemann-Roch number}
Let $\mu \colon (M, \omega )\to B$ be a prequantized locally toric Lagrangian fibration with prequantizing line bundle $(L,\nabla)$. Recall that, as described above, all fibers are smooth. 

\begin{definition}\label{BS}
A fiber $F$ of $\mu$ is said to be {\it Bohr-Sommerfeld} if the restriction $(L,\nabla)|_F$ is trivially flat. A point $b$ of $B$ is also said to be {\it Bohr-Sommerfeld} if the fiber $\mu^{-1}(b)$ is Bohr-Sommerfeld. 
\end{definition}
\begin{remark}
A fiber $F$ of $\mu$ is Bohr-Sommerfeld if and only if the cohomology $H^*(F;(L,\nabla)|_F)$ does not vanish, see Lemma~\ref{vanishing of cohomology}. This is also equivalent to the condition that the de Rham operator on $F$ with coefficients in $(L,\nabla )|_F$ has non zero kernel. 
\end{remark}

First we specify Bohr-Sommerfeld points for the local model. 
\begin{proposition}
Let $(L,\nabla)$ be a prequantizing line bundle on $(\C^n,\omega_{\C^n})$. Then, a point $b\in \R^n_+$ is Bohr-Sommerfeld if and only if $b\in \R^n_+\cap \Z^n$.  
\end{proposition}
\begin{proof}
Since $\C^n$ is contractible $L$ is trivial as a complex line bundle. Then we can assume that $L$ is of the form $L=\C^n\times \C$ without loss of generality. Then, $\nabla$ can be written as 
\[
\nabla =d-2\pi \sqrt{-1}A
\]
for some one form on $\C^n$ with $dA=\omega_{\C^n}$. Moreover $A$ is unique up to exact one form since $\C^n$ is contractible. In particular, $A$ is of the form 
\[
A=\dfrac{\sqrt{-1}}{4\pi}\sum_{i=1}^n(z_id\bar{z}_i-\bar{z}_idz_i)+df
\]
for some smooth function $f$ on $\C^n$. 

By using the polar coordinate $z_i=r_ie^{2\pi \sqrt{-1}\theta_i}$ we can write $\mu_{\C}$ and $A$ in the following forms 
\[
\mu_{\C^n}=(r_1^2,\ldots ,r_n^2),\ A=\sum_ir_i^2d\theta_i+df. 
\]
In particular, we see that the tangent space along a nonsingular fiber of $\mu_{\C^n}$ is spanned by $\partial_{\theta_i}$'s. Thus a direct computation shows that a point $b\in \R^n_+$ is Bohr-Sommerfeld if and only if $b\in\R^n_+\cap \Z^n$. \
\end{proof}

By the above proposition and the definition of a locally toric Lagrangian fibration we can obtain the following corollary. 
\begin{corollary}
For a locally toric Lagrangian fibration Bohr-Sommerfeld fibers appear discretely. 
\end{corollary}
\begin{example}
For a nonsingular projective toric variety it is well-known that Bohr-Sommerfeld fibers correspond one-to-one to the integral points in the moment polytope. For example see \cite{Danilov}. 
\end{example}
\begin{example}\label{nontoricex2}
We consider the locally toric Lagrangian fibration $\mu\colon (M,\omega)\to B$ in Example~\ref{nontoricex}. We show that $(M,\omega)$ is prequantizable. 

Let $(H_c,\nabla^{H_c})$ be the $c$ times tensor power of the hyperplane bundle on $\C P^1$. With the identification of $(\C P^1,c\omega_{FS})$ and the symplectic quotient $\left(\Phi^{-1}(0),\omega_{\C^2}|_{\Phi^{-1}(0)}\right)/S^1$ in Example~\ref{nontoricex} $(H_c,\nabla^{H_c})$ can be written in the following explicit way
\[
(H_c,\nabla^{H_c})=\left( \Phi^{-1}(0)\times \C , d+1/2\sum_i(z_id\overline{z}_i-\overline{z}_idz_i)\right) /S^1,
\]
where the $S^1$-action is defined by
\begin{equation*}
t\cdot (z_0,z_1,w):=(tz_0,tz_1,t^cw). 
\end{equation*}

Now let us define the prequantizing line bundle $(\tilde{L},\tilde{\nabla})$ on $(\tilde{M},\tilde{\omega})$ by
\[
(\tilde{L},\tilde{\nabla}):=\left( \pr_1^*(\R\times S^1\times \C ,d-2\pi \sqrt{-1}rd\theta )\otimes_\C\pr_2^*(H_c,\nabla^{H_c})\right) .
\]
We also define the lift of the $\Z$-action~\eqref{ZonM} on $\tilde{M}$ to $\tilde{L}$ by
\begin{equation}\label{ZonL}
n(r,u,[z_0:z_1,w]):=\left( r+n(-a\abs{z_1}^2+b),u,[z_0:u^{na}z_1,u^{nb}w]\right) .
\end{equation}
It is easy to see that \eqref{ZonL} preserves $\tilde{\nabla}$ and the standard Hermitian metric. We put
\[
(L,\nabla):=(\tilde{L},\tilde{\nabla})/\Z . 
\]
Then $(L,\nabla)$ is a prequantizing line bundle on $(M, \omega)$. 

Next we see the Bohr-Sommerfeld fibers of $\mu$ with respect to $(L,\nabla)$. The direct computation shows that Bohr-Sommerfeld fibers of $\tilde{\mu}$ correspond one-to-one to the elements in $\tilde{B}\cap \Z^2$. Let $F$ be a fundamental domain of the $\Z$-action~\eqref{ZonB} on $\tilde{B}$. $F$ is written as
\[
F:=\{ (r_1,r_2)\in \tilde{B} \mid 0\le r_2\le c,\ -1/2\le r_1<-ar_2+b-1/2\}.
\]
Then, Bohr-Sommerfeld fibers of $\mu$ correspond one-to-one to the elements in $F\cap \Z^2$. See Figure~\ref{fund-domain}. 
\begin{figure}[htb]
\centering
\input{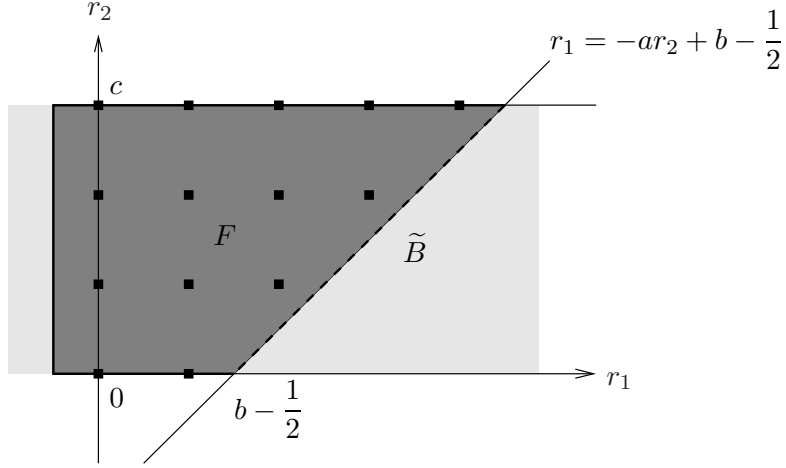}
\caption{Bohr-Sommerfeld points in Example~\ref{nontoricex}}
\label{fund-domain}
\end{figure}
\end{example}

In the rest of this section we assume that $M$ is closed. Let $(g,J)$ be a Hermitian structure on $M$ compatible with $\omega$ as in Corollary~\ref{invariantHermitianstructure}. We define the Hermitian vector bundle $W$ on $M$ by 
\begin{equation}\label{W}
W:=\wedge_\C^\bullet TM_\C\otimes L. 
\end{equation}
$W$ is a $\Z_2$-graded Clifford module bundle with respect to the Clifford module structure \eqref{Clifford}. Let $D$ be the Dirac-type operator on $W$. We define the {\it Riemann-Roch number} to be the  index of $D$. 

The purpose of this section is to show the following theorem. 
\begin{theorem}\label{RR=BSforlocallytoricLagrangianfibrations}
Let $\mu \colon (M, \omega )\to B$ be a four-dimensional prequantizable locally toric Lagrangian fibration with prequantizing line bundle $(L,\nabla)$. Then the Riemann-Roch number is equal to the number of both nonsingular and singular Bohr-Sommerfeld fibers. 
\end{theorem}
\begin{proof}
Let $B_{BS}$ be the set of Bohr-Sommerfeld points of $\mu$ in $B$. We put $V:=\mu^{-1}(B\setminus B_{BS})$. In order to prove Theorem~\ref{RR=BSforlocallytoricLagrangianfibrations} we define a good compatible fibration on $V$ as follows. 

On the regular non Bohr-Sommerfeld points $U_0:={\mathcal S}^{(2)}B\setminus B_{BS}$ of $\mu$ we define the fibration by 
\[
\pi_0:=\mu|_{V_0}\colon V_0:=\mu^{-1}(U_0)\to U_0. 
\]

Since $B$ is compact, there are only finitely many Bohr-Sommerfeld points in ${\mathcal S}^{(1)}B$. Suppose we have exactly $k$ Bohr-Sommerfeld points $p_1, \ldots ,p_k$ in ${\mathcal S}^{(1)}B$, namely, 
\[
\{p_1,\ldots ,p_k\}=B_{BS}\cap {\mathcal S}^{(1)}B. 
\]
For each $i$ we take a contractible open neighborhood $W_i$ of $p_i$ in $B$ which satisfies the following properties. 
\begin{enumerate}
\item\label{Wi1} For each $i$ $W_i$ contains no Bohr-Sommerfeld points except for $p_i$.
\item\label{Wi2} For each $i$ $W_i\cap {\mathcal S}^{(1)}B$ is connected.
\item\label{Wi4} For each $i$ $W_i$ does not intersect ${\mathcal S}^{(0)}B$,  namely, $W_i\cap {\mathcal S}^{(0)}B=\emptyset$. 
\item\label{Wi3} $W_i$'s are pairwise disjoint, namely, $W_i\cap W_j=\emptyset$ for all $i\neq j$. 
\item\label{Wi5} There exist finitely many non Bohr-Sommerfeld points in ${\mathcal S}^{(1)}B$, say $q_1, \ldots ,q_l$, such that we have 
\[
\bigcup_{i=1}^kW_i\cap {\mathcal S}^{(1)}B={\mathcal S}^{(1)}B\setminus \{q_1,\ldots ,q_l\} . 
\]
\end{enumerate}
It is possible to take such neighborhoods since a connected component of $\partial B$ is compact. 
\begin{figure}[htb]
\centering
\input{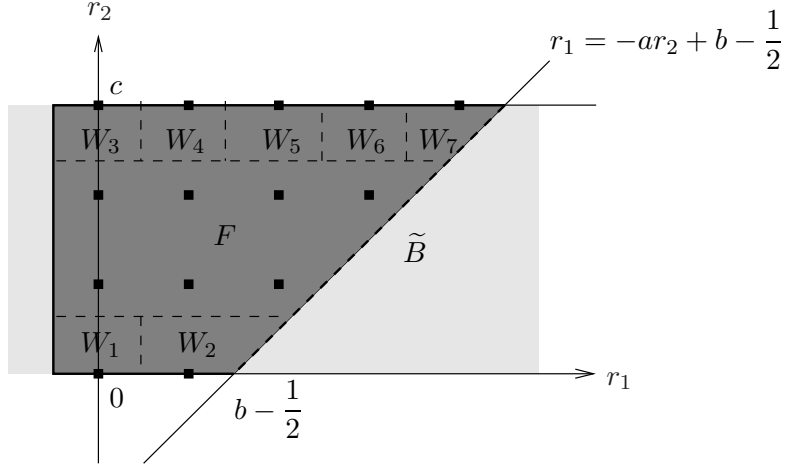}
\caption{$W_i$'s in Example~\ref{nontoricex}}
\label{W_i}
\end{figure}

We put $V'_i:=\mu^{-1}(W_i)$. Since $W_i$ is contractible, by \cite[Proposition~3.5]{Yoshida}, there exists a $T^2$-action on $V'_i$. Moreover, there exist a coordinate neighborhood $(U_{\alpha_i} ,\varphi^B_{\alpha_i})$ of $B$ containing $p_i$, a diffeomorphism $\varphi^M_{\alpha_i}\colon \mu^{-1}(U_{\alpha_i})\to \mu_{\C^2}^{-1}(\varphi^B_{\alpha_i}(U_{\alpha_i}))$ in Definition~\ref{locallytoricLagrangianfibration}, and an automorphism $\rho_{\alpha_i}\in \Aut (T^2)$ which satisfy the following properties. 
\begin{itemize}
\item $\mu_{\C^2}\circ \varphi^M_{\alpha_i}=\varphi^B_{\alpha_i}\circ \mu$. 
\item On $V'_i\cap \mu^{-1}(U_{\alpha_i})$ $\varphi^M_{\alpha_i}$ is $\rho_{\alpha_i}$-equivariant with respect to the $T^2$-action on $V'_i$ and the standard $T^2$-action on $\C^2$. 
\end{itemize}
Let $\varphi^B_{\alpha_i}(p_i)=(r_1,r_2)\in \R^2_+$. Since $p_i\in {\mathcal S}^{(1)}B$ there exists a unique coordinate $r_{j_i}$ such that $r_{j_i}= 0$. We define the circle subgroup $T_i$ of $T^2$ by
\[
T_i:=\rho_{\alpha_i}^{-1}\left(\{t=(t_1,t_2)\in T^2\mid t_{j_i}=e\}\right) . 
\]
By the properties (\ref{Wi2}) and (\ref{Wi4}) of $W_i$, $T_i$ acts on $V'_i$ freely. Then, for each $i$ we define the open set $V_i$ to be the complement of all $T_i$-orbits in $V'_i$ on each of which $(L,\nabla)$ has a non-trivial global parallel section. We also define the fibration $\pi_i\colon V_i\to U_i$ to be the natural projection 
\[
\pi_i\colon V_i\to U_i:=V_i/T_i
\]
with respect to the $T_i$-action on $V_i$. Note that for each $i$ we have the following equality
\[
V_i\cap \mu^{-1}({\mathcal S}^{(1)}B)=\mu^{-1}(W_i\cap {\mathcal S}^{(1)}B)\setminus \mu^{-1}(p_i).
\]
In order to show this equality, it is sufficient to prove the following claim: for a point $x\in \mu^{-1}(W_i\cap {\mathcal S}^{(1)}B)$ the restriction of $(L,\nabla)$ to the fiber $\mu^{-1}\left(\mu(x)\right)$ of $\mu$ is trivially flat if and only if the restriction of $(L,\nabla)$ to the $T_i$-orbit $T_ix$ through $x$ is trivially flat. This can be shown as follows. By the property (\ref{Wi2}) of $W_i$, the $T^2$-action on $\mu^{-1}(W_i\cap {\mathcal S}^{(1)}B)$ has the unique one-dimensional stabilizer $S\subset T^2$. For example, in case of $j_i=2$, $S$ is written as 
\[
S=\rho_{\alpha_i}^{-1}\left(\{t=(t_1,t_2)\in T^2\mid t_1=e\}\right) . 
\]
In any case, by the definition of $T_i$, $T^2$ is decomposed as $T^2=T_i\times S$. Then, the above claim follows from the fact that in the $T^2$-action on $\mu^{-1}\left(\mu(x)\right)$ $S$ acts trivially and $T_i$ acts freely on $\mu^{-1}\left(\mu(x)\right)$. 

By construction $\{ \pi_i\colon V_i\to U_i\mid i=0,\ldots, k\}$ is a good compatible fibration on $V$. Moreover, by Lemma~\ref{averagingoperationonM}, there is an averaging operation with respect to $\{ \pi_i\colon V_i\to U_i\mid i=0,\ldots, k\}$. 

Recall that $(g,J)$ is a Hermitian structure on $M$ compatible with $\omega$ as in Corollary~\ref{invariantHermitianstructure}. Then, as in the case of usual torus actions, $g$ defines the compatible Riemannian metric of $\{ \pi_i\colon V_i\to U_i\mid i=0,\ldots, k\}$ whose restriction to each fiber of $\mu$ is flat. 
We define the acyclic compatible system in the same way as in Section~\ref{singularBSinD}. Then by Theorem~\ref{localization formula}, the Riemann-Roch number is localized at Bohr-Sommerfeld fibers and the fibers at $q_1,\ldots ,q_l$. 

We consider their contributions. Since a fiber of $\mu$ is connected, by Theorem~\cite[Theorem~6.11]{Fujita-Furuta-Yoshida}, the contribution of a regular Bohr-Sommerfeld fiber is equal to one. 

Next we consider the contributions of singular Bohr-Sommerfeld fibers. By Definition each fiber on ${\mathcal S}^{(0)}B$ is Bohr-Sommerfeld, and its contribution is $RR_0(a_1,a_2)$ for some $a_1$ and $a_2$. By Theorem~\ref{calculation} it is equal to one. 

By the construction of the compatible fibration the local Riemann-Roch number for each singular Bohr-Sommerfeld fiber on ${\mathcal S}^{(1)}B$ is obtained from $[BS^+]$ and $[D^+]$ in \cite[Theorem~6.7]{Fujita-Furuta-Yoshida} by the product formula~\ref{product}. It is also one. 

Finally it is easy to see that the contribution of each fibers at $q_1,\ldots ,q_l$ is equal to $RR_1(a_+,a_-)$ in Section~\ref{singularnonBS} for some $a_+$ and $a_-$. Then by Theorem~\ref{calculation} it is zero. This proves Theorem~\ref{RR=BSforlocallytoricLagrangianfibrations}. \
\end{proof}
\begin{example}
Theorem~\ref{RR=BSforlocallytoricLagrangianfibrations} recovers Danilov's result~\cite{Danilov}, which says that for a nonsingular projective toric variety the Riemann-Roch number is equal to the number of the lattice points in the moment polytope, in the four-dimensional case. 
\end{example}
\begin{example}
As we described in Example~\ref{nontoricex2} the Bohr-Sommerfeld fibers correspond one-to-one to the elements in $F\cap \Z^2$. Then by Theorem~\ref{RR=BSforlocallytoricLagrangianfibrations} the Riemann-Roch number of $(M,\omega)$ is equal to the number of the elements in $F\cap \Z^2$ which is $(c+1)(2b-ac)/2$. 
\end{example}

\appendix
\section{Supplements to Section~\ref{An index theory for complete Riemannian manifolds}}\label{Appendix B}
\subsection{Partial integrations}
\begin{lemma}\label{family of cut-off functions}
Let $M$ be a complete Riemmanian manifold.
\begin{enumerate}
\item
There is a smooth proper function $f:M \to {\mathbb R}$
such that $|df|$ is bounded and $f^{-1}((-\infty,c])$ is
compact for any $c$.
\item
There is a constant $C>0$ such that
for each $\epsilon>0$ and $a \in {\mathbb R}$,
we have a compact supported
function $\rho_{a,\epsilon}:M \to [0,1]$
which
is equal to $1$ on  $f^{-1}((-\infty, a])$, and
satisfies
$
|d \rho_{a,\epsilon}|< C\epsilon.
$
\end{enumerate}
\end{lemma}
A proof of the above lemma~\ref{family of cut-off functions} is given in \cite{Green-Wu}. The existence of such a function in (1) of Lemma~\ref{family of cut-off functions} is equivalent to the completeness of  $M$.  For more details see \cite{Green-Wu}. 

By using the family of cut-off functions we can show the following two partial integration formulas.
In general let
$W$ be a Hermitian vector bundle over 
a complete Riemannian manifold $M$, and
$D_{\tau}:\Gamma(W) \to \Gamma(W)$ be
a first order partial differential operator on $W$
with smooth coefficients whose principal symbol is $\tau$.
We assume that 
$D_\tau$ has
finite propagation speed, i.e., 
$\tau$ is a smooth $L^\infty$-bounded
section of $TM \otimes \End (W)$.

\begin{lemma} \label{partial integral 1}
Let $s \in \Gamma(W)$ be an $L^2$-bounded section
such that  $D_{\tau}^*D_{\tau} s$ is also $L^2$-bounded.
Then $D_{\tau}s$ is also $L^2$-bounded and we have
$$
\int_M (D_{\tau}^*D_{\tau}s,s) =\int_M |D_{\tau}s|^2.
$$
\end{lemma}
\begin{proof}
We first assume that $s$ is smooth.
We follow Gromov's proof of \cite[Lemma 1.1 B]{Gromov}. 
From the equality
\begin{align} 
\int_M (D_{\tau}^* D_{\tau} s,\, \rho^2_{a,\epsilon} s)
&=\int_M (D_{\tau}s,\, D_{\tau}(\rho^2_{a,\epsilon} s)) \nonumber \\
&=\int_M (D_{\tau}s,\,  \rho^2_{a,\epsilon} D_{\tau}s)
  +\int_M (D_{\tau}s,\,  2
\rho_{a,\epsilon}\tau (d \rho_{a,\epsilon})s),\label{equation:partial integral}
\end{align}
there is a constant $C$ independent of $s, a, \epsilon$
such that 
$$
|| D_{\tau}^* D_{\tau} s||_2 ||s||_2
\geq ||\rho_{a,\epsilon} D_{\tau}s||_2^2
- C \epsilon||\rho_{a,\epsilon} D_{\tau}s||_2.
$$
It implies that, as $a$ increases, 
$||\rho_{a,\epsilon} D_{\tau}s||_2$ is bounded,
i.e., $D_{\tau}s$ is $L^2$-bounded.
Using (\ref{equation:partial integral}) again we have
$$
\int_M (D_{\tau}^* D_{\tau} s, s)
=||D_{\tau}s||_2^2+I, \qquad
|I|\leq C\epsilon ||D_{\tau}s||_2||s||_2.
$$
Taking $\epsilon \to 0$, we obtain the required
equality.

When $s$ is not smooth,
take a smooth compactly supported section 
which approximate $s$ in $L^2_2$-norm on 
the support of $\rho_{a,\epsilon}$.
Then we can reduce the argument to the smooth case. \

\end{proof}

\begin{lemma} \label{partial integral 2}
Suppose $s_0$ and $s_1$ are  $L^2$-bounded sections of
$W$ such that $D_{\tau} s_0$ and $D_{\tau}^* s_1$ are also
$L^2$-bounded. Then we have
$$
\int_M (D_{\tau} s_0, s_1) =\int_M (s_0, D_{\tau}^*s_1).
$$
\end{lemma}
\begin{proof}
We first assume that $s$ is smooth. We have 
\[
0 =
\int_M (D_{\tau} (\rho_{a,\epsilon}s_0), s_1) -
\int_M (s_0, D_{\tau}^*(\rho_{a,\epsilon}s_1))   
= \int_M (D_{\tau} s_0, s_1) -\int_M (s_0, D_{\tau}^*s_1) + I' 
\]
with an error term $I'$ satisfying
$|I'| \leq C \epsilon ||s_0||_2 ||s_1||_2$,
which implies the required equality.
When $s$ is not smooth,
we can reduce the argument to the smooth case
as in the proof of Lemma~\ref{partial integral 1}. \
\end{proof}
\subsection{Min-max principle}\label{min max}

In this section we use Assumption~\ref{assumption for operator}
for a single operator $D$, and Assumption~\ref{global continuity}
for a one-parameter family $\{D_t\}$.

\begin{lemma} \label{estimate with error term}
For any compact subset $K$ containing $M \setminus V$
there is a compact set $K'$ containing $K$
such that if $s$ and $Ds$ are $L^2$-bounded, then
we have the estimate
$$
\lambda_0^{1/2}||s||_{M\setminus K}
- 2 \lambda_0^{1/2} || s||_{K' \setminus K}
\leq || Ds||_{M\setminus K}.
$$
Moreover if the coefficients of $D_t$ are
$C^\infty$-convergent to those of $D_0=D$
on any compact set as $t \to 0$, then we can choose $K'$ so that
the above estimate is valid for any $t$ sufficiently close to $0$. 
\end{lemma}
\begin{proof}
The essentially self-adjointness of $D$ implies that
we can assume that $s$ is smooth and compactly supported 
without loss of generality.
From Lemma~\ref{family of cut-off functions}, 
for any $\epsilon>0$,
there is a compact set $K'$ containing $K$ and 
a smooth non-negative function ${\rho}:M \to {\mathbb R}$ such that
$\rho=1$ on $M \setminus K'$, 
$\rho=0$ on $K$ and $|d\rho|\leq \epsilon$. 
Then the above estimate
follows from the next two inequalities
\begin{eqnarray*}
||D(\rho s)||_{M}
&\geq& \lambda_0^{1/2}||\rho s||_{V} 
\geq \lambda_0^{1/2}||s||_{M\setminus K'}
\geq \lambda_0^{1/2}||s||_{M\setminus K}-
\lambda_0^{1/2}||s||_{K' \setminus K},
\\
||D(\rho s)||_{M} &\leq&
||\rho Ds||_{M} +||\sigma \cdot ((d \rho)\otimes s) ||_{M}
\leq ||Ds||_{M\setminus K}
+C(D,K') \epsilon  ||s||_{K' \setminus K},
\end{eqnarray*}
where $C(D,K'):=\max_{K'}\abs{\sigma}$. The last statement of the lemma follows from the fact that $C(D_t,K')$ is continuous with respect to $t$. \ 
\end{proof}

\begin{proposition} \label{key proposition}
Suppose $0 \leq \lambda_1 <\lambda_0$.
Let $\{s_i\}$ be a sequence of $L^2$-sections of $W$
satisfying $||s_i||_{M}=1$,
and $\{t_i\}$ is a sequence 
convergent to $0$.
Suppose each $D_{t_i} s_i$ is $L^2$-bounded and satisfies
$||D_{t_i} s_i||_{M}^2 \leq \lambda_1$.
Then there is a
subsequence $\{ s_{i'}\}$ which is weakly convergent to
some non-zero $s_\infty\neq 0$ such that
$D_0 s_\infty$ is
$L^2$-bounded and satisfies
\begin{equation}\label{limit estimate}
||D_{0} s_\infty||_{M}^2 \leq \lambda_1 ||s_\infty||^2_M. 
\end{equation}
\end{proposition}

\begin{proof}
Take a subsequence $\{s_{i'}\}$ so that
$\{s_{i'}\}$ and $\{D_{t_{i'}} s_{i'}\}$ are
weakly convergent to some $s_\infty$ and
$u_\infty$ in $L^2(M,W)$ respectively. Since $D_t$ is a smooth family,
for each smooth compactly-supported section $\phi$
the sequence
$D_{t_{i'}}\phi$ is strongly convergent to $D_0 \phi$.
The equality
$
\int_M (D_{t_{i'}}\phi, s_{i'})=
\int_M (\phi, D_{t_{i'}}s_{i'})
$
implies 
$
\int_M (D_{0}\phi, s_{\infty})=
\int_M (\phi, u_\infty),
$
i.e., $ D_0 s_\infty=u_\infty $ weakly.

Since $\{D_{t_{i'}} s_{i'}\} $ is $L^2$-bounded, 
Assumption~\ref{global continuity} and 
a priori estimate imply that on any compact set
$s_{i'}$ is strongly $L^2$-convergent to $s_\infty$. 

On the other hand for any compact set $K$ 
containing $M\setminus V$ there exists 
a compact set $K'$ such that  
$$
\lambda_0^{1/2}||s_{i'}||_{M\setminus K}
- 2 \lambda_0^{1/2} || s_{i'}||_{K' \setminus K}
\leq || D_{t_{i'}}s_{i'}||_{M\setminus K}\le \lambda_1^{1/2}
$$by Lemma~\ref{estimate with error term}.  
If $s_{\infty}$ is $0$, 
then we have $||s_{i'}||_{K'\setminus K}$ converges to $0$, 
which contradicts to $||s_{i'}||_M=1$ and  $\lambda_1<\lambda_0$.

Suppose the estimate (\ref{limit estimate}) does not hold.
Then for any $\epsilon>0$ and 
any sufficiently small $\epsilon'>0$
there exists a compact set $K$ containing $M \setminus V$
satisfying $||s_\infty||_{M\setminus K}< \epsilon$ and
$
\lambda_1||s_\infty||^2_K +\epsilon' <||D_{0} s_\infty||_{K}^2.
$
We choose $\epsilon$ and $\epsilon'$ so that they satisfy
$
8\epsilon\lambda_0(1+2\epsilon)<\epsilon'/2.
$
Note that the weak convergence implies $||D_0 s_\infty||_K^2
\leq \liminf_{i'\to \infty} ||D_{t_{i'}} s_{i'}||_K^2$.
Since 
$s_{i'}$ is strongly $L^2$-convergent to $s_\infty$
on the compact set $K$, we have
\begin{equation}\label{one estimate}
\lambda_1||s_{i'}||^2_K +\frac{\epsilon'}{2} <||D_{i'} s_{i'}||_{K}^2
\end{equation}
for sufficiently large $i'$.
Let $K'$ be the compact set containing
$K$ which gives the estimate in Lemma~\ref{estimate with error term}
for sufficiently small $t$.
Since ,
$s_{i'}$ is strongly $L^2$-convergent to $s_\infty$
on the pre-compact set $K' \setminus K$, we have
$||s_{i'}||_{K' \setminus K}<2\epsilon$
for sufficiently large $i'$.
The estimate in Lemma~\ref{estimate with error term} implies that
we have
$
\lambda_0^{1/2} ||s_{i'}||_{M\setminus K}
\leq || Ds_{i'}||_{M\setminus K} + 4\epsilon \lambda_0^{1/2}
$
for sufficiently large $i'$.
Taking square, and
using $\lambda_1<\lambda_0$ and
$|| D_{i'}s_{i'}||_{M\setminus K} \leq \lambda_1^{1/2}$,
we obtain
$$
\lambda_1||s_{i'}||^2_{M\setminus K}
\leq || D_{i'}s_{i'}||^2_{M\setminus K} 
+8\epsilon\lambda_0(1+2\epsilon)
$$
Adding with (\ref{one estimate}) we have
$
\lambda_1||s_{i'}||^2_{M}
< || D_{i'}s_{i'}||^2_{M} 
$
for sufficiently large $i'$, which contradicts our assumption. \
\end{proof}
For a single operator $D$ we have
\begin{proposition}\label{single operator}

\begin{enumerate}
\item Suppose $\lambda<\lambda' <\lambda_0$.
\begin{enumerate}
\item
If $s \in E_\lambda$, then
$Ds$ is $L^2$-bounded and  
$||D s||_M^2= \lambda ||s||^2_M$.
\item
$E_\lambda$ and $E_{\lambda'}$ are $L^2$-orthogonal to each other.
\end{enumerate}
\item Suppose $0 \leq \lambda_1 <\lambda_0$.
\begin{enumerate}
\item
$\dim \,\oplus_{\lambda\leq \lambda_1} E_\lambda(D) <\infty.$
\item
Let $R_{\lambda_1}$ be the set of  $L^2$-bounded sections $s$ 
satisfying $||s||_M=1$ and $||D s ||_M^2< \lambda_0$
such that $s$ is $L^2$-orthogonal to
$\oplus_{\lambda\leq \lambda_1} E_\lambda(D)$.
If $R_{\lambda_1}$ is not empty, then
the functional $I_{\lambda_1}:R_{\lambda_1} \to [0, \lambda_0)\,$, 
$I_{\lambda_1}(s)=||D s||_M^2$ attains its minimum value.
\item
Let $\lambda_2=||D s_0||_M^2$ be the minimum value of 
$I_{\lambda_1}$ at
a minimum $s_0$. Then we have $\lambda_1<\lambda_2 <\lambda_0$ and
$
s_0 \in E_{\lambda_2}
$
\end{enumerate}
\end{enumerate}
\end{proposition}
\begin{proof}
The first  statement for $\lambda<\lambda'<\lambda_0$ follows from 
the partial integration formulas
Lemma~\ref{partial integral 1} and Lemma~\ref{partial integral 2}.

Suppose $\lambda_1 <\lambda_0$.
If
$\oplus_{\lambda\leq \lambda_1} E_\lambda(D)$
is not finite dimensional, then
we have a sequence $e_i$ in the infinite space
with $||e_i||_M=1$ and mutually $L^2$-orthogonal each other.
Proposition~\ref{key proposition} implies that
we have a weakly convergent limit for a subsequence
with non zero limit, which is a contradiction.

Suppose $s_i$ is a sequence in $R_{\lambda_1}$ 
such that $I_{\lambda_1}(s_i)$
convergent to the infimum of $I_{\lambda_1}$.
Proposition~\ref{key proposition} implies that
we have a weakly convergent limit $s_\infty\neq 0$ for a subsequence
such that $s_0:=s_\infty/||s_\infty||$ is an element of
$R_{\lambda}$ which attains the infimum.
For any compactly-supported smooth section $s'$,
let $s''$ be the $L^2$-orthogonal projection of $s'$
to 
$\oplus_{\lambda\leq \lambda_1} E_\lambda(D)$
and put $s'''=s'-s''$. 
Since $s_0$ attains the minimum of 
$I_{\lambda_1}$, the derivative of 
$(s_0+ t s''')/||s_0+ts'''||_M$ at $0$ vanishes, 
and  we obtain
$$
\int_M (D s_0, D s''')= \lambda_2 \int_M(s_0, s''').
$$
Since $D s'''=Ds'-Ds''$ and $D^2 s'''=D^2 s' -D^2 s''$ is $L^2$-bounded,
Lemma~\ref{partial integral 2}
implies
$$
\int_M (s_0, D^2 s''')= \lambda_2 \int_M(s_0, s''').
$$
On the other hand we have
$$
\int_M (s_0, D^2 s'')=\int_M(s_0, s'')=0.
$$
These relations imply 
$$
\int_M (s_0, D^2 s')= \lambda_2 \int_M(s_0, s'),
$$
i.e., $D^2 s_0=\lambda_2 s_0$ weakly.
Lemma~\ref{partial integral 1} implies that 
$D s_0$ is $L^2$-bounded and 
$||D s_0||_M^2= \lambda_2 ||s_0||_M=\lambda_2$.
Since $s_0$ is $L^2$-orthogonal to
$\oplus_{\lambda \leq \lambda_1} E_\lambda$,
we have $\lambda_1 < \lambda_2$. The regularity theorem implies $s_0$ is smooth and hence $s_0\in E_{\lambda_2}$. \
\end{proof}

\begin{corollary} \label{single operator 2}
Suppose  $\lambda_1 <\lambda_0$.
Let $E$ be a ${\mathbb Z}/2$ graded subspace of $L^2(M,W)$ such that
such that
$D s$ is in $L^2(M,W)$ and
$
||D s||_{M}^2 \leq 
\lambda_1 ||s||_{M}^2
$
for any $s \in E$.
Then $E$ is finite dimensional and
$$
\dim \oplus_{\lambda\leq \lambda_1} E_\lambda(D) 
\geq \dim E.
$$
Moreover the above inequality holds for
each degree of ${\mathbb Z}/2$.
\end{corollary}

\subsection{Deformation invariance of index}   
\label{proof of index theory}

For a family $\{D_t\}$ we have: 
\begin{proposition}\label{min-max}
Suppose $\lambda_1< \lambda_0$.
Let $\{t_i\}$
be a sequence convergent to $0$
and $E^{(i)}$ 
be a ${\mathbb Z}/2$ graded subspace of $L^2(M,W)$ such that
$D_{t_i} s$ is in $L^2(M,W)$ and
$
||D_{t_i} s||_{M}^2 \leq 
\lambda_1 ||s||_{M}^2
$
for any $i$ and $s \in E_i$.
Then each $E^{(i)}$ is finite dimensional and
$$
\dim \oplus_{\lambda\leq \lambda_1} E_\lambda(D_0) 
\geq \limsup_{i\to \infty}\dim E^{(i)}.
$$
Moreover the above inequality holds for
each degree of ${\mathbb Z}/2$.
\end{proposition}

\begin{proof}
Suppose $\dim \oplus_{\lambda\leq \lambda_1} E_\lambda(D_0) 
< \dim E^{(i')}$ for a subsequence $\{i'\}$.
Let $s_{i'}$ be an element of $E^{(i')}$ with
$||s_{i'}||_M=1$ which is $L^2$-orthogonal to
$\oplus_{\lambda\leq \lambda_1} E_\lambda(D_0)$.
Let $s_\infty$ be the $L^2$-bounded section
given by Proposition~\ref{key proposition}.
Then the weak limit $s_\infty$ is also $L^2$ orthogonal to
$\oplus_{\lambda\leq \lambda_1} E_\lambda(D_0)$,
which contradicts Proposition~\ref{single operator}. \
\end{proof}

\begin{remark}
\label{remark for min-max}
{\rm 
In the above proof the choice of $s_{i'}$ can be 
generalized as follows:
Fix an $L^2$-orthonormal basis $e_1,e_2,\ldots,e_N$  
of the finite dimensional space
$
\oplus_{\lambda \leq\lambda_1} E_\lambda(D_0).
$
Fix any sequence $\{e_k^{(i')}\}_{i'}$ for each $1\leq k\leq N$
which is strongly $L^2$-convergent to $e_k$ as $i' \to \infty$.
Let $s_{i'}$ be
an element of $E^{(i')}$ with
$||s_{i'}||_M=1$ which is $L^2$-orthogonal to
all $e_k^{(i')}$ $(1 \leq k \leq N)$.
Then the rest of the proof remains valid.
}
\end{remark}

\begin{corollary} \label{limsup}
For each degree of ${\mathbb Z}/2$ we have the inequality
$$
\dim \oplus_{\lambda\leq \lambda_1} E_\lambda(D_0) 
\geq
\limsup_{t \to 0} \dim \oplus_{\lambda\leq \lambda_1} E_\lambda(D_t) 
$$
\end{corollary}

\begin{proposition}\label{liminf}
Suppose $\lambda_1< \lambda_0$.
For each degree of ${\mathbb Z}/2$ we have the inequality
$$
\dim \oplus_{\lambda <\lambda_1} E_\lambda(D_0) 
\leq \liminf_{t\to 0}
\dim \oplus_{\lambda <\lambda_1} E_\lambda(D_t) 
$$
\end{proposition}
\begin{proof}
Let $\epsilon_0>0$ be a sufficiently small number, which
we fix later.
Let $e_1,e_2,\ldots,e_N$ be an $L^2$-orthonormal basis
of the finite dimensional space
$
E:=\oplus_{\lambda <\lambda_1} E_\lambda(D_0)
$
consisting of eigenvectors of $D_0$
with eigenvalues $\mu_1,\mu_2,\ldots,\mu_N$ respectively.
For a sufficiently large $a$ and sufficiently
small $\epsilon>0$, the truncated sections
$e_i'=\rho_{a,\epsilon}e_i$ satisfy
$$
\left|\delta_{ij}- \int_M(e_i', e_j')  -\right|<\epsilon_0,
\quad
\left|\mu_i\mu_j \delta_{ij}- \int_M(D_0 e_i', D_0 e_j')  \right|
 < \epsilon_0
$$
for every $1 \leq i,j \leq N$ as in the proof of
Lemma~\ref{partial integral 1} or Lemma~\ref{partial integral 2}.
Here $\delta_{ij}$ is Kronecker's delta.
Let $E'$ be the vector space spanned by $e_i'$.
Since the support of all the $e_i'$ are contained
in the compact support of $\rho_{a,\epsilon}$,
Assumption~\ref{global continuity} implies that
$$
\left|\mu_i\mu_j \delta_{ij}- \int_M(D_t e_i', D_t e_j')   \right|
 < \epsilon_0.
$$
for every $t$ sufficiently closed to $0$.
Let $E'$ be the vector space spanned by $e_i'$.
Then if $\epsilon_0$ is sufficiently small, each element
$s'$ of $E'$ satisfies
$||D_t s'||^2_M \leq  \lambda_1 ||s'||_M^2$.
Corollary~\ref{single operator 2} implies
$
\dim E'\leq  \dim \oplus_{\lambda <\lambda_1} E_\lambda(D_t)
$.
It is easy to check that the above inequality holds for
each degree of ${\mathbb Z}/2$. \
\end{proof}

\begin{proof}[Proof of Theorem~\ref{deformation invariance}]
From Proposition~\ref{single operator}
there is $0< \lambda_1 <\lambda_0$ such that
$E_{\lambda_1}=0$.
Then Corollary~\ref{limsup} and Proposition~\ref{liminf}
imply
$$
\dim \oplus_{\lambda\leq \lambda_1} E_\lambda(D_0) 
=
\lim_{t \to 0} \dim \oplus_{\lambda\leq \lambda_1} E_\lambda(D_t)
$$
and the equality holds for each degree,
from which the claim follows. \
\end{proof}

\subsection{Generalized deformation invariance and gluing formula}\label{proof of Theorem 3.9}
Let $(M, W, \sigma, D,V)$ be 
as in Section~\ref{formulation of index}. 
We also assume that there exists a family of data 
$\{M_i, W_i, V_i, K_i, \iota_i, \tilde\iota_i, \tilde D_i\}_{i\in\N}$ 
together satisfying Assumption~\ref{assumption gluing}.  
\begin{proof}[Proof of Theorem~\ref{generalized deformation invariance}]
The most of the arguments in Sections~\ref{min max} and
\ref{proof of index theory} go through.
The statement and proof of 
Proposition~\ref{key proposition}
is straightforwardly generalized
with replacement of $D_{t_i}$ by $D_i$.
The statement and proof of 
Proposition~\ref{liminf} 
is also straightforwardly generalized.
To generalize Proposition~\ref{min-max}
we need the construction in Remark~\ref{remark for min-max}.
As for the statement we replace $D_{t_i}$ with $D_i$, and
let $E^{(i)}$ 
be a ${\mathbb Z}/2$ graded subspace of $L^2(M_i,W_i)$.
As for the proof take $e_k^{(i')}$ with support contained
in $int(K_i)$. Let $s_{i'}$ be an element of $E^{(i')}$ with
$||s_{i'}||_{M_i}=1$  which is $L^2$-orthogonal
to all $\tilde\iota_i e_k^{(i')} $ $(1\leq k \leq N)$.
Then the rest of the proof remain valid.
Then the argument of Section~\ref{proof of index theory} can be straightforwardly
generalized to show Theorem~\ref{generalized deformation invariance}. \
\end{proof}

\section{Proof of Lemma~\ref{good open covering} }\label{proof of good open covering}

\begin{proof}[Proof of Lemma~\ref{good open covering}]
Let $H_1,H_2,\ldots,H_m$ be the elements of $A$.
Without loss of generality we assume that
$H_i \supset H_j$ implies  $i \leq j$.
We construct
a family of open sets $V_i^{(j)}$ $(1\leq i \leq j \leq m)$
by induction on $1 \leq i\leq m$.
For the construction with $i=i_0$ we assume
the following properties. 
\begin{enumerate}
\item[(B1)]\label{B1}
$V_i^{(j)}$ contains the closure of $V_i^{(j+1)}$ for
all $1 \leq i<i_0$ and $i \leq j < m$.
\item[(B2)]\label{B2}
If $x \in V_i^{(i)}$ for $1 \leq i<i_0$, 
then we have $G_x \subset H_i$. 
\item[(B3)] \label{B3}
For $x\in M$ with $G_x=H_{i}$ for some $1 \leq i < i_0$,
we have 
$$
x \in \bigcup_{\{j\,|\, H_{j} \supset H_{i}\}}  V_{j}^{(m)}.
$$ 
\item[(B4)]
 If the intersection $V_i^{(j)} \cap V_{j}^{(j)}$ is not empty
for $1 \leq i<j < i_0$, then
we have 
$H_i \supset H_{j}$.
\end{enumerate}
If $i_0=1$, then the above is the empty assumption.
For $1 \leq i_0 \leq m$,
using the above properties as the assumption of induction, 
we construct $V_{i_0}^{(j)}$ $(i_0 \leq j \leq m)$
which satisfy 
the above properties with replacement
of $i_0$ by $i_0+1$.

Suppose $1\leq i_0 <m$ and assume (B1),(B2),(B3) and (B4).
Then (B3) implies that the closed set
$$
K_{i_0}:=
M^{H_{i_0}}\setminus \bigcup_{\{ k \,|\, H_k 
\supsetneqq H_{i_0}\}} V_k^{(m)}
$$
is contained in $\{x \in M\,|\, G_x = H_{i_0}\}$, where $M^{H_{i_0}}$
is the fixed point set
$M^{H_{i_0}}=\{x \in M\,|\, G_x \supset H_{i_0}\}.$
Hence (C2) implies that $K_{i_0}$ does not intersect
with the open set
$$
\bigcup_{\{j<i_0 \,|\, H_{i_0} \not\subset H_{j} \}}  V_j^{(i_0-1)}.
$$
Let $L_{i_0}$ be the closure of
$$
\bigcup_{\{j<i_0 \,|\, H_{i_0} \not\subset H_{j} \}}  V_j^{(i_0)}.
$$
Then (B1) implies $K_{i_0}\cap L_{i_0}=\emptyset$. 
Since $K_{i_0}$ is a subset of $\{x \in M\,|\, G_x=H_{i_0}\}$,
there is an open neighborhood $V$
of the closed set $K_{i_0}$ in the complement of $L_{i_0}$ 
such that for each $x\in V$ we have $G_x \subset H_{i_0}$.
Now we take a decreasing sequence of open neighborhoods 
$V_{i_0}^{(j)}$ $(i_0\leq j \leq m )$ of $K_{i_0}$
so that $V_{i_0}^{(i_0)}=V$, $V_{i_0}^{(m)} \supset K_{i_0}$
and 
$V_{i_0}^{(j)}$ contains the closure of
$V_{i_0}^{(j+1)}$ for $i_0 \leq j< m$.
We can choose the decreasing sequence so that
the open sets $V_{i_0}^{(j)}$ $(i_0\leq j \leq m )$ 
are $G$-invariant because
the quotient space $M/G$ is a regular space.

Then it is straightforward to check (B1),(B2),(B3) and (B4)
are satisfied with $i_0$ replaced by $i_0+1$.

The family of open sets $\{V_{H_i}:=V_i^{(m)}\}_{1 \leq i \leq m}$
is an open covering of $M$ and satisfies the required properties.

Now if a family of open sets $\{V_x\}_{x\in M}$ satisfying $gV_x=V_{gx}$ 
for all $g\in G$ is given, then since  $K_{i_0}$ is contained in 
$\displaystyle\bigcup_{\{x \ | \ G_x=H_{i_0}\}}V_x$ 
we can take a $G$-invariant open neighborhood $V=V_{i_0}^{(i_0)}$ 
of $K_{i_0}$ so that $V\subset \bigcup_{G_x=H_{i_0}}V_x$, 
and hence, we have the required open covering $\{V_H\}_{H\in A}$. \
\end{proof}

\noindent{\bf Acknowledgements.}
The authors would like to thank
Yukio Kametani, Shinichiroh Matsuo, Nobuhiro Nakamura and Hirohumi Sasahira.
Throughout the seminar with them
the authors could simplify the proof of the key proposition (Proposition~\ref{KeyProp-local}). 
The second author is grateful to the hospitality of
MIT and University of Minnesota. 

The first author is partly supported by Grant-in-Aid for Young Scientists (Start-up) 21840045 and Grant-in-Aid for Young Scientists (B) 23740059. The second author is partly supported by Grant-in-Aid for Scientific Research (A) 19204003 and Grant-in-Aid for Scientific Research (B) 19340015. The third author is partly supported by Grant-in-Aid for Young Scientists (B) 20740029, Grant-in-Aid for Young Scientists (B) 22740046, and Fujyukai Foundation.

\bibliographystyle{amsplain}
\bibliography{reference}
\end{document}